\documentclass[12pt,a4paper,reqno]{amsart}

\usepackage[lmargin=1.2in, rmargin= 1.2in, tmargin=1.5in]{geometry}

\title{Enriched quasi-categories and the templicial homotopy coherent nerve}

\author{Wendy Lowen} 
\address[Wendy Lowen]{Universiteit Antwerpen, Departement Wiskunde, Middelheimcampus,
Middelheimlaan 1,
2020 Antwerp, Belgium}
\email{wendy.lowen@uantwerpen.be}

\author{Arne Mertens}
\address[Arne Mertens]{Universiteit Antwerpen, Departement Wiskunde, Middelheimcampus,
Middelheimlaan 1,
2020 Antwerp, Belgium}
\email{arne.mertens@uantwerpen.be}

\thanks{
This project has received funding from the European Research Council (ERC) under the European Union’s Horizon 2020 research and innovation programme (grant agreement No. 817762).\\
The second named author was a predoctoral fellow of the Research Foundation -  Flanders (FWO), file number 1137921N. The current paper is based on part of the resulting PhD thesis \cite{mertens2022templicial}.
}

\makeatletter
\@namedef{subjclassname@2020}{
  \textup{2020} Mathematics Subject Classification}
\makeatother



\usepackage{graphicx}
\usepackage[english]{babel}
\usepackage{amsmath}
\usepackage{amssymb}
\usepackage{amsthm}
\usepackage[all,cmtip]{xy}
\usepackage{cancel}
\usepackage[alpine]{ifsym}
\usepackage{enumerate}
\usepackage{stmaryrd}
\usepackage{tikz}
\usepackage{tikz-cd}
\usetikzlibrary{matrix}
\usepackage[toc,page]{appendix}
\usepackage{hyperref}
\hypersetup{breaklinks=true}
\usepackage[anythingbreaks]{breakurl}

\usepackage[style=alphabetic, maxnames = 50, maxalphanames=50, urldate=long]{biblatex}
\DeclareMathOperator{\Ob}{Ob}

\DeclareMathOperator{\Hom}{Hom}
\DeclareMathOperator{\id}{id}

\DeclareMathOperator{\Lan}{Lan}

\DeclareMathOperator*{\colim}{colim}

\DeclareMathOperator{\Set}{Set}
\DeclareMathOperator{\Mon}{Mon}

\DeclareMathOperator{\Ab}{Ab}
\DeclareMathOperator{\Mod}{Mod}

\DeclareMathOperator{\SSet}{SSet}

\DeclareMathOperator{\Colax}{Colax}

\DeclareMathOperator{\Quiv}{Quiv}
\DeclareMathOperator{\Cat}{Cat}

\newcommand{\fint}{\mathbf{\Delta}_{f}} 
\newcommand{\simp}{\mathbf{\Delta}} 
\newcommand{\asimp}{\mathbf{\Delta}_{+}} 
\newcommand{\nec}{\mathcal{N}ec} 
\newcommand{\fnec}{\mathcal{N}ec^\text{\VarFlag}\!\!} 
\newcommand{\ffnec}{\mathcal{N}ec^\text{\VarFlag}_{f}\!\!} 
\newcommand{\ts}{S_{\otimes}} 

\newtheorem{Thm}{Theorem}[section]
\newtheorem*{Thm*}{Theorem}
\newtheorem{Lem}[Thm]{Lemma}
\newtheorem{Prop}[Thm]{Proposition}
\newtheorem{Cor}[Thm]{Corollary}

\theoremstyle{definition}
\newtheorem{Def}[Thm]{Definition}
\newtheorem{Ex}[Thm]{Example}

\newtheorem{Con}[Thm]{Construction}
\newtheorem{Not}[Thm]{Notation}

\theoremstyle{remark}
\newtheorem{Rem}[Thm]{Remark}

\begin{document}

\begin{abstract}
We lay the foundations for a theory of quasi-categories in a monoidal category $\mathcal{V}$ replacing $\Set$, aimed at realising weak enrichment in the category $S\mathcal{V}$ of simplicial objects in $\mathcal{V}$. To accomodate non-cartesian monoidal products, we make use of an ambient category $\ts\mathcal{V}$ of templicial - or “tensor-simplicial” - objects in $\mathcal{V}$, which are certain colax monoidal functors following Leinster. Inspired by the description of the categorification functor due to Dugger and Spivak, we construct a templicial analogue of the homotopy coherent nerve functor which goes from $S\mathcal{V}$-enriched categories to $\ts\mathcal{V}$.
We show that an $S\mathcal{V}$-enriched category whose underlying simplicial category is locally Kan, is turned into a quasi-category in $\mathcal{V}$ by this nerve functor.
\end{abstract}

\maketitle

\tableofcontents

\section{Introduction}\label{section: Introduction}

\subsection{The main goal}\label{parin}

The theory of $(\infty,1)$-categories (or simply $\infty$-categories) is by now well established, with notable models including simplicial categories \cite{bergner2007model}, Segal categories \cite{hirschowitz1998descente}, complete Segal spaces \cite{rezk2001model} and quasi-categories \cite{joyal2004notes}. These models were all shown to be homotopically equivalent by the work of \cite{bergner2007three}\cite{joyal2007quasisegal}\cite{joyal2007quasisimplicial}\cite{lurie2009higher}.

One may view $\infty$-categories as being ``weakly enriched in spaces'', that is between objects we have morphism spaces (usually formalised as simplicial sets) along with compositions that are only well-defined and associative up to coherent homotopy.
In analogy with ordinary enriched categories, one may thus conceive (weakly) enriched $\infty$-categories by replacing $\SSet$ by a suitable category $\mathcal{M}$ possessing some weak or higher structure. This is particularly easy starting from the “strict model” of simplicial categories, which one may replace by categories strictly enriched in a suitable monoidal model category $\mathcal{M}$, see \cite{berger2013homotopy}\cite{stanculescu2014constructing}\cite{muro2015Dwyer}. 

Further, enriched counterparts of some of the other classical models have been proposed in the literature.
In the case where $\mathcal{M}$ is cartesian (i.e. its monoidal structure is given by the cartesian product), Simpson introduced $\mathcal{M}$-enriched Segal categories \cite{simpson2012homotopy} as certain simplicial objects in $\mathcal{M}$. Building on this work and generalising Leinster’s homotopy monoids \cite{leinster2000homotopy}, Bacard defined $\mathcal{M}$-enriched Segal categories over a general monoidal model category $\mathcal{M}$ \cite{bacard2010Segal} in order to encompass enriching categories of interest like chain complexes over a commutative ring. Recently, an approach to complete dg-Segal spaces of a quite different flavour was put forward by Dimitriadis Bermejo \cite{bermejo2022dgSegal}, replacing the simplex category by a certain category of free dg-categories of finite type.

A general approach to enrichment which is formulated entirely within the framework of $\infty$-categories is due to Gepner and Haugseng. In \cite{gepner2015enriched} they developed a theory of $\infty$-categories weakly enriched in a monoidal $\infty$-category, based on Lurie's $\infty$-operads \cite{lurie2016higher}. In \cite{haugseng2015rectification}, their theory was shown to be presented by both the strict model and, in the cartesian case, by Simpson’s model.

As a particularly tangible model for $\infty$-categories, quasi-categories have proven very successful, and their theory has seen an extensive development due to the work of Joyal, Lurie and many others. The main goal of the present paper is to lay the basic foundations for a concrete model of “enriched quasi-categories” which stand to quasi-categories as Bacard’s enriched Segal categories stand to  Segal categories. While the development of the homotopy theory of these objects is relegated to subsequent work, our constructions are motivated by homotopy theoretic considerations as we further explain.

For a suitable monoidal category $\mathcal{V}$, we define \emph{quasi-categories in $\mathcal{V}$}. Here, like the category $\Set$ of sets, $\mathcal{V}$ is a category not necessarily having any weak or higher structure. Instead, the weak enrichment should be realised in the monoidal category $S\mathcal{V}$ of simplicial objects in $\mathcal{V}$, equipped with the right-transferred model structure from $\SSet$. This model structure exists for example if the monoidal unit $I$ is a projective generator of $\mathcal{V}$ \cite{quillen1967homotopical}. The restriction to the case $\mathcal{M} = S\mathcal{V}$ allows us to keep our model tangible and elementary.

Like in the classical situation, quasi-categories in $\mathcal{V}$ arise as a subclass of a larger category.
We denote this category by $\ts\mathcal{V}$ and call its objects \emph{templicial} (short for \emph{tensor-simplicial}) \emph{objects} in $\mathcal{V}$. It is important to note that while the hom-spaces are purported to be simplicial objects, templicial objects themselves are not. This change of perspective is necessary in order to make sense of basic constructions like the nerve, as we will explain in \S\ref{parfr}.
Nonetheless, when $\mathcal{V} = \Set$, both simplicial and templicial objects recover simplicial sets. Another case of particular interest is $\mathcal{V} = \Mod(k)$, the category of modules over a commutative ring $k$. Motivated by (non-commutative) algebraic geometry, where higher categorical structures like dg- and $A_{\infty}$-categories play prominent roles as models for spaces, we focus on this case in subsequent work \cite{lowen2023frobenius}. This is an update of \cite{lowen2020linear}, which was based on a weaker notion of enriched quasi-categories.

Classically, the Joyal model structure for quasi-categories on $\SSet$ \cite{joyal2004notes} and the model structure for simplicial categories by Bergner \cite{bergner2007model} are related by the homotopy coherent nerve functor, which is the right-adjoint in a Quillen equivalence. The construction of the homotopy coherent nerve goes back to Cordier \cite{cordier1982sur} and in fact it was already shown in \cite{cordier1986Vogt} (though not with this terminology) that it preserves fibrant objects. Indeed, given a locally Kan simplicial category (that is, all its hom-objects are Kan complexes), its homotopy coherent nerve is a quasi-category. Taking this fact as a starting point, our main result is the following.

\begin{Thm*}
There is a right-adjoint functor
$$
N^{hc}_{\mathcal{V}}: \mathcal{V}\Cat_{\Delta}\rightarrow \ts\mathcal{V}
$$
from the category of small $S\mathcal{V}$-enriched categories to the category of templicial objects in $\mathcal{V}$ with the following properties.
\begin{enumerate}[1.]
\item If $\mathcal{V} = \Set$, then $N^{hc}_{\mathcal{V}}$ recovers the classical homotopy coherent nerve.
\item If $\mathcal{C}$ is a small $S\mathcal{V}$-enriched category such that its underlying simplicial category is locally Kan, then $N^{hc}_{\mathcal{V}}(\mathcal{C})$ is a quasi-category in $\mathcal{V}$.
\end{enumerate}
\end{Thm*}

We call this functor the \emph{templicial homotopy coherent nerve}. It is constructed in \S\ref{subsection: The templicial homotopy coherent nerve} and the theorem is proven in Corollary \ref{corollary: hc-nerve preserves fibrant objects}. Some other enriched versions of the homotopy coherent nerve exist in the cartesian context \cite{gindi2021coherent}\cite{moser2022homotopy}. We will investigate the relation of the latter nerve with ours in subsequent work \cite{mertens2023nerves}.

The model structure on $S\mathcal{V}$-enriched categories, which exists by \cite{berger2013homotopy}\cite{stanculescu2014constructing} \cite{muro2015Dwyer}, generalises the one of Bergner on simplicial categories. The authors expect that a generalisation of Joyal's model structure on $\SSet$ exists for $\ts\mathcal{V}$ (under suitable conditions on $\mathcal{V}$), having quasi-categories in $\mathcal{V}$ as fibrant objects and making the templicial homotopy coherent nerve into a Quillen equivalence. The weak equivalences are likely reflected by the left-adjoint, which we call the \emph{categorification functor}, for instance in the case of left transfer. This is work in progress. Note that by \cite{haugseng2015rectification}, this would establish quasi-categories in $\mathcal{V}$ as a model for $\infty$-categories enriched over $S\mathcal{V}$ in the sense of \cite{gepner2015enriched}.

\subsection{Templicial objects and necklace categories} \label{parfr}

In order to define quasi-categories in a monoidal category $\mathcal{V}$ and prove the Theorem from \ref{parin}, two larger categories play an important role: the category $\ts\mathcal{V}$ of templicial objects and the category $\mathcal{V}\Cat_{\nec}$ of necklace categories. In this section, we will explain and motivate their occurrence, starting with the former. 

Let $\mathcal{C}$ be a category enriched over the (possibly non-cartesian) monoidal category $\mathcal{V}$. When defining the nerve of $\mathcal{C}$, a natural first attempt is to put
$$
N(\mathcal{C})_{n} = \coprod_{A_{0},...,A_{n}\in \Ob(\mathcal{C})}\mathcal{C}(A_{0},A_{1})\otimes ...\otimes \mathcal{C}(A_{n-1},A_{n})\in \mathcal{V}
$$
and try to make this into a simplicial object in $\mathcal{V}$. It is readily seen that we can define inner face morphisms and degeneracy morphisms in the usual way. However, the same is not true for the outer face morphisms because in general there are no projections out of a tensor product, whence we cannot “project away” the factor $\mathcal{C}(A_{0},A_{1})$, respectively $\mathcal{C}(A_{n-1},A_{n})$, in the higher expression. As a consequence we do not obtain a simplicial object, but the higher data can be organised into a colax monoidal functor
$$
X: \fint^{op}\rightarrow \mathcal{V}
$$
where $\fint$ is the monoidal category of finite intervals (see \S\ref{subsection: Notations and conventions}). Restricting from the usual simplex category $\simp$ to $\fint$, it follows that $X$ no longer has any outer face maps, a loss which is compensated by the colax monoidal structure. It was shown by Leinster in \cite{leinster2000homotopy} (see Proposition \ref{proposition: colax functors in cartesian monoidal cat.}) that if $\mathcal{V}$ is cartesian, $X$ may still be identified with a simplicial object in $\mathcal{V}$.

The philosophy of introducing coalgebraic structure in the non-cartesian context is not uncommon. For example, Hopf algebras may be considered the group objects internal to $\Mod(k)$. Similarly, in their PhD thesis \cite{aguiar1997internal}, Aguiar introduced graphs and categories internal to a monoidal category by means of bicomodules over comonoids. Such structure is invisible in a cartesian monoidal category because then every object has a unique comonoid structure.
The same philosophy was applied by Bacard in their definition of $\mathcal{M}$-enriched Segal categories \cite{bacard2010Segal}. These are many object versions of Leinster's homotopy monoids, based on the colax monoidal functors above \cite{leinster2000homotopy}.

Let us describe templicial objects in a little more detail. In a similar but non-equivalent way to \cite{bacard2010Segal}, we define templicial objects as certain colax monoidal functors on $\fint$ with a discrete set of vertices. More precisely, a templicial object in $\mathcal{V}$ with vertex set $S$ is a strongly unital, colax monoidal functor
$$
X: \fint^{op}\rightarrow \mathcal{V}\Quiv_{S}
$$
where $\mathcal{V}\Quiv_{S}$ denotes the category of $\mathcal{V}$-enriched quivers with vertex set $S$. The colax monoidal structure equips $X$ with quiver morphisms $\mu_{k,l}: X_{k+l}\rightarrow X_{k}\otimes_{S} X_{l}$ for all $k,l\geq 0$. For example $\mu_{1,2}$ may be pictured as

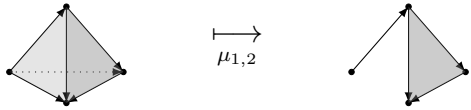
\begin{center}
\begin{tikzpicture}
\filldraw
(0,-0.289) circle (1pt)
(0.75,0.577) circle (1pt)
(1.5,-0.289) circle (1pt)
(0.75,-0.7) circle (1pt);
\draw[-latex]
(0,-0.289) -- (0.75,0.577);
\draw[-latex]
(0.75,0.577) -- (1.5,-0.289);
\draw[-latex]
(1.5,-0.289) -- (0.75,-0.7);
\draw[-latex]
(0.75,0.577) -- (0.75,-0.7);
\draw[-latex]
(0,-0.289) -- (0.75,-0.7);
\draw[-latex, dotted]
(0,-0.289) -- (1.5,-0.289);
\fill[fill=gray,opacity=0.4]
(1.5,-0.289) -- (0.75,0.577) -- (0.75,-0.7);
\fill[fill=gray,opacity=0.2]
(0,-0.289) -- (0.75,0.577) -- (0.75,-0.7);

\draw
(3,0) node{$\underset{\mu_{1,2}}{\longmapsto}$};

\filldraw
(4.5,-0.289) circle (1pt)
(5.25,0.577) circle (1pt)
(6,-0.289) circle (1pt)
(5.25,-0.7) circle (1pt);
\draw[-latex]
(4.5,-0.289) -- (5.25,0.577);
\draw[-latex]
(5.25,0.577) -- (6,-0.289);
\draw[-latex]
(6,-0.289) -- (5.25,-0.7);
\draw[-latex]
(5.25,0.577) -- (5.25,-0.7);
\fill[fill=gray,opacity=0.4]
(6,-0.289) -- (5.25,0.577) -- (5.25,-0.7);
\end{tikzpicture}
\end{center}
Intuitively, $\mu_{k,l}$ involves pulling apart $(k+l)$-simplices into $k$-simplices attached to $l$-simplices at a vertex. We can thus no longer access outer faces of a simplex directly. The shift of focus to faces joint at a vertex naturally leads us to considering necklaces (see Section \ref{section: Necklaces and necklace categories}).

Necklaces were introduced by Baues \cite{baues1980geometry} (under a different name) and popularised by Dugger and Spivak in their description of the categorification functor \cite{dugger2011rigidification}. Roughly, a necklace is a sequence of simplices glued at the end points.\\

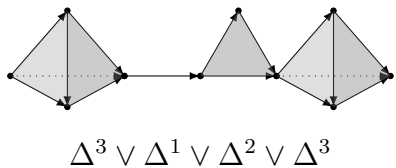
\begin{center}
\begin{tikzpicture}
\filldraw
+(0,-0.289) circle (1pt)
+(0.75,0.577) circle (1pt)
+(1.5,-0.289) circle (1pt)
+(0.75,-0.7) circle (1pt);
\draw[-latex]
+(0,-0.289) -- +(0.75,0.577);
\draw[-latex]
+(0.75,0.577) -- +(1.5,-0.289);
\draw[-latex]
+(0.75,-0.7) -- +(1.5,-0.289);
\draw[-latex]
+(0.75,+0.577) -- +(0.75,-0.7);
\draw[-latex]
+(0,-0.289) -- +(0.75,-0.7);
\draw[dotted,-latex]
+(0,-0.289) -- +(1.5,-0.289);
\fill[fill=gray,opacity=0.25]
+(0,-0.289) -- +(0.75,0.577) -- +(0.75,-0.7);
\fill[fill=gray,opacity=0.4]
+(1.5,-0.289) -- +(0.75,0.577) -- +(0.75,-0.7)
++(1.5,-0.289) node(end){};

\filldraw
(end)
+(0,0) circle (1pt)
+(1,0) circle (1pt);
\draw[-latex]
(end)
+(0,0) -- +(1,0);
\draw
(end)
++(1,+0.289) node (end){};

\draw
(end)
+(0,-1.3) node{$\Delta^{3}\vee \Delta^{1}\vee \Delta^{2}\vee \Delta^{3}$};

\filldraw
(end)
+(0,-0.289) circle (1pt)
+(0.5,0.577) circle (1pt)
+(1,-0.289) circle (1pt);
\draw[-latex]
(end)
+(0,-0.289) -- +(0.5,0.577);
\draw[-latex]
(end)
+(0.5,0.577) -- +(1,-0.289);
\draw[-latex]
(end)
+(0,-0.289) -- +(1,-0.289);
\fill[fill=gray,opacity=0.4]
(end)
+(0,-0.289) -- +(0.5,0.577) -- +(1,-0.289)
++(1,0) node(end){};

\filldraw
(end)
+(0,-0.289) circle (1pt)
+(0.75,0.577) circle (1pt)
+(1.5,-0.289) circle (1pt)
+(0.75,-0.7) circle (1pt);
\draw[-latex]
(end)
+(0,-0.289) -- +(0.75,0.577);
\draw[-latex]
(end)
+(0.75,0.577) -- +(1.5,-0.289);
\draw[-latex]
(end)
+(0.75,-0.7) -- +(1.5,-0.289);
\draw[-latex]
(end)
+(0.75,+0.577) -- +(0.75,-0.7);
\draw[-latex]
(end)
+(0,-0.289) -- +(0.75,-0.7);
\draw[dotted,-latex]
(end)
+(0,-0.289) -- +(1.5,-0.289);
\fill[fill=gray,opacity=0.25]
(end)
+(0,-0.289) -- +(0.75,0.577) -- +(0.75,-0.7);
\fill[fill=gray,opacity=0.4]
(end)
+(1.5,-0.289) -- +(0.75,0.577) -- +(0.75,-0.7);
\end{tikzpicture}
\end{center}
Necklaces naturally occur in the interpretation of the comultiplications, with $\mu_{1,2}$ being parametrised by the necklace map $$\nu_{1,2}: \Delta^{1}\vee \Delta^{2}\rightarrow \Delta^{3}$$ (see Notation \ref{notation: generating necklace maps}). As such, they allow us to turn colax monoidal functors on $\fint$ into ordinary functors on the category $\nec$ of necklaces, putting $X(\Delta^{1}\vee \Delta^{2}) = X_{1}\otimes_{S} X_{2}$ and $X(\nu_{1,2}) = \mu_{1,2}$ in the higher example.
Endowing the functor category $\mathcal{V}^{\nec^{op}}$ with the Day convolution, we define a necklace category to be a category enriched in $\mathcal{V}^{\nec^{op}}$. This allows us to realise $\ts\mathcal{V}$ as a coreflective subcategory of the category $\mathcal{V}\Cat_{\nec}$ of necklace categories:
\begin{equation}\label{diagram: coreflective embedding of temp. objs. into necklace cats.}
\ts\mathcal{V}\hookrightarrow \mathcal{V}\Cat_{\nec}.
\end{equation}
This embedding will turn up as a crucial intermediate step in defining the templicial homotopy coherent nerve in Section \ref{section: Enriching the homotopy coherent nerve}.
In the definition of quasi-categories in $\mathcal{V}$ in Section \ref{section: Quasi-categories in a monoidal category}, the category $\mathcal{V}^{\nec^{op}}$ also plays a fundamental role as the context in which we express the familiar lifting property with respect to inner horn inclusions.

\subsection{Overview of the paper}
Next, we give an overview of the contents of the paper. In Section \ref{section: Templicial objects} we formally introduce templicial objects and prove some basic properties. Starting in \S\ref{subsection: Simplicial versus templicial objects}, we compare them to simplicial sets. In \S\ref{subsection: The templicial nerve} we construct the templicial analogue of the classical nerve functor for small $\mathcal{V}$-enriched categories. For templicial objects which have non-degenerate simplices in an appropriate sense, in \S\ref{subsection: Eilenberg-Zilber}, we prove a version of the Eilenberg-Zilber lemma. In general, the structure of a templicial object $X$ is considerably richer than that of the underlying simplicial set $\tilde{U}(X)$ (see Proposition \ref{proposition: free-forget adjunction} and Remark \ref{remark: simplices of underlying simp. set}). In particular, unlike in the classical case, simplices are no longer represented by morphisms from standard simplices (the “representation problem”, see Example \ref{example: free non-cell-cofibrant temp. obj.}).

This representation problem is solved in Section \ref{section: Necklaces and necklace categories} with the introduction of necklace categories. In \S\ref{subsection: Necklaces} we recall necklaces and give a combinatorial characterisation of their category $\nec$. In \S\ref{subsection: Necklace categories} we define necklace categories and realise the category of templicial objects as a coreflective subcategory of the category $\mathcal{V}\Cat_{\nec}$ (Theorem \ref{theorem: nec functor fully faithful left-adjoint}). Finally, in \S\ref{subsection: Some constructions revisited} we observe that both the underlying simplicial set functor $\tilde{U}$ and the templicial nerve $N_{\mathcal{V}}$ naturally factor through $\mathcal{V}\Cat_{\nec}$.

In Section \ref{section: Enriching the homotopy coherent nerve}, we generalise the classical homotopy coherent nerve and the categorification functor (Definition \ref{definition: temp. hc-nerve}). We follow the elegant approach from \cite{dugger2011rigidification}, which we recall in \S\ref{subsection: The classical homotopy coherent nerve} before presenting our enriched counterpart in \S\ref{subsection: The templicial homotopy coherent nerve}. The key observation relating the two is a description of the categorification by means of a weighted colimit (Proposition \ref{proposition: categorification as weighted colimit}).

Starting from the embedding \eqref{diagram: coreflective embedding of temp. objs. into necklace cats.}, in order to construct the categorification, we construct a functor from necklace categories to $S\mathcal{V}$-enriched categories. Following \cite{dugger2011rigidification}, the categorification is simplified using flanked flags in \S\ref{subsection: Simplification of the categorification functor}, in the presence of the non-degenerate simplices from \S\ref{subsection: Eilenberg-Zilber}. Finally in \S\ref{subsection: Comparison with the templicial nerve} we show that the templicial homotopy coherent nerve reduces to the templicial nerve in the desired way. As usual, for a templicial object $X$, we naturally obtain a $\mathcal{V}$-enriched homotopy category $h_{\mathcal{V}}X$ as $\pi_{0}$ of the categorification. In general, the underlying simplicial set and underlying category functors do not commute with taking homotopy categories, as shown in Example \ref{example: homotopy and forgetful don't commute}.

In Section \ref{section: Quasi-categories in a monoidal category}, we introduce the natural analogue of quasi-categories in the templicial setting, which will remedy the aforementioned failure to commute. A quasi-category in $\mathcal{V}$ is defined as a templicial object satisfying a familiar lifting property with respect to inner horn inclusions (Definition \ref{definition: necklace functor lifts inner horns}). In contrast to the classical setup, this lifting property is considered in the category $\mathcal{V}^{\nec^{op}}$ rather than $\ts\mathcal{V}$ because of the representation problem. The resulting notion is in general strictly stronger than requiring $\tilde{U}(X)$ to be a quasi-category. Nonetheless, when $\mathcal{V} = \Set$, we still recover ordinary quasi-categories. We also show our main result (Corollary \ref{corollary: hc-nerve preserves fibrant objects}, see Theorem above). Finally, in \S\ref{subsection: Simplification of the homotopy category}, we show how the description of the homotopy category $h_{\mathcal{V}}X$ can be simplified when $X$ is a quasi-category in $\mathcal{V}$. Moreover, the underlying category of the homotopy category corresponds to the homotopy category of the underlying ordinary quasi-category.

\subsection{Notations and conventions}\label{subsection: Notations and conventions}
\begin{enumerate}[1.]
\item Throughout the text, we let $(\mathcal{V},\otimes,I)$ be a fixed bicomplete, symmetric monoidal closed category (i.e. a B\'enabou cosmos in the sense of \cite{street1974elementary}). Up to natural isomorphism, there is a unique colimit preserving functor $F: \Set\rightarrow \mathcal{V}$ such that $F(\{*\}) = I$. This functor is left-adjoint to the forgetful functor $U = \mathcal{V}(I,-): \mathcal{V}\rightarrow \Set$. Endowing $\Set$ with the cartesian monoidal structure, $F$ is strong monoidal and $U$ is lax monoidal. These notations will remain fixed as well.

\item Let $(\mathcal{W},\otimes,I)$ be an arbitrary monoidal category. Given a set $S$, we refer to a collection $Q = (Q(a,b))_{a,b\in S}$ with $Q(a,b)\in \mathcal{W}$ as a \emph{$\mathcal{W}$-enriched quiver} with $S$ its set of \emph{vertices}. A \emph{quiver morphism} $f: Q\rightarrow P$ is a collection $(f_{a,b})_{a,b\in S}$ of morphisms $f_{a,b}: Q(a,b)\rightarrow P(a,b)$ in $\mathcal{W}$. $\mathcal{W}$-enriched quivers with a fixed set of vertices $S$ and morphisms between them form a category which we denote by
$$
\mathcal{W}\Quiv_{S}
$$
This category is monoidal with product $\otimes_{S}$ and unit $I_{S}$ defined by
$$
(Q\otimes_{S} P)(a,b) = \coprod_{c\in S}Q(a,c)\otimes P(c,b)\quad \text{and}\quad I_{S}(a,b) =
\begin{cases}
I & \text{if }a = b\\
0 & \text{if }a\neq b
\end{cases}
$$
for all $Q,P\in \mathcal{W}\Quiv_{S}$ and $a,b\in S$.

\item Let $f: S\rightarrow T$ be a map of sets. We have an induced lax monoidal functor $f^{*}: \mathcal{W}\Quiv_{T}\rightarrow \mathcal{W}\Quiv_{S}$ given by $f^{*}(Q)(a,b) = Q(f(a),f(b))$ for all $\mathcal{W}$-enriched quivers $Q$ and $a,b\in S$. The functor $f^{*}$ has a left-adjoint which we denote by $f_{!}: \mathcal{W}\Quiv_{S}\rightarrow \mathcal{W}\Quiv_{T}$. It is given by
$$
f_{!}(Q)(x,y) = \coprod_{\substack{a,b\in S\\ f(a) = x, f(b) = y}}\!\!\!\!\!Q(a,b)
$$
for all $Q\in \mathcal{W}\Quiv_{S}$ and $x,y\in T$. As $f^{*}$ is canonically lax monoidal, $f_{!}$ comes equipped with an induced colax monoidal structure.

\item To relate $\mathcal{V}$-enriched and $S\mathcal{V}$-enriched categories to templicial objects (see \S\ref{subsection: The templicial nerve}, \S\ref{subsection: The templicial homotopy coherent nerve}), it will be more convenient for us to consider a $\mathcal{W}$-enriched category (or $\mathcal{W}$-category for short) as a pair $(\mathcal{C},\Ob(\mathcal{C}))$ with $\Ob(\mathcal{C})$ its set of objects and $\mathcal{C}$ a monoid in $\mathcal{W}\Quiv_{\Ob(\mathcal{C})}$. Note that this convention implies that the composition in $\mathcal{C}$ is given by a collection of morphisms in $\mathcal{W}$, for all $A,B,C\in \Ob(\mathcal{C})$:
$$
m_{\mathcal{C}}: \mathcal{C}(A,B)\otimes \mathcal{C}(B,C)\rightarrow \mathcal{C}(A,C)
$$
as opposed to the more conventional $\mathcal{C}(B,C)\otimes \mathcal{C}(A,B)\rightarrow \mathcal{C}(A,C)$. A $\mathcal{W}$-functor $\mathcal{C}\rightarrow \mathcal{D}$ is then a pair $(H,f)$ with $f: \Ob(\mathcal{C})\rightarrow \Ob(\mathcal{D})$ a map of sets and $H: \mathcal{C}\rightarrow f^{*}(\mathcal{D})$ a morphism of monoids in $\mathcal{W}\Quiv_{\Ob(\mathcal{C})}$, where we used the lax structure of $f^{*}$. We denote the category of small $\mathcal{W}$-categories and $\mathcal{W}$-functors between them by
$$
\mathcal{W}\Cat
$$

\item We will make use of the simplex categories $\simp_{surj}\subseteq \fint\subseteq \simp$, where:
\begin{itemize} 
\item $\simp$ is the ordinary \emph{simplex category}. Its objects are the posets $[n] = \{0,...,n\}$ with $n\geq 0$, and its morphisms are the order morphisms $[m]\rightarrow [n]$.
\item $\fint$ is the category of \emph{finite intervals}, which is the subcategory of $\simp$ consisting of all morphisms $f: [m]\rightarrow [n]$ that preserve the endpoints, that is $f(0) = 0$ and $f(m) = n$.
\item $\simp_{surj}$ is the subcategory of $\simp$ of all surjective morphisms $[m]\twoheadrightarrow [n]$.
\end{itemize}
Unlike $\simp$, $\fint$ and $\simp_{surj}$ carry a monoidal structure $(+,[0])$ which is given by identifying respective top and bottom endpoints, as follows. 
For all $m,n\geq 0$:
$$
[m] + [n] = [m+n]
$$
For morphisms $f: [m]\rightarrow [m']$ and $g: [n]\rightarrow [n']$ in $\fint$ or $\simp_{surj}$:
$$
(f + g)(i) =
\begin{cases}
f(i) & \text{if }i\leq m\\
m' +g(i - m) & \text{if }i\geq m
\end{cases}
$$
Note that for any morphism $f: [m]\rightarrow [n]$ in $\fint$ and $k,l\geq 0$ such that $k + l = m$, there exist unique morphisms $f_{1}: [k]\rightarrow [p]$ and $f_{2}: [l]\rightarrow [q]$ in $\fint$ such that $f_{1} + f_{2} = f$.

There is a well-known monoidal equivalence between $\fint^{op}$ and the augmented simplex category $\asimp$ (equipped with the join as monoidal product), see \cite{joyal1997disks}.
\end{enumerate}

\vspace{0,3cm}
\noindent \emph{Acknowledgement.}
The first named author would like to thank Boris Shoikhet for introducing her to Leinster's homotopy monoids and for pointing out \cite{bacard2010Segal}. The second named author is thankful to Clemens Berger for pointing out \cite{baues1980geometry} and to Lander Hermans for \cite{aguiar1997internal} as well as interesting discussions on the subject, and for valuable feedback on the introduction of the present paper. Both authors are grateful to Rune Haugseng, Bernhard Keller, Dmitry Kaledin, Tom Leinster, Michel Van den Bergh and Ittay Weiss for interesting comments and questions on the project.

\section{Templicial objects}\label{section: Templicial objects}

In \cite{aguiar1997internal}, Aguiar defined a graph internal to a monoidal category $\mathcal{W}$ as a pair $(G_{1},G_{0})$ where $G_{0}$ is a comonoid in $\mathcal{W}$ and $G_{1}$ is a bicomodule over $G_{0}$. When $\mathcal{W}$ is cartesian monoidal, then this recovers the usual notion of a graph internal to a category, namely a pair of morphisms $s,t: G_{1}\rightrightarrows G_{0}$ expressing the source and target. Extending this philosophy to higher dimensions, we propose to define a \emph{simplicial object internal to a monoidal category} $\mathcal{W}$ as a colax monoidal functor
$$
X: \fint^{op}\rightarrow \mathcal{W}
$$
where $\fint$ is the monoidal category of finite intervals (see \S\ref{subsection: Notations and conventions}). The restriction to $\fint$ precisely gets rid of the outer face maps, which are replaced by the colax monoidal structure. To justify this change, let us remark that the colax structure of $X$ provides $X_{0}\in \mathcal{W}$ with the structure of a comonoid in $\mathcal{W}$ and $X_{1}$ with that of bicomodule over $X_{0}$. In other words, $(X_{1},X_{0})$ is a graph internal to $\mathcal{W}$ in the sense of \cite{aguiar1997internal}. Moreover, it was shown by Leinster \cite{leinster2000homotopy} (reappearing here as Proposition \ref{proposition: colax functors in cartesian monoidal cat.}) that if $\mathcal{W}$ is cartesian, $X$ may be identified with a simplicial object in $\mathcal{W}$.

The need for this change in perspective is highlighted when we try to construct a nerve of an enriched category. Given a small ordinary category $\mathcal{C}$, recall that its nerve $N(\mathcal{C})$ is the simplicial set whose set of $n$-simplices is given by
$$
N(\mathcal{C})_{n} = \coprod_{A_{0},...,A_{n}\in \Ob(\mathcal{C})}\mathcal{C}(A_{0},A_{1})\times ...\times \mathcal{C}(A_{n-1},A_{n})
$$
The inner face maps $d_{j}$ for $0 < j < n$ are given by composing two consecutive morphisms in a sequence, and the degeneracy maps $s_{i}$ for $0\leq i\leq n$ are given by inserting an identity in the sequence. The outer face maps $d_{0}$ and $d_{n}$ are defined by deleting respectively the first and last entry in a sequence. Suppose now that $\mathcal{C}$ is enriched in a (possibly non-cartesian) monoidal category $(\mathcal{W},\otimes,I)$. The natural analogous way to define its nerve would be:
$$
N(\mathcal{C})_{n} = \coprod_{A_{0},...,A_{n}\in \Ob(\mathcal{C})}\mathcal{C}(A_{0},A_{1})\otimes ...\otimes \mathcal{C}(A_{n-1},A_{n})\in \mathcal{W}
$$
Then we can define inner face morphisms and degeneracy morphisms in the same way as above. However, we cannot define the outer face morphisms because in general there are no projections on the factors of a tensor product $V \otimes W$ in $\mathcal{W}$. Hence, we do not obtain a simplicial object in $\mathcal{W}$, but the higher data can be organised into a colax monoidal functor $\fint^{op}\rightarrow \mathcal{W}$.

Because enriched categories have a set of objects, we also equip such colax monoidal functors with a discrete set of vertices. Formally, we achieve this by putting $\mathcal{W} = \mathcal{V}\Quiv_{S}$ for a set $S$ (see \S\ref{subsection: Notations and conventions}) and imposing the functor is strongly unital. This leads to the definition of our main objects of study, templicial objects (Definition \ref{definition: temp. obj.}). We then define the natural analogue of the classical nerve functor in this context (Definition \ref{definition: templicial nerve}), taking $\mathcal{V}$-categories to templicial objects in $\mathcal{V}$. Finally, we show a generalisation of the Eilenberg-Zilber lemma (Lemma \ref{lemma: temp. Eilenberg-Zilber}).

\subsection{Simplicial versus templicial objects}\label{subsection: Simplicial versus templicial objects}

Let us make explicit what data are contained in a colax monoidal functor $X: \fint^{op}\rightarrow \mathcal{W}$ for a monoidal category $\mathcal{W}$ and compare it to the data of a simplicial object. For general background on (co)lax monoidal functors, we refer to the relevant literature e.g. \cite{aguiar2010monoidal}. Explicitly, such a functor $X$ consists of a collection of objects $X_{0}, X_{1}, X_{2}, ...$ along with
\begin{itemize}
\item inner face morphisms $d^{X}_{j}: X_{n}\rightarrow X_{n-1}$ for all $0 < j < n$,
\item degeneracy morphisms $s^{X}_{i}: X_{n}\rightarrow X_{n+1}$ for all $0\leq i\leq n$,
\item comultiplication morphisms $\mu^{X}_{k,l}: X_{k+l}\rightarrow X_{k}\otimes X_{l}$ for all $k,l\geq 0$,
\item a counit morphism $\epsilon^{X}: X_{0}\rightarrow I$.
\end{itemize}
These data moreover must satisfy:
\begin{itemize}
\item (Simplicial identities) For all $i,j\geq 0$, whenever these equations are well-defined,
$$
\begin{gathered}
d^{X}_{i}s^{X}_{j} = \begin{cases}
s^{X}_{j-1}d^{X}_{i} & \text{if }i < j\\
\id & \text{if } i = j\text{ or }i = j+1\\
s^{X}_{j}d^{X}_{i-1} & \text{if }i > j+1
\end{cases}\\
d^{X}_{i}d^{X}_{j} = d^{X}_{j-1}d^{X}_{i}\text{ if }i < j\qquad s^{X}_{i}s^{X}_{j} = s^{X}_{j}s^{X}_{i-1}\text{ if }i > j
\end{gathered}
$$
\item (Naturality of $\mu^{X}$) For all $k,l\geq 0$ and $0 < j < k+l+1$, $0\leq i\leq k+l-1$,
$$
\begin{aligned}
\mu^{X}_{k,l}d^{X}_{j} &=
\begin{cases}
(d^{X}_{j}\otimes \id_{X_{l}})\mu^{X}_{k+1,l} & \text{if }j\leq k\\
(\id_{X_{k}}\otimes d^{X}_{j-k})\mu^{X}_{k,l+1} & \text{if }j > k
\end{cases}\\
\mu^{X}_{k,l}s^{X}_{i} &=
\begin{cases}
(s^{X}_{i}\otimes \id_{X_{l}})\mu^{X}_{k-1,l} & \text{if }i < k\\
(\id_{X_{k}}\otimes s^{X}_{i-k})\mu^{X}_{k,l-1} & \text{if }i\geq k
\end{cases}
\end{aligned}
$$
\item (Coassociativity of $\mu^{X}$) For all $r,s,t\geq 0$, we have
$$
(\id_{X_{r}}\otimes \mu^{X}_{s,t})\mu^{X}_{r,s+t} = (\mu^{X}_{r,s}\otimes \id_{X_{t}})\mu^{X}_{r+s,t}
$$
\item (Counitality of $\mu^{X}$ with $\epsilon^{X}$) For all $n\geq 0$, we have
$$
(\id_{X_{n}}\otimes \epsilon^{X})\mu^{X}_{n,0} = \id_{X_{n}} = (\epsilon^{X}\otimes \id_{X_{n}})\mu^{X}_{0,n}
$$
\end{itemize}
Note that by the coassociativity, we have a well-defined morphism
$$
\mu^{X}_{k_{1},...,k_{n}}: X_{k_{1}+...+k_{n}}\rightarrow X_{k_{1}}\otimes ... \otimes X_{k_{n}}
$$
for all $n\geq 2$ and $k_{1},...,k_{n}\geq 0$. Further, we will set $\mu^{X}_{k_{1},...,k_{n}}$ to be the identity on $X_{k_{1}}$ if $n = 1$, and to be the counit $\epsilon^{X}$ if $n = 0$.

Moreover, under these identifications, a monoidal natural transformation $\alpha: X\rightarrow Y$ between colax monoidal functors $X,Y: \fint^{op}\rightarrow \mathcal{W}$ is equivalent to a collection of morphisms $\left(\alpha_{n}: X_{n}\rightarrow Y_{n}\right)_{n\geq 0}$ which satisfy:
\begin{itemize}
\item (Naturality of $\alpha$) For all $0 < j < n$ and $0\leq i\leq n$, we have
$$
\alpha_{n-1}d^{X}_{j} = d^{Y}_{j}\alpha_{n}\quad \text{and}\quad \alpha_{n+1}s^{X}_{i} = s^{Y}_{i}\alpha_{n}
$$
\item (Monoidality of $\alpha$) For all $k,l\geq 0$, we have
$$
\mu^{Y}_{k,l}\alpha_{k+l} = (\alpha_{k}\otimes \alpha_{l})\mu^{X}_{k,l}\quad \text{and}\quad \epsilon^{Y}\alpha_{0} = \epsilon^{X}
$$
\end{itemize}
Often we will drop the superscript $X$ when it is clear from context.

We denote by $\Colax(\fint^{op},\mathcal{W})$ the category of colax monoidal functors $\fint^{op}\rightarrow \mathcal{W}$ and monoidal natural transformations between them.

\begin{Prop}[\cite{leinster2000homotopy}, Proposition 3.1.7] \label{proposition: colax functors in cartesian monoidal cat.}
Let $\mathcal{W}$ be a cartesian monoidal category. There is an isomorphism of categories
$$
\Colax(\fint^{op}, \mathcal{W}) \simeq S\mathcal{W}.
$$
\end{Prop}

\begin{Rem}
Suppose $\mathcal{W}$ is cartesian and let $X$ be a simplicial object in $\mathcal{W}$. Its associated colax monoidal functor $\fint^{op}\rightarrow \mathcal{W}$ has comultiplication morphisms given by
$$
\mu_{k,l}: X_{k+l}\xrightarrow{(d_{k+1}...d_{n},d_{0}...d_{0})} X_{k}\times X_{l}
$$
Conversely, if $X: \fint^{op}\rightarrow \mathcal{W}$ is a colax monoidal functor. The outer face morphisms of the associated simplicial object are obtained by using the projections of the product:
$$
d_{0}: X_{n}\xrightarrow{\mu_{1,n-1}} X_{1} \times X_{n-1}\xrightarrow{\pi_{2}} X_{n-1}
$$
and
$$
d_{n}: X_{n}\xrightarrow{\mu_{n-1,1}} X_{n-1} \times X_{1}\xrightarrow{\pi_{1}} X_{n-1}
$$

If $\mathcal{W}$ is not cartesian, these projections are not available in general and the comultiplication $\mu$ of a colax monoidal functor can be considered as a replacement for the outer face morphisms in the monoidal context.
\end{Rem}

From now on for the rest of the paper, we will only consider such colax monoidal functors $X: \fint^{op}\rightarrow \mathcal{W}$ with a discrete set of vertices. We formalise this by replacing $\mathcal{W}$ by a category $\mathcal{V}\Quiv_{S}$ of $\mathcal{V}$-enriched quivers (see \ref{subsection: Notations and conventions}) for some set $S$, and imposing that $X$ is strongly unital.

An alternative but non-equivalent way to realise a set of vertices $S$ consists in turning the monoidal category $\fint$ (which is a one object bicategory) into a bicategory with object set $S$. This approach goes back to \cite{lurie2009higher} and was used in \cite{simpson2012homotopy}, \cite{bacard2010Segal}.

\begin{Def}\label{definition: temp. obj.}
A \emph{tensor-simplicial} or \emph{templicial object} in $\mathcal{V}$ is a pair $(X,S)$ with $S$ a set and
$$
X: \fint^{op} \rightarrow \mathcal{V}\Quiv_{S}
$$
a colax monoidal functor which is strongly unital, i.e. its counit $\epsilon: X_{0}\rightarrow I_{S}$ is an isomorphism. We call the elements of $S$ the \emph{vertices} of $X$. For $n > 0$, an \emph{$n$-simplex} of $X$ is an element of the set $U(X_{n}(a,b))\in \mathcal{V}$ for some $a,b\in S$.

Let $(X,S)$ and $(Y,T)$ be templicial objects. A \emph{templicial morphism} $(X,S)\rightarrow (Y,T)$ is a pair $(\alpha,f)$ with $f: S\rightarrow T$ a map of sets and $\alpha: f_{!}X\rightarrow Y$ a monoidal natural transformation between colax monoidal functors $\fint^{op}\rightarrow \mathcal{V}\Quiv_{T}$. Here, we used the colax monoidal structure of $f_{!}$.

Sometimes we will denote a templicial object $(X,S)$ or a templicial morphism $(\alpha,f)$ simply by $X$ or $\alpha$ respectively, assuming the underlying set or map of sets is clear.
\end{Def}

\begin{Rem}\label{remark: interpretation of comult. maps}
Let $(X,S)$ be a templicial object in $\mathcal{V}$ and consider $a,b \in S$. Then $X_{n}(a,b)\in \mathcal{V}$ should be interpreted as the \emph{object of $n$-simplices of $X$ with first vertex $a$ and last vertex $b$}. Moreover, for all $k,l\geq 0$ and $a,b\in S$, the comultiplication morphism
$$
(\mu^{X}_{k,l})_{a,b}: X_{k+l}(a,b)\rightarrow \coprod_{c\in S}X_{k}(a,c)\otimes X_{l}(c,b)
$$
should be interpreted as taking a $(k+l)$-simplex from $a$ to $b$ and sending it to a $k$-simplex from $a$ to some $c\in S$, along with an $l$-simplex from $c$ to $b$, which are outer faces of the original $(k+l)$-simplex.
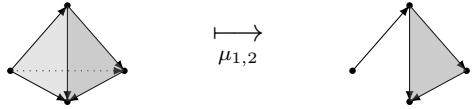
\begin{center}
\begin{tikzpicture}
\filldraw
(0,-0.289) circle (1pt)
(0.75,0.577) circle (1pt)
(1.5,-0.289) circle (1pt)
(0.75,-0.7) circle (1pt);
\draw[-latex]
(0,-0.289) -- (0.75,0.577);
\draw[-latex]
(0.75,0.577) -- (1.5,-0.289);
\draw[-latex]
(1.5,-0.289) -- (0.75,-0.7);
\draw[-latex]
(0.75,0.577) -- (0.75,-0.7);
\draw[-latex]
(0,-0.289) -- (0.75,-0.7);
\draw[-latex, dotted]
(0,-0.289) -- (1.5,-0.289);
\fill[fill=gray,opacity=0.4]
(1.5,-0.289) -- (0.75,0.577) -- (0.75,-0.7);
\fill[fill=gray,opacity=0.2]
(0,-0.289) -- (0.75,0.577) -- (0.75,-0.7);

\draw
(3,0) node{$\underset{\mu_{1,2}}{\longmapsto}$};

\filldraw
(4.5,-0.289) circle (1pt)
(5.25,0.577) circle (1pt)
(6,-0.289) circle (1pt)
(5.25,-0.7) circle (1pt);
\draw[-latex]
(4.5,-0.289) -- (5.25,0.577);
\draw[-latex]
(5.25,0.577) -- (6,-0.289);
\draw[-latex]
(6,-0.289) -- (5.25,-0.7);
\draw[-latex]
(5.25,0.577) -- (5.25,-0.7);
\fill[fill=gray,opacity=0.4]
(6,-0.289) -- (5.25,0.577) -- (5.25,-0.7);
\end{tikzpicture}
\end{center}
Unlike for simplicial objects, we thus no longer have direct access to the outer faces of a simplex, only to to outer faces which are glued at a vertex.
\end{Rem}

\begin{Def}
Given maps of sets $f: S\rightarrow T$ and $g: T\rightarrow U$, there is a canonical monoidal natural isomorphism $(gf)_{!}\simeq g_{!}f_{!}$ between colax monoidal functors $\mathcal{V}\Quiv_{S}\rightarrow \mathcal{V}\Quiv_{U}$. Consequently, we can define the \emph{composition} of two templicial morphisms $(\alpha,f): (X,S)\rightarrow (Y,T)$ and $(\beta,g): (Y,T)\rightarrow (Z,U)$ as the templicial morphism $(\gamma,gf)$ with
$$
\gamma: (gf)_{!}X\simeq g_{!}f_{!}X\xrightarrow{g_{!}\alpha} g_{!}Y\xrightarrow{\beta} Z
$$
Further, we have a canonical monoidal natural isomorphism $\varphi: (\id_{S})_{!}\xrightarrow{\sim} \id_{\mathcal{V}\Quiv_{S}}$ for any set $S$, and the \emph{identity} at $(X,S)$ is defined as the templicial morphism $(\varphi X,\id_{S})$. It is then easy to see that templicial objects in $\mathcal{V}$ and templicial morphisms between them form a category which we denote by
$$
\ts\mathcal{V}
$$
\end{Def}

\begin{Rem}
A more abstract construction of the category $\ts\mathcal{V}$ of templicial objects is as follows. One can show that the assignment $S\mapsto \Colax(\fint^{op},\Quiv_{S})$ extends to a pseudofunctor $\Phi: \Set\rightarrow \Cat$ where $\Phi(f)$ is given by post-composition with the colax monoidal functor $f_{!}$, for any map of sets $f: S\rightarrow T$. Taking the Grothendieck construction $\int\Phi$, we find $\ts\mathcal{V}$ as the full subcategory spanned by the strongly unital colax functors $\fint^{op}\rightarrow \mathcal{V}\Quiv_{S}$.
\end{Rem}

\begin{Prop}\label{proposition: temp. sets are simp. sets}
There is an equivalence of categories:
$$
S_{\times}\Set\simeq \SSet
$$
\end{Prop}
\begin{proof}
Let $K$ be a simplicial set. By Proposition \ref{proposition: colax functors in cartesian monoidal cat.}, we may consider $K$ as a colax monoidal functor $\fint^{op}\rightarrow \Set$ with comultiplication $\mu$ and counit $\epsilon$. Then define for all $n\geq 0$ and $a,b\in K_{0}$:
$$
K_{n}(a,b) = \{\sigma\in K_{n}\mid d_{1}...d_{n}(\sigma) = a, d_{0}...d_{0}(\sigma) = b\}
$$
Given $f: [m]\rightarrow [n]$ in $\mathbf{\Delta}_{f}$, it follows from the simplicial identities that $K(f): K_{n}\rightarrow K_{m}$ restricts to a map $K(f)_{a,b}: K_{n}(a,b)\rightarrow K_{m}(a,b)$. Moreover, it is clear that for all $k,l\geq 0$ and $a,b\in K_{0}$, $\mu_{k,l}$ restricts to
$$
\mu_{k,l}\vert_{K_{k+l}(a,b)}: K_{k+l}(a,b)\rightarrow \coprod_{c\in K_{0}}K_{k}(a,c)\times K_{l}(c,b)
$$
and $K_{0}(a,a) = \{a\}$ if $a = b$, and $K_{0}(a,b) = \emptyset$ if $a\neq b$. Consequently, the functor
$$
\varphi(K): \fint^{op}\rightarrow \Quiv_{K_{0}}: [n]\mapsto \left(K_{n}(a,b)\right)_{a,b\in K_{0}}$$
is strongly unital and colax monoidal. Hence $(\varphi(K),K_{0})$ is a templicial object.

Conversely, if $(X,S)$ is a templicial object in $\Set$, then we can define a simplicial set $\mathfrak{c}(X)$ by setting for all $n\geq 0$:
$$
\mathfrak{c}(X)_{n} = \coprod_{a,b\in S}X_{n}(a,b)
$$
It is readily verified that the assignments $K\mapsto \varphi(K)$ and $X\mapsto \mathfrak{c}(X)$ can be extended to mutually inverse equivalences between $\SSet$ and $S_{\times}\Set$.
\end{proof}

As $F: \Set\rightarrow \mathcal{V}$ preserves colimits and is strong monoidal, post-composition with $F$ induces a functor
$$
\tilde{F}: \SSet\simeq S_{\times}\Set\rightarrow \ts\mathcal{V}
$$
More precisely, given a simplicial set $K$, $\tilde{F}(K)$ has vertex set $K_{0}$ and for all $a,b\in K_{0}$ and $n\geq 0$ we have
$$
\tilde{F}(K)_{n}(a,b) = F(K_{n}(a,b))
$$
where $K_{n}(a,b) = \{\sigma\in K_{n}\mid d_{1}...d_{n}(\sigma) = a, d_{0}...d_{0}(\sigma) = b\}$ is the set of $n$-simplices of $K$ with first vertex $a$ and last vertex $b$, as above.

\begin{Prop}\label{proposition: free-forget adjunction}
The category of templicial objects $\ts\mathcal{V}$ is cocomplete and the functor $\tilde{F}: \SSet\rightarrow \ts\mathcal{V}$ has a right-adjoint
$$
\tilde{U}: \ts\mathcal{V}\rightarrow \SSet
$$
\end{Prop}
\begin{proof}
As $\mathcal{V}$ is cocomplete, so is $\mathcal{V}\Quiv_{S}$. It is readily verified that then also $\Colax(\fint^{op},\mathcal{V}\Quiv_{S})$ is cocomplete with colimits given pointwise. Now consider a diagram $D: \mathcal{J}\rightarrow \ts\mathcal{V}: j\mapsto (X^{j},S^{j})$ with $\mathcal{J}$ a small category. Let $S = \colim_{j\in\mathcal{J}}S^{j}$ in $\Set$ and denote $\iota_{j}: S^{j}\rightarrow S$ for the canonical map. Then consider the colimit $X = \colim_{j\in \mathcal{J}}(\iota_{j})_{!}X^{j}$ in $\Colax(\fint^{op},\mathcal{V}\Quiv_{S})$. The counit of $X$ is
$$
\colim_{j\in \mathcal{J}}(\iota_{j})_{!}(X^{j}_{0})\xrightarrow{\colim_{j}(\iota_{j})_{!}(\epsilon_{X^{j}})} \colim_{j\in \mathcal{J}}(\iota_{j})_{!}(I_{S^{j}})\xrightarrow{\sim} I_{S}
$$
which is an isomorphism since each $\epsilon_{X^{j}}$ is. Thus the pair $(X,S)$ is a templicial object, which is easily seen to be the colimit of the diagram $D$ in $\ts\mathcal{V}$.

With the above description of the colimits in $\ts\mathcal{V}$, it is clear that $\tilde{F}$ preserves colimits and therefore has a right-adjoint $\tilde{U}: \ts\mathcal{V}\rightarrow \SSet$ given by $\tilde{U}(X)_{n} = \ts\mathcal{V}(\tilde{F}(\Delta^{n}),X)$ for all templicial objects $X$ and integers $n\geq 0$.
\end{proof}

\begin{Rem}\label{remark: simplices of underlying simp. set}
Let us make the right-adjoint $\tilde{U}$ a bit more explicit. Given a templicial object $(X,S)$, an $n$-simplex of $\tilde{U}(X)$ is a templicial morphism $\tilde{F}(\Delta^{n})\rightarrow X$, which is equivalent to a pair
$$
\left((a_{i})_{i=0}^{n}, (\alpha_{i,j})_{0\le i < j\leq n}\right)
$$
with $a_{i}\in S$ and $\alpha_{i,j}\in U(X_{j-i}(a_{i},a_{j}))$ such that for all $0\leq i < k < j\leq n$:
$$
\mu_{k-i,j-k}(\alpha_{i,j}) = \alpha_{i,k}\otimes \alpha_{k,j}
$$
in $U((X_{k-i}\otimes_{S} X_{j-k})(a_{i},a_{j}))$. For example, $\tilde{U}(X)_{0}$ recovers the set $S$ and $\tilde{U}(X)_{1}$ is given by the disjoint union of all sets $U(X_{1}(a,b))$ with $a,b\in S$.
\end{Rem}

Unlike the case for simplicial objects $S\mathcal{V}$, not every $n$-simplex of a templicial object $(X,S)\in \ts\mathcal{V}$ is uniquely represented by a morphism $\tilde{F}(\Delta^{n})\rightarrow X_{n}$. More precisely, the canonical map
\begin{equation}\label{diagram: interior of an n-simplex}
\tilde{U}(X)_{n}\rightarrow \coprod_{a,b\in S}U(X_{n}(a,b)): (\alpha_{i,j})_{i,j}\mapsto \alpha_{0,n}
\end{equation}
need not be injective or surjective for $n\geq 2$, as is shown in Example \ref{example: free non-cell-cofibrant temp. obj.}.

The lack of representation of simplices by morphisms makes templicial objects considerably harder to work with than ordinary simplicial sets. In an effort to resolve this issue, we extend $\ts\mathcal{V}$ to a category of enriched categories $\mathcal{V}\Cat_{\nec}$ in \S\ref{subsection: Necklace categories}.

\begin{Ex}\label{example: free non-cell-cofibrant temp. obj.}

Let $\mathcal{V} = \Ab$ be the monoidal category of abelian groups with the tensor product as monoidal product and $\mathbb{Z}$ as monoidal unit. Consider the simplicial set $K = \partial\Delta^{2}\amalg_{\Delta^{1}}\partial\Delta^{2}$:
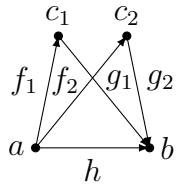
\begin{center}
\begin{tikzpicture}[scale=1.5]
\filldraw
(0,-0.289) circle (1pt) node[left]{$a$}
(0.2,0.7) circle (1pt) node[above]{$c_{1}$}
(0.8,0.7) circle (1pt) node[above]{$c_{2}$}
(1,-0.289) circle (1pt) node[right]{$b$};
\draw[-latex] (0,-0.289) -- node[left,pos=0.6]{$f_{1}$} (0.2,0.7);
\draw[-latex] (0,-0.289) -- node[left,pos=0.6]{$f_{2}$} (0.8,0.7);
\draw[-latex] (0.2,0.7) -- node[right,pos=0.4]{$g_{1}$} (1,-0.289);
\draw[-latex] (0.8,0.7) -- node[right,pos=0.4]{$g_{2}$} (1,-0.289);
\draw[-latex] (0,-0.289) -- node[below,pos=0.5]{$h$} (1,-0.289);
\end{tikzpicture}
\end{center}
We can extend $\tilde{F}(K)$ to a templicial object $X$ with a $2$-simplex $w\in X_{2}(a,b)$ by setting $X_{2}(a,b) = \mathbb{Z}w\oplus F(K_{2}(a,b))$ and similarly adding the degeneracies of $w$. The inner face maps and comultiplication maps of $X$ are uniquely determided by setting
$$
d_{1}(w) = h\quad \text{and}\quad \mu_{1,1}(w) = f_{1}\otimes g_{1} + f_{2}\otimes g_{2}
$$
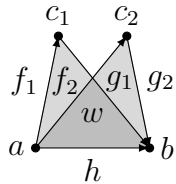
\begin{center}
\begin{tikzpicture}[scale=1.5]
\filldraw[fill=gray,opacity=0.3]
(0,-0.289) -- (0.2,0.7) -- (1,-0.289);
\filldraw[fill=gray,opacity=0.3]
(0,-0.289) -- (0.8,0.7) -- (1,-0.289);
\filldraw
(0,-0.289) circle (1pt) node[left]{$a$}
(0.2,0.7) circle (1pt) node[above]{$c_{1}$}
(0.8,0.7) circle (1pt) node[above]{$c_{2}$}
(1,-0.289) circle (1pt) node[right]{$b$}
(0.5,0) node{$w$};
\draw[-latex] (0,-0.289) -- node[left,pos=0.6]{$f_{1}$} (0.2,0.7);
\draw[-latex] (0,-0.289) -- node[left,pos=0.6]{$f_{2}$} (0.8,0.7);
\draw[-latex] (0.2,0.7) -- node[right,pos=0.4]{$g_{1}$} (1,-0.289);
\draw[-latex] (0.8,0.7) -- node[right,pos=0.4]{$g_{2}$} (1,-0.289);
\draw[-latex] (0,-0.289) -- node[below,pos=0.5]{$h$} (1,-0.289);
\end{tikzpicture}
\end{center}

Then $w$ does not lie in the image of the map \eqref{diagram: interior of an n-simplex} for $n = 2$. Indeed, this would require a $2$-simplex $(\alpha_{i,j})_{0\leq i < j\leq 2}$ of $\tilde{U}(X)$ with $\alpha_{0,2} = w$. But $\mu_{1,1}(w)$ is not a pure tensor while $\mu_{1,1}(\alpha_{0,2}) = \alpha_{0,1}\otimes \alpha_{1,2}$. Moreover, since $\mu_{1,1}(2s_{0}(h)) = 2s_{0}(a)\otimes h = s_{0}(a)\otimes 2h$, we have two distinct $2$-simplices of $\tilde{U}(X)_{2}$ which map to $2s_{0}(h)\in X_{2}(a,b)$.
\end{Ex}

\subsection{The templicial nerve}\label{subsection: The templicial nerve}

Given a monoidal category $\mathcal{W}$, there is a well-known equivalence (this goes back to Maclane \cite[\S V.II]{maclane1971categories}):
\begin{equation}\label{equation: universal monoid}
\Mon(\mathcal{W})\simeq \mathrm{StrMon}(\asimp,\mathcal{W})
\end{equation}
between the categories of monoids in $\mathcal{W}$ and strong monoidal functors $\asimp\rightarrow \mathcal{W}$. Due to the monoidal equivalence $\asimp\simeq \fint^{op}$, we may as well consider strong monoidal functors $\fint^{op}\rightarrow \mathcal{W}$.

\begin{Def}\label{definition: templicial nerve}
Let $\mathcal{C}$ be a small $\mathcal{V}$-category, which we consider as a monoid $(\mathcal{C},m_{\mathcal{C}},u_{\mathcal{C}})$ in $\mathcal{V}\Quiv_{\Ob(\mathcal{C})}$. Applying \eqref{equation: universal monoid} to the case $\mathcal{W} = \mathcal{V}\Quiv_{\Ob(\mathcal{C})}$, we obtain an associated strong monoidal functor which we denote by
$$
N_{\mathcal{V}}(\mathcal{C}): \fint^{op}\rightarrow \mathcal{V}\Quiv_{\Ob(\mathcal{C})}
$$
In particular, the pair $(N_{\mathcal{V}}(\mathcal{C}),\Ob(\mathcal{C}))$ forms a templicial object which we call the \emph{templicial nerve} of the $\mathcal{V}$-category $\mathcal{C}$.
\end{Def}

Explicitly, $N_{\mathcal{V}}(\mathcal{C})$ is given by taking the $n$-fold monoidal product of the $\mathcal{V}$-quiver $\mathcal{C}$:
$$
N_{\mathcal{V}}(\mathcal{C})_{n} = \mathcal{C}^{\otimes n}
$$
for all integers $n\geq 0$. Further, the inner face and degeneracy morphisms are
\begin{align*}
d_{j} &= \id_{\mathcal{C}}^{\otimes j-1}\otimes_{S} m_{\mathcal{C}}\otimes_{S} \id_{\mathcal{C}}^{\otimes n-j-1}: \mathcal{C}^{\otimes n}\rightarrow \mathcal{C}^{\otimes n-1}\\
s_{i} &= \id_{\mathcal{C}}^{\otimes i}\otimes_{S} u_{\mathcal{C}}\otimes_{S} \mathcal{C}^{\otimes n-i}: \mathcal{C}^{\otimes n}\rightarrow \mathcal{C}^{\otimes n+1}
\end{align*}
for all $0\leq i\leq n$ and $0 < j < n$. Finally, the comultiplication morphisms and counit are given by the canonical isomorphisms
$$
\mu_{k,l}: \mathcal{C}^{\otimes k+l}\xrightarrow{\sim} \mathcal{C}^{\otimes k}\otimes_{S} \mathcal{C}^{\otimes l}\quad \text{and}\quad \epsilon: \mathcal{C}^{\otimes 0}\xrightarrow{\sim} I_{\Ob(\mathcal{C})}
$$
for any $k,l\geq 0$.

Recall the base change functor $f_{!}: \mathcal{V}\Quiv_{S}\rightarrow \mathcal{V}\Quiv_{T}$ and its right-adjoint $f^{*}: \mathcal{V}\Quiv_{T}\rightarrow \mathcal{V}\Quiv_{S}$ for a given map of sets $f: S\rightarrow T$ (see \S\ref{subsection: Notations and conventions}).

\begin{Lem}\label{lemma: nerve unique extension}
Let $(X,S)$ be a templicial object, $\mathcal{C}$ a small $\mathcal{V}$-enriched category and $f: S\rightarrow \Ob(\mathcal{C})$ a map of sets. Then we have a bijection between monoidal natural transformations $f_{!}X\rightarrow N_{\mathcal{V}}(\mathcal{C})$ and quiver morphisms $H: X_{1}\rightarrow f^{*}(\mathcal{C})$ such that the diagrams
\begin{equation}\label{diagram: nerve unique extension}
\begin{tikzcd}
	{X^{\otimes 2}_{1}} & {f^{*}(\mathcal{C})^{\otimes 2}} & {f^{*}(\mathcal{C}^{\otimes 2})} \\
	{X_{2}} & {X_{1}} & {f^{*}(\mathcal{C})}
	\arrow["{H^{\otimes 2}}", from=1-1, to=1-2]
	\arrow[from=1-2, to=1-3]
	\arrow["{\mu_{1,1}}", from=2-1, to=1-1]
	\arrow["{f^{*}(\tilde{m}_{\mathcal{C}})}", from=1-3, to=2-3]
	\arrow["H"', from=2-2, to=2-3]
	\arrow["{d_{1}}"', from=2-1, to=2-2]
\end{tikzcd}
\begin{tikzcd}
	& {I_{S}} & {f^{*}(I_{\Ob(\mathcal{C})})} \\
	{X_{0}} & {X_{1}} & {f^{*}(\mathcal{C})}
	\arrow["\epsilon", from=2-1, to=1-2]
	\arrow["\sim"', from=2-1, to=1-2]
	\arrow[from=1-2, to=1-3]
	\arrow["{f^{*}(u_{\mathcal{C}})}", from=1-3, to=2-3]
	\arrow["{s_{0}}"', from=2-1, to=2-2]
	\arrow["H"', from=2-2, to=2-3]
\end{tikzcd}
\end{equation}
commute.
\end{Lem}
\begin{proof}
Given a monoidal natural transformation $\alpha: f_{!}X\rightarrow N_{\mathcal{V}}(\mathcal{C})$, define $H_{\alpha}: X_{1}\rightarrow f^{*}(\mathcal{C})$ to be adjoint to $\alpha_{1}: f_{!}(X_{1})\rightarrow \mathcal{C}$. It follows from the monoidality of $\alpha$ that for all $n\geq 0$, $\alpha_{n}$ is the composite
$$
f_{!}(X_{n})\xrightarrow{f_{!}(\mu_{1,...,1})} f_{!}(X_{1}^{\otimes n})\rightarrow f_{!}(X_{1})^{\otimes n}\xrightarrow{\alpha^{\otimes n}_{1}} \mathcal{C}^{\otimes n}
$$
where we used the colax monoidal structure of $f_{!}$. So the assignment $\alpha\mapsto H_{\alpha}$ is injective. Moreover, it then follows from the naturality of $\alpha$ that $H_{\alpha}$ satisfies (\ref{diagram: nerve unique extension}). Conversely, if $H: X_{1}\rightarrow f^{*}(\mathcal{C})$ satisfies (\ref{diagram: nerve unique extension}), then defining $\alpha_{1}$ as adjoint to $H$ and $\alpha_{n}$ as above, it follows that $\alpha: f_{!}X\rightarrow N_{\mathcal{V}}(\mathcal{C})$ is a natural transformation. It is immediate that $\alpha$ is also monoidal.
\end{proof}

\begin{Prop}\label{proposition: nerve fully faithful}
The assignment $\mathcal{C}\mapsto N_{\mathcal{V}}(\mathcal{C})$ of Definition \ref{definition: templicial nerve} extends to a fully faithful functor $N_{\mathcal{V}}: \mathcal{V}\Cat\rightarrow \ts\mathcal{V}$. The essential image of $N_{\mathcal{V}}$ consists of all templicial objects $(X,S)$ for which $X: \fint^{op}\rightarrow \mathcal{V}\Quiv_{S}$ is strong monoidal.
\end{Prop}
\begin{proof}
Let $\mathcal{C}$ and $\mathcal{D}$ be small $\mathcal{V}$-enriched categories, $f: \Ob(\mathcal{D})\rightarrow \Ob(\mathcal{C})$ a map of sets and $H: \mathcal{D}\rightarrow f^{*}(\mathcal{C})$ a morphism in $\mathcal{V}\Quiv_{\Ob(\mathcal{D})}$. Then the diagrams (\ref{diagram: nerve unique extension}) with $X = N_{\mathcal{V}}(\mathcal{D})$ precisely express that $(H,f)$ is a $\mathcal{V}$-functor $\mathcal{D}\rightarrow\mathcal{C}$, and thus we have a bijection $\mathcal{V}\Cat(\mathcal{D},\mathcal{C})\simeq \ts\mathcal{V}(N_{\mathcal{V}}(\mathcal{D}),N_{\mathcal{V}}(\mathcal{C}))$. More precisely, the templicial morphism $N_{\mathcal{V}}(H)$ corresponding to some $\mathcal{V}$-functor $H: \mathcal{D}\rightarrow \mathcal{C}$ is given by
$$
N_{\mathcal{V}}(H)_{n}: f_{!}(\mathcal{D}^{\otimes n})\rightarrow f_{!}(\mathcal{D})^{\otimes n}\xrightarrow{N_{\mathcal{V}}(H)^{\otimes n}_{1}} \mathcal{C}^{\otimes n}
$$
for all $n\geq 0$, where $N_{\mathcal{V}}(H)_{1}: f_{!}(\mathcal{D})\rightarrow \mathcal{C}$ is adjoint to $H: \mathcal{D}\rightarrow f^{*}(\mathcal{D})$. Thus clearly this defines a functor which is necessarily fully faithful. The characterisation of the essential image follows immediately from \eqref{equation: universal monoid}.
\end{proof}

\begin{Rem}\label{remark: temp. nerve generalizes classical nerve}
Note that in case $(\mathcal{V},\otimes,I) = (\Set,\times,\{*\})$, the templicial nerve functor $N_{\mathcal{V}}: \mathcal{V}\Cat\rightarrow \ts\mathcal{V}$ clearly recovers the classical nerve functor $N: \Cat\rightarrow \SSet$.
\end{Rem}

The adjunction $F\dashv U$ induces an adjunction between small categories and small $\mathcal{V}$-enriched categories, whiich we denote by
\[\begin{tikzcd}
	\Cat & {\mathcal{V}\Cat}
	\arrow[""{name=0, anchor=center, inner sep=0}, "{\mathcal{F}}", shift left=2, from=1-1, to=1-2]
	\arrow[""{name=1, anchor=center, inner sep=0}, "{\mathcal{U}}", shift left=2, from=1-2, to=1-1]
	\arrow["\dashv"{anchor=center, rotate=-90}, draw=none, from=0, to=1]
\end{tikzcd}\]

\begin{Prop}\label{proposition: nerve commutativity results}
We have natural isomorphisms
$$
N_{\mathcal{V}}\circ \mathcal{F}\simeq \tilde{F}\circ N\quad \text{and}\quad \tilde{U}\circ N_{\mathcal{V}}\simeq N\circ \mathcal{U}
$$
\end{Prop}
\begin{proof}
The first isomorphism follows from the construction of $N_{\mathcal{V}}$ and the fact that $F$ is strong monoidal and preserves colimits. Further let $\mathcal{C}$ be a small $\mathcal{V}$-category and $n\geq 0$. Then we have isomorphisms, natural in $\mathcal{C}$ and $n$:
\begin{align*}
\tilde{U}(N_{\mathcal{V}}(\mathcal{C}))_{n}&= \ts\mathcal{V}(\tilde{F}(\Delta^{n}),N_{\mathcal{V}}(\mathcal{C})) \simeq \ts\mathcal{V}(N_{\mathcal{V}}(\mathcal{F}([n])),N_{\mathcal{V}}(\mathcal{C}))\\
&\simeq \mathcal{V}\Cat(\mathcal{F}([n]),\mathcal{C})\simeq \Cat([n],\mathcal{U}(\mathcal{C}))\simeq N(\mathcal{U}(\mathcal{C}))_{n}
\end{align*}
where we subsequently used the first isomorphism and Proposition \ref{proposition: nerve fully faithful}.
\end{proof}

\subsection{The Eilenberg-Zilber lemma}\label{subsection: Eilenberg-Zilber}

Recall the classical Eilenberg-Zilber lemma for simplicial sets  \cite[(8.3)]{eilenberg1950semi-simplicial}. It states that for any $n$-simplex $x$ of a simplicial set $K$, there is a unique non-degerate $k$-simplex $y$ of $K$ and a unique surjective map $\sigma: [n]\twoheadrightarrow [k]$ in $\simp$ such that $x = K(\sigma)(y)$. Equivalently, there exists a bijection
$$
K_{n}\simeq \coprod_{\substack{\sigma: [n]\twoheadrightarrow [k]\\ \text{in }\simp_{surj}}}\!\!\!K^{nd}_{k}
$$
where $K^{nd}_{k}\subseteq K_{k}$ denotes the subset of non-degenerate $k$-simplices of $K$, and $\simp_{surj}\subseteq \simp$ is the subcategory of surjective maps (see \S\ref{subsection: Notations and conventions}).

The analoguous statement for templicial objects (Lemma \ref{lemma: temp. Eilenberg-Zilber}) also holds, but this requires an extra condition to ensure that they have a well-behaved notion of non-degenerate simplices. We will make use of this lemma when we reformulate the left-adjoint of the templicial homotopy coherent nerve in \S\ref{subsection: Simplification of the categorification functor}.

\begin{Def}\label{definition: latching obj. of a temp. obj.}
Consider a functor $X: \simp^{op}_{surj}\rightarrow \mathcal{W}$ with $\mathcal{W}$ a cocomplete category. For every integer $n\geq 0$, we denote
$$
X^{deg}_{n} = \colim_{\substack{\sigma: [n]\twoheadrightarrow [k]\\ \text{in }\simp_{surj}\\ 0\leq k < n}}X_{k}
$$
Note that we have a canonical morphism $X^{deg}_{n}\rightarrow X_{n}$ in $\mathcal{W}$.

Let $(X,S)$ be a templicial object and consider the restricted functor $X\vert_{\simp^{op}_{surj}}: \simp^{op}_{surj}\rightarrow \mathcal{V}\Quiv_{S}$. For $n\geq 0$, we call $X^{deg}_{n}$ the \emph{quiver of degenerate $n$-simplices} of $X$. We say $X$ \emph{has non-degenerate simplices} if for every $n\geq 0$, the quiver morphism $X^{deg}_{n}\rightarrow X_{n}$ is isomorphic to a coprojection
$$
X^{deg}_{n}\rightarrow X^{deg}_{n}\amalg N
$$
for some $N\in \mathcal{V}\Quiv_{S}$. In this case, we'll often denote $N$ by $X^{nd}_{n}$. When considering an abstract templicial object that has non-degenerate simplices, we implicitly assume a choice for $X^{nd}_{n}$ in each dimension has been made. Note that $X^{deg}_{0} = 0$ and so we always have $X^{nd}_{0}\simeq X_{0}$.
\end{Def}

\begin{Ex}
Let $(X,S)$ be a templicial object and suppose the underlying functor $X\vert_{\simp^{op}_{surj}}: \simp^{op}_{surj}\rightarrow \mathcal{V}\Quiv_{S}$ is isomorphic to $FZ$ for some $Z: \simp^{op}_{surj}\rightarrow \Quiv_{S}$ with $Z^{deg}_{n}(a,b)\rightarrow Z_{n}(a,b)$ injective for all $a,b\in S$ and $n\geq 0$. Then $X$ has non-degenerate simplices. Indeed, simply set
$$
X^{nd}_{n}(a,b) = F(Z_{n}(a,b)\setminus Z^{deg}_{n}(a,b))
$$
In particular, for any a simplicial set $K$, the templicial object $\tilde{F}(K)$ has non-degenerate simplices.
\end{Ex}

Certainly not every templicial object has non-degenerate simplices, as the following example shows.

\begin{Ex}\label{example: temp. obj. with ill-defined non-deg. simplices}
Consider the monoidal category $\mathcal{V} = \Ab$ of abelian groups with the tensor product. Let $S = \{*\}$ be a singleton and define a functor $X: \fint^{op}\rightarrow \Ab$ by setting $X_{n} = \mathbb{Z}$ for all $n\geq 0$ with
$$
s_{0}: X_{0} = \mathbb{Z}\xrightarrow{2\cdot} X_{1} = \mathbb{Z}
$$
and all other face and degeneracy maps given by the identity on $\mathbb{Z}$. Then $X$ is a strongly unital, colax monoidal functor with comultiplication map $\mu_{k,l}$ for $k,l\geq 0$ given by
$$
\mu_{k,l}: X_{k+l} = \mathbb{Z}\rightarrow X_{k}\otimes X_{l}\simeq \mathbb{Z}: z\mapsto
\begin{cases}
2z & \text{if }k,l > 0\\
z & \text{if }k = 0\text{ or }l = 0
\end{cases}
$$
We thus find a templicial abelian group $(X,S)$ for which $X^{deg}_{1}\rightarrow X_{1}$ is given by the inclusion $2\mathbb{Z}\subseteq \mathbb{Z}$ which doesn't have a direct complement.
\end{Ex}

\begin{Lem}\label{lemma: temp. Eilenberg-Zilber}
Let $X$ be a templicial object and assume it has non-degenerate simplices. For any integer $n\geq 0$, we have an isomorphism of quivers:
$$
X_{n}\simeq \coprod_{\substack{\sigma: [n]\twoheadrightarrow [k]\\ \text{in }\simp_{surj}}}\!\!\!X^{nd}_{k}
$$
\end{Lem}
\begin{proof}
By definition, $X_{0} = X^{nd}_{0}$. Take $n > 0$, then it follows by induction that
\begin{align*}
X_{n} &\simeq X^{nd}_{n}\amalg X^{deg}_{n} = X^{nd}_{n}\amalg \colim_{\substack{[n]\twoheadrightarrow [k]\\ 0\leq k < n}}X_{k}\simeq X^{nd}_{n}\amalg \colim_{\substack{[n]\twoheadrightarrow [k]\\ 0\leq k < n}}\coprod_{\sigma: [k]\twoheadrightarrow [l]}X^{nd}_{l}\\
&\simeq X^{nd}_{n}\amalg \coprod_{\substack{\sigma: [n]\twoheadrightarrow [l]\\ 0\leq l < n}}\colim_{\substack{[n]\underset{\sigma_{1}}{\twoheadrightarrow} [k]\underset{\sigma_{2}}{\twoheadrightarrow} [l]\\ \sigma = \sigma_{2}\sigma_{1}}}\!\!\!X^{nd}_{l}\simeq X^{nd}_{n}\amalg \coprod_{\substack{\sigma: [n]\twoheadrightarrow [l]\\ 0\leq l < n}}X^{nd}_{l}
\end{align*}
The last isomorphism is obtained by noting that the colimit on the left hand side is taken over a category which has a terminal object given by the factorisation $[n]\overset{=}{\twoheadrightarrow} [n]\overset{\sigma}{\twoheadrightarrow} [l]$.
\end{proof}

\section{Necklaces and necklace categories}\label{section: Necklaces and necklace categories}

Necklaces were first introduced by Baues \cite{baues1980geometry} and popularised by Dugger and Spivak in \cite{dugger2011rigidification}. Their category $\nec$ will play a crucial role in what follows. Morally a necklace is simply a sequence of simplices glued together at vertices. In view of Remark \ref{remark: interpretation of comult. maps}, necklaces appear naturally when applying the comultiplication morphism $\mu_{k,l}$ of a templicial object $X$. In this way, maps between necklaces parametrise the degeneracy and inner face morphisms of $X$, as well as its comultiplication morphisms. This change in perspective leads us to considering the category $\mathcal{V}\Cat_{\nec}$ of small categories enriched in $\mathcal{V}^{\nec^{op}}$. We call these necklace categories and show in \S\ref{subsection: Necklace categories} that we can recover $\ts\mathcal{V}$ as a coreflective subcategory of $\mathcal{V}\Cat_{\nec}$ (see Theorem \ref{theorem: nec functor fully faithful left-adjoint}).

\subsection{Necklaces}\label{subsection: Necklaces}

We quickly recall the definition of a necklace and give a new combinatorial description of the category $\nec$ of necklaces in Proposition \ref{proposition: combinatorial charac. of nec}.

\begin{Def}
We denote by $\SSet_{*,*} = (\partial\Delta^{1}\downarrow \SSet)$ the category of \emph{bipointed simplicial sets}. Its objects can be identified with tuples $(K,a,b)$ where $K$ is a simplicial set and $a,b\in K_{0}$ are called the \emph{distinguished points} of $K$. We will also denote $K_{a,b} = (K,a,b)$. A morphism $K_{a,b}\rightarrow L_{c,d}$ in $\SSet_{*,*}$ is a simplicial map $f: K\rightarrow L$ such that $f(a) = c$ and $f(b) = d$.

Let $K_{a,b}$ and $L_{c,d}$ be bipointed simplicial sets. The \emph{wedge sum} $K\vee L$ of $K$ and $L$ is constructed by glueing $K$ and $L$ at the distinguished points $b$ and $c$. More precisely, $K\vee L$ is given by the coequaliser
\[\begin{tikzcd}
	{\Delta^{0}} & {K\amalg L} & {K\vee L}
	\arrow["b", shift left=1, from=1-1, to=1-2]
	\arrow["c"', shift right=1, from=1-1, to=1-2]
	\arrow[two heads, from=1-2, to=1-3]
\end{tikzcd}\]
We consider $K\vee L$ again as bipointed with distinguished points $(a,d)$.
\end{Def}

\begin{Rem}
It is not difficult to verify that the wedge $\vee$ is a monoidal product on the category of bipointed simplicial sets $\SSet_{*,*}$ whose unit is given by $\Delta^{0}$.
\end{Rem}

\begin{Def}\label{definition: necklace}
For any $n\geq 0$, we consider the standard simplex $\Delta^{n}$ as bipointed with distinguished points $0$ and $n$. A \emph{necklace} $T$ is an iterated wedge of standard simplices. That is,
$$
T = \Delta^{n_{1}}\vee ...\vee \Delta^{n_{k}}\in \SSet_{*,*}
$$
for some $k\geq 0$ and $n_{1},...,n_{k} > 0$ (if $k = 0$, then $T = \Delta^{0}$). We refer to the standard simplices $\Delta^{n_{1}},...,\Delta^{n_{k}}$ as the \emph{beads} of $T$. The distinguished points in every bead are called the \emph{joints} of $T$.

We let $\nec$ denote the full subcategory of $\SSet_{*,*}$ spanned by all necklaces. By construction, $(\nec,\vee,\Delta^{0})$ is again a monoidal category.
\end{Def}

\begin{Prop}\label{proposition: combinatorial charac. of nec}
The category of necklaces $\nec$ is equivalent to the category defined as follows:

The objects are pairs $(T,p)$ with $p\geq 0$ and $\{0 < p\}\subseteq T\subseteq [p]$. The morphisms $(T,p)\rightarrow (U,q)$ are morphisms $f: [p]\rightarrow [q]$ in $\fint$ such that $U\subseteq f(T)$, with compositions and identities defined as in $\fint$.

Moreover, under this equivalence, the wedge $\vee$ corresponds to
$$
(T,p)\vee (U,q) = (T\cup (p + U), p + q)
$$
where $p + U = \{p + u\mid u\in U\}$.
\end{Prop}
\begin{proof}
It is clear that we may identify a necklace $T = \Delta^{n_{1}}\vee ...\vee \Delta^{n_{k}}$ with the subset $T = \{0 < n_{1} < n_{1} + n_{2} < ... < p\}\subseteq [p]$ where $p = n_{1} + ... + n_{k}$, and we will do so for the rest of this proof. Note that under this identification, $[p]$ is the set of vertices and $T$ is the set of joints of the necklace. Further, a necklace map $T\rightarrow U$ is completely determined on vertices, which is a morphism $[p]\rightarrow [q]$ in $\fint$. It remains to show that a morphism $f: [p]\rightarrow [q]$ in $\fint$ is the vertex map of a necklace map $T\rightarrow U$ if and only if $U\subseteq f(T)$.

Suppose $f$ is the vertex map of some necklace map $T\rightarrow U$. Assume that there exists a $u\in U\setminus f(T)$. Then we may choose subsequent joints $t < t'$ in $T$ such that $f(t) < u < f(t')$. Now the unique edge in $T$ between $t$ and $t'$ must be sent to an edge in $U$ between $f(t)$ and $f(t')$. But there is no such edge. Conversely, assume $U\subseteq f(T)$. We may write $T = \{0 = t_{0} < t_{1} < ... < t_{k-1} < t_{k} = p\}$ and $U = \{0 = u_{0} < u_{1} < ... < u_{l-1} < u_{l} = q\}$. Then we have $f = f_{1} + ... + f_{k}$ for some unique $f_{i}: [t_{i} - t_{i-1}]\rightarrow [f(t_{i}) - f(t_{i-1})]$ in $\fint$. Fixing $i\in \{1,...,k\}$, there is a unique $j\in \{1,...,l\}$ such that $u_{j-1}\leq f(t_{i-1})\leq f(t_{i})\leq u_{j}$. So we can extend $f_{i}$ to an order morphism $[t_{i} - t_{i-1}]\rightarrow [u_{j} - u_{j-1}]$, which induces a simplicial map $\Delta^{t_{i} - t_{i-1}}\rightarrow \Delta^{u_{j} - u_{j-1}}\rightarrow U$. These maps combine to give a map of necklaces $T\rightarrow U$ whose vertex map is $f$.

Clearly, this correspondence is functorial and preserves the wedge sum $\vee$.
\end{proof}

Henceforth, we will identify $\nec$ with the category described in Proposition \ref{proposition: combinatorial charac. of nec}. So we will also use the notation
$$
T = \{0 = t_{0} < t_{1} < t_{2} < ...  < t_{k} = p\}
$$
to refer to the necklace $\Delta^{t_{1}}\vee \Delta^{t_{2}-t_{1}}\vee ...\vee \Delta^{p-t_{k-1}}$. We will often refer to a necklace $(T,p)$ just by its underlying set of joints $T$.

\begin{Def}\label{definition: active and inert necklace map}
Let $f: (T,p)\rightarrow (U,q)$ be a map of necklaces. We say $f$ is \emph{inert} if $p = q$ and $f = \id_{[p]}$. We say $f$ is \emph{active} if $f(T) = U$.
\end{Def}

\begin{Rem}
Note that every necklace map $f: (T,p)\rightarrow (U,q)$ can be uniquely factored as an active necklace map $(T,p)\rightarrow (f(T),q)$ followed by an inert necklace map $(f(T),q)\rightarrow (U,q)$. In fact, it is easy to see that the active and inert necklace maps form an (orthogonal) factorisation system on $\nec$ in the sense of \cite{bousfield1976constructions}.
\end{Rem}

\begin{Rem}\label{remark: simplices and spines in cat. of necklaces}
A simplex $\Delta^{n}$, considered as a necklace with a single bead, is represented in $\nec$ by the pair $(\{0 < n\},n)$. On the other hand, the necklace $([n],n)$ represents the spine of $\Delta^{n}$, that is the union of the edges $0\rightarrow 1\rightarrow ... \rightarrow n$ in $\Delta^{n}$.

More generally, for any necklace $(T,p)$ we can consider $([p],p)$, which is the spine passing through all the vertices of $T$. Note that there is a unique inert necklace map $([p],p)\rightarrow (T,p)$ which represents the inclusion of the spine into $T$. Further, there is a unique order isomorphism $[k]\simeq T$ where $k$ is the number of beads of $T$. Thus there is a unique active map $([k],k)\rightarrow (T,p)$, which is the inclusion of the spine passing through all the joints of $T$.
\end{Rem}

\subsection{Necklace categories}\label{subsection: Necklace categories}

Consider the category $\mathcal{V}^{\nec^{op}}$ of functors $\nec^{op}\rightarrow \mathcal{V}$. As $\nec^{op}$ and $\mathcal{V}$ are both monoidal categories, we can endow $\mathcal{V}^{\nec^{op}}$ with the (non-symmetric) monoidal structure given by Day convolution (see \cite{day1970closed}). We denote the resulting monoidal category by $(\mathcal{V}^{\nec^{op}},\otimes_{Day},\underline{I})$.

Given two functors $X,Y: \nec^{op}\rightarrow \mathcal{V}$, their Day convolution $X\otimes_{Day} Y$ is obtained by the left Kan extension of the composite
$$
\nec^{op}\times \nec^{op}\xrightarrow{X\times Y} \mathcal{V}\times \mathcal{V}\xrightarrow{-\otimes -} \mathcal{V}
$$
along $\vee: \nec^{op}\times \nec^{op}\rightarrow \nec^{op}$:
$$
X\otimes_{Day} Y = \Lan_{\vee}(X(-)\otimes Y(-))
$$

Further, the monoidal unit of $\mathcal{V}^{\nec^{op}}$ is given by the representable functor on the monoidal unit $\{0\}$ of $\nec$. As $\{0\}$ is also the terminal object of $\nec$, we find that $F(\nec(-,\{0\}))\simeq \underline{I}$ is the constant functor on $I$, the monoidal unit of $\mathcal{V}$.

\begin{Def}\label{definition: necklace cats.}
Consider the category
$$
\mathcal{V}\Cat_{\nec} = \mathcal{V}^{\nec^{op}}\text{-}\Cat
$$
of small categories enriched in the monoidal category $(\mathcal{V}^{\nec^{op}},\otimes_{Day},\underline{I})$. We call the objects of $\mathcal{V}\Cat_{\nec}$ \emph{necklace categories} and its morphisms \emph{necklace functors}.

If $\mathcal{V} = \Set$, we simply write $\Cat_{\nec}$ for $\Set\Cat_{\nec}$.
\end{Def}

\begin{Con}\label{construction: necklace cat. associated to temp. obj.}
We construct a functor
$$
(-)^{nec}: \ts\mathcal{V}\rightarrow \mathcal{V}\Cat_{\nec}
$$
as follows. Let $(X,S)$ be a templicial object. Define
$$
X_{T} = X_{t_{1}}\otimes_{S} ...\otimes_{S} X_{n-t_{k-1}}\in \mathcal{V}\Quiv_{S}
$$
for any necklace $T = \{0 = t_{0} < t_{1} < ... < t_{k} = p\}$,. We also denote the quiver morphism $\mu_{t_{1},t_{2}-t_{1},...,p-t_{k-1}}: X_{p}\rightarrow X_{T}$ by $\mu_{T}$.

This extends to a functor $X^{nec}_{\bullet}: \nec^{op}\rightarrow \mathcal{V}\Quiv_{S}$ as follows. Take a necklace map $f: (T,p)\rightarrow (U,q)$ and write $U = \{0 = u_{0} < u_{1} < ... < u_{l} = q\}$.
\begin{itemize}
\item If $f$ is inert, then $p = q$ and $U\subseteq T$. Then there exist unique necklaces\\ $(T_{i},u_{i}-u_{i-1})$ for $i\in \{1,...,l\}$ such that $T = T_{1}\vee ...\vee T_{l}$. Now set
$$
X(f): X_{U}\xrightarrow{\mu_{T_{1}}\otimes ...\otimes \mu_{T_{l}}} X_{T}
$$
\item If $f$ is active, then there exist unique $f_{i}: [t_{i}-t_{i-1}]\rightarrow [f(t_{i}) - f(t_{i-1})]$ in $\fint$ for all $i\in \{1,...,k\}$ such that $f = f_{1} + ... + f_{k}$. Now set
$$
X(f): X_{U}\simeq X_{f(t_{1})}\otimes_{S} ...\otimes_{S} X_{q - f(t_{k-1})}\xrightarrow{X(f_{1})\otimes ...\otimes X(f_{k})} X_{T}
$$
where the isomorphism is induced by the strong unitality of $X$ and the fact that $U = f(T)$.
\end{itemize}
It follows from the coassociativity of $\mu$ that $X_{\bullet}$ is functorial on inert morphisms, and from the functoriality of $X$ that $X_{\bullet}$ is functorial on active morphisms. Then it follows from the naturality of $\mu$ that $X_{\bullet}$ is functorial on all morphisms.

If we fix vertices $a,b\in S$, then we obtain a functor
$$
X_{\bullet}(a,b): \nec^{op}\rightarrow \mathcal{V}: T\mapsto X_{T}(a,b)
$$
Let now $X^{nec}$ denote the necklace category with $S$ as its object set and $X_{\bullet}(a,b)$ as its hom-object for all $a,b\in S$. The composition $X_{\bullet}(a,b)\otimes_{Day} X_{\bullet}(b,c)\rightarrow X_{\bullet}(a,c)$ for $a,b,c\in S$ is induced by the canonical morphism
$$
X_{T}(a,b)\otimes X_{U}(b,c)\rightarrow X_{T\vee U}(a,c)
$$
and the identities are given by the morphism $\underline{I}\rightarrow X_{\bullet}(a,a)$ for $a\in S$ induced by the isomorphism $I\simeq X_{0}(a,a) = X_{\{0\}}(a,a)$.

This clearly extends to a functor $(-)^{nec}: \ts\mathcal{V}\rightarrow \mathcal{V}\Cat_{\nec}$.
\end{Con}

\begin{Not}\label{notation: generating necklace maps}
As in $\simp$, we distinguish some special maps in $\nec$.
\begin{itemize}
\item For any $0 < j < n$, we write
$$
\delta_{j}: \{0 < n-1\}\rightarrow \{0 < n\}
$$
for the active necklace map whose underlying morphism in $\fint$ is the coface map $\delta_{j}: [n-1]\rightarrow [n]$ , i.e. $\delta_{j}(i) = i$ if $i < j$ and $\delta_{j}(i) = j+1$ if $j\geq i$.
\item For any $k,l > 0$, we write
$$
\nu_{k,l}: \{0 < k < k+l\}\rightarrow \{0 < k+l\}
$$
for the unique inert necklace map.
\end{itemize}
\end{Not}

\begin{Con}\label{construction: temp. obj. associated to necklace cat.}
Let $\mathcal{C}$ be a necklace category with set of objects $S$. We construct a templicial object $(\mathcal{C}^{temp},S)$ as follows. For every neckace $T$, we have a $\mathcal{V}$-enriched quiver $\mathcal{C}_{T} = \left(\mathcal{C}_{T}(a,b)\right)_{a,b\in S}$ Then the composition and identities of $\mathcal{C}$ induce quiver morphisms
$$
m_{U,V}: \mathcal{C}_{U}\otimes_{S} \mathcal{C}_{V}\rightarrow \mathcal{C}_{U\vee V}\quad \text{and}\quad u: I_{S}\rightarrow \mathcal{C}_{\{0\}}
$$
for all necklaces $U$ and $V$. Set $\mathcal{C}^{temp}_{0} = I_{S}$ and $p_{0} = u: \mathcal{C}^{temp}_{0}\rightarrow \mathcal{C}_{\{0\}}$. Now let $n > 0$. We inductively define an object $\mathcal{C}^{temp}_{n}\in \mathcal{V}\Quiv_{S}$ along with morphisms $p_{n}$ and $\mu_{k,l}$ as the limit of the following diagram of solid arrows in $\mathcal{V}\Quiv_{S}$:
\begin{equation}\label{diagram: temp. obj. associated to neck. obj. limit}
\begin{tikzcd}
	{\mathcal{C}^{temp}_{n}} && {\underset{\substack{k,l > 0\\ k + l = n}}{\prod}\mathcal{C}^{temp}_{k}\otimes_{S} \mathcal{C}^{temp}_{l}} & {\underset{\substack{r,s,t > 0\\ r + s + t = n}}{\prod}\mathcal{C}^{temp}_{r}\otimes_{S} \mathcal{C}^{temp}_{s}\otimes_{S} \mathcal{C}^{temp}_{t}} \\
	{\mathcal{C}_{\{0 < n\}}} && {\underset{\substack{k,l > 0\\ k + l = n}}{\prod}\mathcal{C}_{\{0 < k\}}\otimes_{S} \mathcal{C}_{\{0 < l\}}} \\
	&& {\underset{\substack{k,l > 0\\ k + l = n}}{\prod}\mathcal{C}_{\{0 < k < k+l\}}}
	\arrow["{p_{n}}"', dashed, from=1-1, to=2-1]
	\arrow["{(\mu_{k,l})_{k,l}}", dashed, from=1-1, to=1-3]
	\arrow["{\prod_{k,l}p_{k}\otimes p_{l}}", from=1-3, to=2-3]
	\arrow["{\prod_{k,l}m_{\{0 < k\},\{0 < l\}}}", from=2-3, to=3-3]
	\arrow["{(\mathcal{C}(\nu_{k,l}))_{k,l}}"', from=2-1, to=3-3]
	\arrow["\beta"', shift right=1, from=1-3, to=1-4]
	\arrow["\alpha", shift left=1, from=1-3, to=1-4]
\end{tikzcd}
\end{equation}
where $\alpha$ and $\beta$ are defined by
$$
\pi_{r,s,t}\alpha = (\id_{r}\otimes \mu_{s,t})\pi_{r,s+t}\quad \text{and}\quad \pi_{r,s,t}\beta = (\mu_{r,s}\otimes \id_{t})\pi_{r+s,t}
$$
For example, $\mathcal{C}^{temp}_{1} = \mathcal{C}_{\{0 < 1\}}$ with $p_{1} = \id_{\mathcal{C}_{\{0,1\}}}$, and $\mathcal{C}^{temp}_{2}$ is the pullback of $m_{\{0 < 1\},\{0 < 1\}}$ and $\mathcal{C}(\nu_{1,1})$. We further set $\mu_{0,n}$ and $\mu_{n,0}$ to be the left and right unit isomorphisms respectively:
$$
\mathcal{C}^{temp}_{n}\xrightarrow{\sim} \mathcal{C}^{temp}_{0}\otimes_{S} \mathcal{C}^{temp}_{n},\quad \mathcal{C}^{temp}_{n}\xrightarrow{\sim} \mathcal{C}^{temp}_{n}\otimes_{S} \mathcal{C}^{temp}_{0}
$$

Further, let $f: [m]\rightarrow [n]$ be a morphism in $\fint$. We define a quiver morphism $\mathcal{C}^{temp}(f): \mathcal{C}^{temp}_{n}\rightarrow \mathcal{C}^{temp}_{m}$ by induction on $m$. Set $\mathcal{C}^{temp}(\id_{[0]})$ to be the identity on $I_S$. If $m > 0$, we let $\mathcal{C}^{temp}(f)$ be the unique morphism satisfying, for all $k,l > 0$ with $k + l = m$:
$$
\mu_{k,l}\mathcal{C}^{temp}(f) = (\mathcal{C}^{temp}(f_{1})\otimes_{S} \mathcal{C}^{temp}(f_{2}))\mu_{p,q}
$$
and 
$$
p_{m}\mathcal{C}^{temp}(f) = \mathcal{C}(f)p_{n}
$$
where $f_{1}: [k]\rightarrow [p]$ and $f_{2}: [l]\rightarrow [q]$ are unique in $\fint$ such that $f_{1} + f_{2} = f$. (Note that in case $m = 1$, the first condition is empty and $\mathcal{C}^{temp}(f)$ is just $\mathcal{C}(f)p_{n}$.)

We have thus constructed a well-defined functor
$$
\mathcal{C}^{temp}: \fint^{op}\rightarrow \mathcal{V}\Quiv_{S}
$$
By construction, $\mathcal{C}^{temp}$ is strongly unital and colax monoidal with comultiplication given by the morphisms $(\mu_{k,l})_{k,l\geq 0}$.
\end{Con}

\begin{Thm}\label{theorem: nec functor fully faithful left-adjoint}
The functor $(-)^{nec}: \ts\mathcal{V}\rightarrow \mathcal{V}\Cat_{\nec}$ is fully faithful and left-adjoint to a functor $(-)^{temp}: \mathcal{V}\Cat_{\nec}\rightarrow \ts\mathcal{V}$ which is given on objects by the assignment $\mathcal{C}\mapsto \mathcal{C}^{temp}$ of Construction \ref{construction: temp. obj. associated to necklace cat.}.
\end{Thm}
\begin{proof}
Let $\mathcal{C}$ be a necklace category and define a necklace functor $\varepsilon_{\mathcal{C}}: (\mathcal{C}^{temp})^{nec}\rightarrow \mathcal{C}$ by the quiver morphism
$$
\varepsilon_{\mathcal{C}_{T}}: (\mathcal{C}^{temp})_{T}\xrightarrow{p_{t_{1}}\otimes ...\otimes p_{p-t_{k-1}}} \mathcal{C}_{\{0 < t_{1}\}}\otimes ...\otimes \mathcal{C}_{\{0 < p-t_{k-1}\}}\xrightarrow{m_{\mathcal{C}}} \mathcal{C}_{T}
$$
for any necklace $T = \{0 = t_{0} < t_{1} < ... < t_{k} = p\}$.

Let $(X,S)$ be a templicial object and $H: X^{nec}\rightarrow \mathcal{C}$ a necklace functor with object map $f: S\rightarrow \Ob(\mathcal{C})$. We will construct a templicial morphism $(\alpha,f): (X,S)\rightarrow (\mathcal{C}^{temp},\Ob(\mathcal{C}))$ by induction. Let $\alpha_{0}$ be the canonical quiver morphism $f_{!}(X_{0})\simeq f_{!}(I_{S})\rightarrow  I_{\Ob(\mathcal{C})} = \mathcal{C}^{temp}_{0}$. For $n > 0$, let $\beta_{n}: f_{!}(X_{n}) = f_{!}(X_{\{0 < n\}})\rightarrow \mathcal{C}_{\{0 < n\}}$ be adjoint to $H_{\{0 < n\}}: X_{\{0 < n\}}\rightarrow f^{*}(\mathcal{C}_{\{0 < n\}})$. By Construction \ref{construction: temp. obj. associated to necklace cat.}, we have a unique morphism $\alpha_{n}: f_{!}(X_{n})\rightarrow \mathcal{C}^{temp}_{n}$ such that
$$
p_{n}\alpha_{n} = \beta_{n}
$$
and for all $k,l > 0$ with $k + l = n$, $\mu_{k,l}\circ\alpha_{n}$ is equal to the composite
$$
f_{!}(X_{n})\xrightarrow{f_{!}(\mu^{X}_{k,l})}f_{!}(X_{k}\otimes X_{l})\rightarrow f_{!}(X_{k})\otimes f_{!}(X_{l})\xrightarrow{\alpha_{k}\otimes \alpha_{l}} \mathcal{C}^{temp}_{k}\otimes \mathcal{C}^{temp}_{l}
$$
where we used to colax monoidal structure of $f_{!}$. It follows that $(\alpha,f)$ is a well-defined templicial morphism which is unique such that $\varepsilon_{\mathcal{C}}\circ \alpha^{nec} = H$. Hence, the assignment $\mathcal{C}\mapsto \mathcal{C}^{temp}$ extends to a right-adjoint of $(-)^{nec}$.

Now let $X$ be a templicial object. Since the composition $X^{nec}_{T}\otimes_{S} X^{nec}_{U}\rightarrow X^{nec}_{T\vee U}$ of $X^{nec}$ is an isomorphism for all $T,U\in \nec$, it follows from Construction \ref{construction: temp. obj. associated to necklace cat.} that $p_{n}: (X^{nec})^{temp}_{n}\rightarrow X_{\{0 < n\}} = X_{n}$ is an isomorphism. We thus obtain a natural isomorphism $(-)^{temp}\circ (-)^{nec}\simeq \id_{\ts\mathcal{V}}$ which shows that $(-)^{nec}$ is fully faithful.
\end{proof}

\subsection{Some constructions revisited}\label{subsection: Some constructions revisited}

We show that the functor $\tilde{U}: \ts\mathcal{V}\rightarrow \SSet$ (Proposition \ref{proposition: free-forget adjunction}) and the templicial nerve $N_{\mathcal{V}}: \mathcal{V}\Cat\rightarrow \ts\mathcal{V}$ (Definition \ref{definition: templicial nerve}) factor through the category $\mathcal{V}\Cat_{\nec}$ of necklace categories.

\begin{Not}
By post-composition, the adjunction $F: \Set\leftrightarrows \mathcal{V}: U$ induces an adjunction $F: \Set^{\nec^{op}}\leftrightarrows \mathcal{V}^{\nec^{op}}: U$. Note that as $F$ is strong monoidal and preserves colimits, the induced functor $F: \Set^{\nec^{op}}\rightarrow \mathcal{V}^{\nec^{op}}$ is strong monoidal as well. Therefore, we have an induced adjunction which we denote by
\[\begin{tikzcd}
	{\Cat_{\nec}} & {\mathcal{V}\Cat_{\nec}}
	\arrow[""{name=0, anchor=center, inner sep=0}, "{\mathcal{F}}", shift left=2, from=1-1, to=1-2]
	\arrow[""{name=1, anchor=center, inner sep=0}, "{\mathcal{U}}", shift left=2, from=1-2, to=1-1]
	\arrow["\dashv"{anchor=center, rotate=-90}, draw=none, from=0, to=1]
\end{tikzcd}\]
\end{Not}

\begin{Prop}\label{proposition: free and nec functors commute}
There is diagram of adjunctions
\[\begin{tikzcd}
	{\SSet} && {\ts\mathcal{V}} \\
	\\
	{\Cat_{\nec}} && {\mathcal{V}\Cat_{\nec}}
	\arrow[""{name=0, anchor=center, inner sep=0}, "{\tilde{F}}", shift left=2, from=1-1, to=1-3]
	\arrow[""{name=1, anchor=center, inner sep=0}, "(-)^{nec}"', shift right=2, hook, from=1-3, to=3-3]
	\arrow[""{name=2, anchor=center, inner sep=0}, "{\mathcal{F}}", shift left=2, from=3-1, to=3-3]
	\arrow[""{name=3, anchor=center, inner sep=0}, "(-)^{nec}"', shift right=2, hook, from=1-1, to=3-1]
	\arrow[""{name=4, anchor=center, inner sep=0}, "(-)^{temp}"', shift right=2, from=3-1, to=1-1]
	\arrow[""{name=5, anchor=center, inner sep=0}, "{\mathcal{U}}", shift left=2, from=3-3, to=3-1]
	\arrow[""{name=6, anchor=center, inner sep=0}, "{\tilde{U}}", shift left=2, from=1-3, to=1-1]
	\arrow[""{name=7, anchor=center, inner sep=0}, "(-)^{temp}"', shift right=2, from=3-3, to=1-3]
	\arrow["\dashv"{anchor=center}, draw=none, from=1, to=7]
	\arrow["\dashv"{anchor=center}, draw=none, from=3, to=4]
	\arrow["\dashv"{anchor=center, rotate=-90}, draw=none, from=0, to=6]
	\arrow["\dashv"{anchor=center, rotate=-90}, draw=none, from=2, to=5]
\end{tikzcd}\]
which commutes in the sense that we have natural isomorphisms
$$
(-)^{nec}\circ \tilde{F}\simeq \mathcal{F}\circ (-)^{nec}\quad \text{and}\quad \tilde{U}\circ (-)^{temp}\simeq (-)^{temp}\circ \mathcal{U}
$$
In particular, we have a natural isomorphism
$$
\tilde{U}\simeq (-)^{temp}\circ \mathcal{U}\circ (-)^{nec}
$$
\end{Prop}
\begin{proof}
It suffices to show the commutativity of the left-adjoints. But this immediately follows from the fact that $F: \Set\rightarrow \mathcal{V}$ is strong monoidal and preserves colimits. The final isomorphism $\tilde{U}\simeq (-)^{temp}\circ \mathcal{U}\circ (-)^{nec}$ follows from the fact that $(-)^{nec}$ is fully faithful.
\end{proof}

\begin{Lem}\label{lemma: temp. construction of constant necklace functor}
Let $\mathcal{C}$ be a finitely complete category. Let $f: A\rightarrow B$ be a morphism in $\mathcal{C}$ and $n\geq 2$. Then $A$ is the limit of the following diagram of solid arrows:
\[\begin{tikzcd}
	A & {\underset{\substack{k,l > 0\\ k + l = n}}{\prod}A} & {\underset{\substack{r,s,t > 0\\ r + s + t = n}}{\prod}A} \\
	B & {\underset{\substack{k,l > 0\\ k + l = n}}{\prod}B}
	\arrow["\alpha", shift left=1, from=1-2, to=1-3]
	\arrow["\beta"', shift right=1, from=1-2, to=1-3]
	\arrow["\Delta", dashed, from=1-1, to=1-2]
	\arrow["{\prod_{k,l}f}", from=1-2, to=2-2]
	\arrow["\Delta"', from=2-1, to=2-2]
	\arrow["f"', dashed, from=1-1, to=2-1]
\end{tikzcd}\]
where $\Delta$ is the diagonal morphism and $\alpha$ and $\beta$ are defined by
$$
\pi_{r,s,t}\alpha = \pi_{r+s,t}\quad \text{and}\quad \pi_{r,s,t}\beta = \pi_{r,s+t}
$$
for all $r,s,t > 0$ with $r + s + t = n$.
\end{Lem}
\begin{proof}
This is an easy verification.
\end{proof}

Let $\underline{(-)}: \mathcal{V}\rightarrow \mathcal{V}^{\nec^{op}}$ denote the diagonal functor associating to every object $V\in \mathcal{V}$ the constant functor on $V$. Then $\underline{(-)}$ is easily seen to be strong monoidal and thus it induces a functor
$$
\underline{(-)}: \mathcal{V}\Cat\rightarrow \mathcal{V}\Cat_{\nec}
$$

\begin{Prop}\label{proposition: nerve is temp functor of constant functor}
We have a natural isomorphism
$$
N_{\mathcal{V}}\simeq (-)^{temp}\circ \underline{(-)}
$$
\end{Prop}
\begin{proof}
Let $\mathcal{C}$ be a small $\mathcal{V}$-enriched category and $n\geq 0$. Applying Lemma \ref{lemma: temp. construction of constant necklace functor} to the $n$-fold composition $m^{(n)}_{\mathcal{C}}: \mathcal{C}^{\otimes n}\rightarrow \mathcal{C}$ as a morphism in $\mathcal{V}\Quiv_{\Ob(\mathcal{C})}$, it follows from Construction \ref{construction: temp. obj. associated to necklace cat.} that $\underline{\mathcal{C}}^{temp}_{n}\simeq \mathcal{C}^{\otimes n}$ by induction on $n$. It quickly follows that this identification induces an isomorphism of templicial objects $\underline{\mathcal{C}}^{temp}\simeq N_{\mathcal{V}}(\mathcal{C})$, which is clearly natural in $\mathcal{C}$.
\end{proof}

\section{Enriching the homotopy coherent nerve}\label{section: Enriching the homotopy coherent nerve}

Let $\Cat_{\Delta}$ denote the category of small simplicial categories, that is categories enriched in the cartesian monoidal category of simplicial sets $(\SSet,\times,\Delta^{0})$. In this section we generalise the classical adjunction between the categorification functor $\mathfrak{C}: \SSet\rightarrow \Cat_{\Delta}$ and the homotopy coherent nerve $N^{hc}: \Cat_{\Delta}\rightarrow \SSet$ to the templicial level, yielding an adjunction $\mathfrak{C}_{\mathcal{V}}\dashv N^{hc}_{\mathcal{V}}$ which depends on $\mathcal{V}$ (Definition \ref{definition: temp. hc-nerve}). We first recall in \S\ref{subsection: The classical homotopy coherent nerve} the homotopy coherent nerve due to Cordier \cite{cordier1982sur}. It is most easily constructed as the formal right-adjoint to the categorification functor $\mathfrak{C}$, which has historically gone through several different equivalent descriptions. This goes back to Cordier and Porter \cite{cordier1986Vogt} and a different definition is given in \cite{lurie2009higher} which we outline below. Later, Dugger and Spivak \cite{dugger2011rigidification} gave a very elegant and simple description of $\mathfrak{C}$ by means of necklaces. We will closely follow their approach and adapt it to the templicial setting in \S\ref{subsection: Simplification of the categorification functor}.

\subsection{The classical homotopy coherent nerve}\label{subsection: The classical homotopy coherent nerve}

We recall Cordier's homotopy coherent nerve. Further, we give a new expression for its left-adjoint (Proposition \ref{proposition: categorification as weighted colimit}) which will allow us to generalise it more easily to the templicial setting in \S\ref{subsection: The templicial homotopy coherent nerve}.

\begin{Not}
Given a necklace $(T,p)$, consider the poset
$$
\mathcal{P}_{T} = \{U\subseteq [p]\mid T\subseteq U\}
$$
ordered by inclusion. Equivalently, it is the poset of inert necklace maps $U\hookrightarrow T$.

If $T = \{0 < p\}$ is a simplex, we also write $\mathcal{P}_{T} = \mathcal{P}_{p}$.
\end{Not}

\begin{Rem}\label{remark: poset of inert inclusions strong monoidal}
It is easy to see that the assignment $T\mapsto \mathcal{P}_{T}$ extends to a strong monoidal functor
$$
\mathcal{P}: \nec\rightarrow \Cat
$$
where for every necklace map $f: T\rightarrow U$, $\mathcal{P}(f)(V) = f(V)$ for all $V\in \mathcal{P}_{T}$. For necklaces $T$ and $U$, the monoidal structure is given by
$$
\mathcal{P}_{T}\times \mathcal{P}_{U}\rightarrow \mathcal{P}_{T\vee U}: (V,W)\mapsto (V\vee W)
$$
which is clearly an order isomorphism.
\end{Rem} 

In \cite[\S 1.1.5]{lurie2009higher}, a simplicial category $\mathfrak{C}[\Delta^{n}]$ is constructed as follows. Its objects are given by the set $[n]$ and for all $i,j\in [n]$, we have
$$
\mathfrak{C}[\Delta^{n}](i,j) =
\begin{cases}
N(\mathcal{P}_{j-i}) & \text{if }i\leq j\\
\emptyset & i > j
\end{cases}
$$
Note that $N(\mathcal{P}_{j-i})\simeq (\Delta^{1})^{\times j-i-1}$ if $i < j$ and $N(\mathcal{P}_{j-i})\simeq \Delta^{0}$ if $i = j$. Further, given $i\leq j\leq k$ in $[n]$, the composition
$$
m_{i,j,k}: \mathfrak{C}[\Delta^{n}](i,j)\times \mathfrak{C}[\Delta^{n}](j,k)\rightarrow \mathfrak{C}[\Delta^{n}](i,k)
$$
is given by applying $N$ to the order morphism
$$
\mathcal{P}_{j-i}\times \mathcal{P}_{k-j}\simeq \mathcal{P}_{\{0 < j-i < k-i\}}\hookrightarrow \mathcal{P}_{k-i}: (T,U)\mapsto T\vee U
$$
Finally, the identities are given by the unique vertex of $\mathfrak{C}[\Delta^{n}](i,i)\simeq \Delta^{0}$ for $i\in [n]$.

It is now easy to see that the above construction extends to a functor
$$
\mathfrak{C}[\Delta^{(-)}]: \simp\rightarrow \Cat_{\Delta}
$$
Then by left Kan extension along the Yoneda embedding $\simp\hookrightarrow \SSet$, the cosimplicial object $\mathfrak{C}[\Delta^{(-)}]$ induces an adjunction
\[\begin{tikzcd}
	\SSet & {\Cat_{\Delta}}
	\arrow[""{name=0, anchor=center, inner sep=0}, "{\mathfrak{C}}", shift left=2, from=1-1, to=1-2]
	\arrow[""{name=1, anchor=center, inner sep=0}, "{N^{hc}}", shift left=2, from=1-2, to=1-1]
	\arrow["\dashv"{anchor=center, rotate=-90}, draw=none, from=0, to=1]
\end{tikzcd}\]
For all small simplicial categories $\mathcal{C}$ and $n\geq 0$ we have that
$$
N^{hc}(\mathcal{C})_{n}\simeq \Cat_{\Delta}(\mathfrak{C}[\Delta^{n}],\mathcal{C})
$$

\begin{Ex}\label{example: homotopy coherent nerve in low dimensions}
Given a small simplicial category $\mathcal{C}$, let us describe its homotopy coherent nerve in low dimensions.
\begin{itemize}
\item The vertices of $N^{hc}(\mathcal{C})$ are given by the set of objects $\Ob(\mathcal{C})$.
\item The edges of $N^{hc}(\mathcal{C})$ are given by the morphisms of $\mathcal{C}$ (that is vertices $f\in \mathcal{C}_{0}(A,B)$ for some $A,B\in \Ob(\mathcal{C})$).
\item A $2$-simplex of $N^{hc}(\mathcal{C})$ is given by a (not necessarily commutative) diagram of morphisms in $\mathcal{C}$:
\[\begin{tikzcd}
	& C \\
	A && B
	\arrow["f", from=2-1, to=1-2]
	\arrow["g", from=1-2, to=2-3]
	\arrow[""{name=0, anchor=center, inner sep=0}, "h"', from=2-1, to=2-3]
	\arrow["\sigma"', shorten <=3pt, shorten >=3pt, Rightarrow, from=0, to=1-2]
\end{tikzcd}\]
along with an edge $\sigma$ of $\mathcal{C}(A,B)$ from $h$ to the composition $g\circ f$.
\end{itemize}
\end{Ex}

\begin{Prop}[\cite{dugger2011rigidification}, Proposition 3.7]\label{proposition: categorification of a necklace}
There is an isomorphism of simplicial sets
$$
\mathfrak{C}[T](0,p)\simeq N(\mathcal{P}_{T})
$$
that is natural in all necklaces $(T,p)\in \nec$.
\end{Prop}

\begin{Prop}[\cite{dugger2011rigidification}, Proposition 4.3]\label{proposition: classical categorification}
For every simplicial set $K$ with vertices $a$ and $b$, there is an isomorphism of simplicial sets
$$
\mathfrak{C}[K](a,b)\simeq \colim_{\substack{T\rightarrow K_{a,b}\text{ in }\SSet_{*,*}\\ (T,p)\in \nec}}\mathfrak{C}[T](0,p) 
$$
\end{Prop}

Recall by Proposition \ref{proposition: free-forget adjunction} that we may view a simplicial set $K$ as a templicial set and thus for any $a,b\in K_{0}$, we can apply Construction \ref{construction: necklace cat. associated to temp. obj.} to obtain a functor $K_{\bullet}(a,b): \nec^{op}\rightarrow \Set: T\mapsto K_{T}(a,b)$. It follows from Yoneda's lemma that we have a canonical bijection, natural in $T\in \nec$:
$$
\SSet_{*,*}(T,K_{a,b})\simeq K_{T}(a,b)
$$
With this in mind, we introduce another description of $\mathfrak{C}$ by means of a weighted colimit. For background on weighted colimits, we refer to the literature (see \cite[Definition 7.4.1]{riehl2014categorical} for example).

\begin{Prop}\label{proposition: categorification as weighted colimit}
For any simplicial set $K$ with vertices $a$ and $b$, $\mathfrak{C}[K](a,b)$ is isomorphic to the weighted colimit in $\SSet$:
$$
{\colim}^{K_{\bullet}(a,b)}N\mathcal{P}_{(-)}
$$
of $N\mathcal{P}_{(-)}: \nec\rightarrow \SSet$ with weight $K_{\bullet}(a,b): \nec^{op}\rightarrow \Set$.
\end{Prop}
\begin{proof}
From Propositions \ref{proposition: categorification of a necklace} and \ref{proposition: classical categorification}, it is clear that $\mathfrak{C}[K](a,b)$ is given by the following coequaliser in $\SSet$:
\[\begin{tikzcd}
	{\underset{\substack{T\rightarrow U\rightarrow K_{a,b}\\ T,U\in \nec}}{\coprod}N(\mathcal{P}_{T})} & {\underset{\substack{T\rightarrow K_{a,b}\\ T\in \nec}}{\coprod}N(\mathcal{P}_{T})} & {\mathfrak{C}[K](a,b)}
	\arrow["\alpha", shift left=1, from=1-1, to=1-2]
	\arrow["\beta"', shift right=1, from=1-1, to=1-2]
	\arrow[two heads, from=1-2, to=1-3]
\end{tikzcd}\]
where $\alpha$ and $\beta$ are given by respectively projecting onto $T\rightarrow K_{a,b}$ and applying $N\mathcal{P}_{(-)}$ to $T\rightarrow U$ for any $T\rightarrow U\rightarrow K_{a,b}$ in $\SSet_{*,*}$.

Since morphisms $T\rightarrow K_{a,b}$ in $\SSet_{*,*}$ with $T$ a necklace correspond to elements of the set $K_{T}(a,b)$, we obtain a coequaliser diagram
\[\begin{tikzcd}
	{\underset{\substack{T\rightarrow U\\ \text{in }\nec}}{\coprod}K_{U}(a,b)\times N(\mathcal{P}_{T})} & {\underset{T\in \nec}{\coprod}K_{T}(a,b)\times N(\mathcal{P}_{T})} & {\mathfrak{C}[K](a,b)}
	\arrow["\alpha", shift left=1, from=1-1, to=1-2]
	\arrow["\beta"', shift right=1, from=1-1, to=1-2]
	\arrow[two heads, from=1-2, to=1-3]
\end{tikzcd}\]
where $\alpha$ and $\beta$ are given by respectively applying $K_{\bullet}(a,b)$ and $N\mathcal{P}_{(-)}$ to $T\rightarrow U$ in $\nec$. But this coequaliser is precisely the weighted colimit described in the statement.
\end{proof}

\subsection{The templicial homotopy coherent nerve}\label{subsection: The templicial homotopy coherent nerve}

Consider the category $S\mathcal{V}$ of simplicial objects in $\mathcal{V}$, with the pointwise symmetric monoidal structure induced by that of $\mathcal{V}$. Note that its monoidal unit is the simplicial object $F(\Delta^{0}) = \underline{I}$ (the constant functor on $I$). Further, $S\mathcal{V}$ is canonically enriched and tensored over $\mathcal{V}$. We denote the enrichment over $\mathcal{V}$ by $[-,-]$. For every $V\in \mathcal{V}$ and $A\in S\mathcal{V}$, the tensoring $V\cdot A$ is given by the monoidal product $\underline{V}\otimes A$ where $\underline{V}$ denotes the constant functor on $V$.

We denote the category of small $S\mathcal{V}$-enriched categories by $\mathcal{V}\Cat_{\Delta}$. Note that the adjunction $F: \Set\leftrightarrows \mathcal{V}: U$ induces an adjunction $F: \SSet\leftrightarrows S\mathcal{V}: U$ by post-composition for which $F$ is still strong monoidal. Hence, we have an induced adjunction between simplicial categories and $S\mathcal{V}$-categories which we denote by
\[\begin{tikzcd}
	{\Cat_{\Delta}} & {\mathcal{V}\Cat_{\Delta}}
	\arrow[""{name=0, anchor=center, inner sep=0}, "{\mathcal{F}}", shift left=2, from=1-1, to=1-2]
	\arrow[""{name=1, anchor=center, inner sep=0}, "{\mathcal{U}}", shift left=2, from=1-2, to=1-1]
	\arrow["\dashv"{anchor=center, rotate=-90}, draw=none, from=0, to=1]
\end{tikzcd}\]

From Proposition \ref{proposition: categorification as weighted colimit} it is easy to see that the adjunction $\mathfrak{C}\dashv N^{hc}$ actually factors through the category $\Cat_{\nec}$ of Definition \ref{definition: necklace cats.}. Moreover, it suggests that we can define the templicial categorification functor by means of a similar weighted colimit.

\begin{Con}\label{construction: ajunction between simplicial and necklace objects}
We construct an adjunction
\[\begin{tikzcd}
	{\mathcal{V}^{\nec^{op}}} & {S\mathcal{V}}
	\arrow[""{name=0, anchor=center, inner sep=0}, "{\mathfrak{s}}", shift left=2, from=1-1, to=1-2]
	\arrow[""{name=1, anchor=center, inner sep=0}, "{\mathfrak{n}}", shift left=2, from=1-2, to=1-1]
	\arrow["\dashv"{anchor=center, rotate=-90}, draw=none, from=0, to=1]
\end{tikzcd}\]
as follows. Given a functor $X: \nec^{op}\rightarrow \mathcal{V}$, consider the weighted colimit in $S\mathcal{V}$:
$$
\mathfrak{s}(X) = {\colim}^{X}FN\mathcal{P}_{(-)}
$$
of the composite $\nec\xrightarrow{\mathcal{P}_{(-)}} \Cat\xrightarrow{N} \SSet\xrightarrow{F} S\mathcal{V}$ with weight $X$. Explicitly, $\mathfrak{s}(X)$ may be realised as the following coequaliser in $S\mathcal{V}$:
\begin{equation}\label{diagram: simplification coequalizer}
\begin{tikzcd}
	{\underset{\substack{f: T\rightarrow U\\ \text{in }\nec}}{\coprod}X_{U}\otimes FN(\mathcal{P}_{T})} & {\underset{T\in\nec}{\coprod}X_{T}\otimes FN(\mathcal{P}_{T})} & {\mathfrak{s}(X)}
	\arrow[two heads, from=1-2, to=1-3]
	\arrow["\beta"', shift right=1, from=1-1, to=1-2]
	\arrow["\alpha", shift left=1, from=1-1, to=1-2]
\end{tikzcd}
\end{equation}
where $\alpha$ and $\beta$ are given by respectively applying $X$ and $FN\mathcal{P}_{(-)}$ to a necklace morphism $f: T\rightarrow U$.

As a weighted colimit, $\mathfrak{s}(X)$ fits into a canonical bijection of sets
$$
S\mathcal{V}(\mathfrak{s}(X),Y)\simeq \mathcal{V}^{\nec^{op}}(X, [FN\mathcal{P}_{(-)},Y])
$$
which is natural in $Y\in S\mathcal{V}$. Hence, the assignment $X\mapsto \mathfrak{s}(X)$ extends to a functor $\mathfrak{s}: \mathcal{V}^{\nec^{op}}\rightarrow S\mathcal{V}$ which is left-adjoint to the functor
$$
\mathfrak{n}: S\mathcal{V}\rightarrow \mathcal{V}^{\nec^{op}}: Y\mapsto [FN\mathcal{P}_{(-)},Y]
$$
\end{Con}

\begin{Prop}\label{proposition: necklace-simplex adjunction is monoidal}
The functor $\mathfrak{s}: \mathcal{V}^{\nec^{op}}\rightarrow S\mathcal{V}$ of Construction \ref{construction: ajunction between simplicial and necklace objects} is strong monoidal.
\end{Prop}
\begin{proof}
For functors $X,Y: \nec^{op}\rightarrow \mathcal{V}$, we have
\begin{align*}
\mathfrak{s}(X\otimes_{Day} Y) &= {\colim_{T\in \nec}}^{(X\otimes_{Day} Y)(T)}FN\mathcal{P}_{T}\simeq {\colim_{U,V\in \nec}}^{X(U)\otimes Y(V)}FN\mathcal{P}_{U\vee V}\\
&\simeq {\colim_{U\in \nec}}^{X(U)}FN\mathcal{P}_{U}\otimes {\colim_{V\in \nec}}^{Y(V)}FN\mathcal{P_{V}} = \mathfrak{s}(X)\otimes \mathfrak{s}(Y)
\end{align*}
where we have used the presentation of $X\otimes_{Day} Y$ as a left Kan extension and the strong monoidality of $F$, $N$ and $\mathcal{P}_{(-)}$ (see Remark \ref{remark: poset of inert inclusions strong monoidal}). Further, since $\underline{I} = F(\nec(-,\{0\}))$,
$$
\mathfrak{s}(\underline{I}) = {\colim_{T\in \nec}}^{\underline{I}}FN\mathcal{P}_{T}\simeq FN\mathcal{P}_{\{0\}}\simeq F(\Delta^{0})
$$
\end{proof}

\begin{Def}\label{definition: temp. hc-nerve}
By virtue of Proposition \ref{proposition: necklace-simplex adjunction is monoidal}, the adjunction $\mathfrak{s}\dashv \mathfrak{n}$ between $\mathcal{V}^{\nec^{op}}$ and $S\mathcal{V}$ induces an adjunction
\[\begin{tikzcd}
	{\mathcal{V}\Cat_{\nec}} & {\mathcal{V}\Cat_{\Delta}}
	\arrow[""{name=0, anchor=center, inner sep=0}, "{\mathfrak{s}}", shift left=2, from=1-1, to=1-2]
	\arrow[""{name=1, anchor=center, inner sep=0}, "{\mathfrak{n}}", shift left=2, from=1-2, to=1-1]
	\arrow["\dashv"{anchor=center, rotate=-90}, draw=none, from=0, to=1]
\end{tikzcd}\]
We call the composite
$$
\mathfrak{C}_{\mathcal{V}}: \ts\mathcal{V}\xrightarrow{(-)^{nec}} {\mathcal{V}\Cat_{\nec}}\xrightarrow{\mathfrak{s}} {\mathcal{V}\Cat_{\Delta}}
$$
the \emph{categorification functor}. It is left-adjoint to the composite
$$
N^{hc}_{\mathcal{V}}: \mathcal{V}\Cat_{\Delta}\xrightarrow{\mathfrak{n}} \mathcal{V}\Cat_{\nec}\xrightarrow{(-)^{temp}} \ts\mathcal{V}
$$
which we call the \emph{templicial homotopy coherent nerve}.
\end{Def}

\begin{Rem}\label{remark: categorification and homotopy coherent nerve for sets}
Suppose $\mathcal{V} = \Set$. Then the adjunction $\mathfrak{C}_{\mathcal{V}}\dashv N^{hc}_{\mathcal{V}}$ reduces to the classical adjunction $\mathfrak{C}\dashv N^{hc}$. Indeed, it suffices to note that $\mathfrak{C}_{\mathcal{V}}$ reduces to $\mathfrak{C}$, which follows from Proposition \ref{proposition: categorification as weighted colimit} and Construction \ref{construction: ajunction between simplicial and necklace objects}.
\end{Rem}

\begin{Ex}
Let $\mathcal{C}$ be a small $S\mathcal{V}$-category. We describe the templicial object $N^{hc}_{\mathcal{V}}(\mathcal{C})$ in low dimensions using Construction \ref{construction: temp. obj. associated to necklace cat.}. Note the analogy with Example \ref{example: homotopy coherent nerve in low dimensions}.
\begin{itemize}
\item The vertex set of $N^{hc}_{\mathcal{V}}(\mathcal{C})$ is simply $\Ob(\mathcal{C})$.
\item Further for any $A,B\in \Ob(\mathcal{C})$, it follows from $N(\mathcal{P}_{\{0 < 1\}})\simeq \Delta^{0}$ that
$$
N^{hc}_{\mathcal{V}}(\mathcal{C})_{1}(A,B) = \mathfrak{n}(\mathcal{C})_{\{0 < 1\}}(A,B) = [FN(\mathcal{P}_{\{0 < 1\}}),\mathcal{C}(A,B)]\simeq \mathcal{C}_{0}(A,B)
$$
\item In dimension $2$, it follows from $N(\mathcal{P}_{\{0 < 2\}})\simeq \Delta^{1}$ and $N(\mathcal{P}_{\{0 < 1 < 2\}})\simeq \Delta^{0}$ that
\begin{align*}
\mathfrak{n}(\mathcal{C})_{\{0 < 2\}}(A,B) &= [FN(\mathcal{P}_{\{0 < 2\}}),\mathcal{C}(A,B)]\simeq \mathcal{C}_{1}(A,B)\\
\mathfrak{n}(\mathcal{C})_{\{0 < 1 < 2\}}(A,B) &= [FN(\mathcal{P}_{\{0 < 1 < 2\}}), \mathcal{C}(A,B)]\simeq \mathcal{C}_{0}(A,B)
\end{align*}
The morphism $\mathfrak{n}(\mathcal{C})_{\{0 < 2\}}(A,B)\rightarrow \mathfrak{n}(\mathcal{C})_{\{0 < 1 < 2\}}(A,B)$ is induced by the inert necklace map $\nu_{1,1}: \{0 < 1 < 2\}\hookrightarrow \{0 < 2\}$ and thus corresponds to the face map\\ $d_{0}: \mathcal{C}_{1}(A,B)\rightarrow \mathcal{C}_{0}(A,B)$. It follows from \eqref{diagram: temp. obj. associated to neck. obj. limit} that we have a pullback diagram:
\[\begin{tikzcd}
	{N^{hc}_{\mathcal{V}}(\mathcal{C})_{2}(A,B)} & {\underset{C\in \Ob(\mathcal{C})}{\coprod}\mathcal{C}_{0}(A,C)\otimes \mathcal{C}_{0}(C,B)} \\
	{\mathcal{C}_{1}(A,B)} & {\mathcal{C}_{0}(A,B)}
	\arrow[from=1-1, to=2-1]
	\arrow[from=1-1, to=1-2]
	\arrow["{m_{0,0}}", from=1-2, to=2-2]
	\arrow["{d_{0}}"', from=2-1, to=2-2]
\end{tikzcd}\]
In particular, we see that the underlying set of the object $N^{hc}_{\mathcal{V}}(\mathcal{C})_{2}(A,B)$ consists of pairs $(\sigma,\alpha)$ with $\alpha\in U(\coprod_{C\in \Ob(\mathcal{C})}\mathcal{C}_{0}(A,C)\otimes \mathcal{C}_{0}(C,B))$ and $\sigma\in U(\mathcal{C}_{1}(A,B))$ an edge from $h = d_{1}(\sigma)$ to $m(\alpha)$.
\end{itemize}
\end{Ex}

\begin{Prop}\label{proposition: free and categorification commute}
There are canonical natural isomorphisms
$$
\mathfrak{C}_{\mathcal{V}}\circ \tilde{F}\simeq \mathcal{F}\circ \mathfrak{C}\quad \text{and}\quad \tilde{U}\circ N^{hc}_{\mathcal{V}}\simeq N^{hc}\circ \mathcal{U}
$$
\end{Prop}
\begin{proof}
As $\mathfrak{C}_{\mathcal{V}}\dashv N^{hc}_{\mathcal{V}}$, $\mathfrak{C}\dashv N^{hc}$, $\tilde{F}\dashv \tilde{U}$ and $\mathcal{F}\dashv \mathcal{U}$, it suffices to only show the first natural isomorphism. Since $F: \SSet\rightarrow S\mathcal{V}$ preserves colimits and is strong monoidal, it is clear that
$$
{\colim_{T\in \nec}}^{FX_{T}}FN\mathcal{P}_{T}\simeq F\left({\colim_{T\in \nec}}^{X_{T}}N\mathcal{P}_{T}\right)
$$
for any functor $X: \nec^{op}\rightarrow \Set$. It follows that we have a natural isomorphism $F\circ \mathfrak{s}\simeq \mathfrak{s}\circ F$ of functors $\Set^{\nec^{op}}\rightarrow S\mathcal{V}$, and thus also $\mathcal{F}\circ \mathfrak{s}\simeq \mathfrak{s}\circ \mathcal{F}$ of functors $\Cat_{\nec}\rightarrow \mathcal{V}\Cat_{\Delta}$. Thus by Proposition \ref{proposition: free and nec functors commute} we have $\mathcal{F}\circ \mathfrak{C}\simeq \mathfrak{C}_{\mathcal{V}}\circ \tilde{F}$ as well.
\end{proof}

\subsection{Simplification of the categorification functor}\label{subsection: Simplification of the categorification functor}

Following \cite[\S 4]{dugger2011rigidification}, we give a simplified description of the categorification functor $\mathfrak{C}_{\mathcal{V}}: \ts\mathcal{V}\rightarrow \mathcal{V}\Cat_{\Delta}$ of Definition \ref{definition: temp. hc-nerve}. Let us first recall their approach.

Let $(T,p)$ be a necklace and $n\geq 0$. A \emph{flag of length $n$} on $T$ is defined as an $n$-simplex of the nerve $N(\mathcal{P}_{T})$. Explicitly, a flag of length $n$ on $T$ is a sequence of inclusions
$$
\vec{T} = (T_{0}\subseteq ...\subseteq T_{n})
$$
such that $T\subseteq T_{0}$ and $T_{n}\subseteq [p]$. We call a flag $\vec{T}$ on a necklace $T$ \emph{flanked} if $T = T_{0}$ and $T_{n} = [p]$.

Given a simplicial set $K$ with vertices $a$ and $b$, and $T = \Delta^{n_{1}}\vee ...\vee \Delta^{n_{k}}$ a necklace. A map $T\rightarrow K_{a,b}$ in $\SSet_{*,*}$ is \emph{totally non-degenerate} if for every $i\in \{1,...,k\}$, the composite map in $\SSet$
$$
\Delta^{n_{i}}\hookrightarrow T\rightarrow K_{a,b}
$$
represents a non-degenerate $n_{i}$-simplex of $K$.

As an immediate consequence of Proposition \ref{proposition: classical categorification}, an $n$-simplex of $\mathfrak{C}[K](a,b)$ consists of an equivalence class
\begin{equation}\label{equation: n-simplex of categorification}
[T,T\rightarrow K_{a,b}, \vec{T}]
\end{equation}
of triples $(T,T\rightarrow K_{a,b}, \vec{T})$ where
\begin{itemize}
\item $T$ is a necklace,
\item $T\rightarrow K_{a,b}$ is a map in $\SSet_{*,*}$ (equivalently, an element of $K_{T}(a,b)$),
\item $\vec{T}$ is a flag of length $n$ on $T$.
\end{itemize}
The equivalence relation is generated by setting two triples $(T,T\rightarrow K_{a,b},\vec{T})$ and $(U,U\rightarrow K_{a,b},\vec{U})$ to be equivalent if there exists a map of necklaces $f: T\rightarrow U$ making the obvious diagram commute, such that $f(T_{i}) = U_{i}$ for all $0\leq i\leq n$.

Then one can make the following reductions:
\begin{enumerate}
\item In every equivalence class \eqref{equation: n-simplex of categorification} there exists a triple $(T,T\rightarrow K_{a,b},\vec{T})$ such that $\vec{T}$ is flanked. Moreover, two
such triples are equivalent if and only if they can
be connected by a zig-zag of morphisms of flagged necklaces in which every flag is flanked.
\item In every equivalence class \eqref{equation: n-simplex of categorification} there exists a \emph{unique} triple $(T,T\rightarrow K_{a,b},\vec{T})$ such that $\vec{T}$ is flanked and $T\rightarrow K_{a,b}$ is totally non-degenerate. In other words, there is a bijection
$$
\mathfrak{C}[K]_{n}(a,b)\simeq \coprod_{\substack{T\in \nec\\ \vec{T}\text{ flanked flag}\\ \text{ of length }n}}\!\!\!K^{nd}_{T}(a,b)
$$
where $K^{nd}_{T}(a,b)\subseteq K_{T}(a,b)$ is the subset of totally non-degenerate maps $T\rightarrow K_{a,b}$.
\end{enumerate}
The main ingredient in the first reduction is flankification. Given a necklace $T$ with a flag $\vec{T} = (T_{0}\subseteq ...\subseteq T_{n})$, there is a unique order isomorphism $T_{n}\simeq [k]$ where $k$ is the number of beads of $T_{n}$. For all $i\in [n]$, write $T'_{i}$ for the image of $T_{i}$ under this isomorphism so that $T'_{0}\subseteq ... \subseteq T'_{n} = [k]$. Further set $T' = T'_{0}$ so that the flag $\vec{T}'$ is flanked on $T'$. Then $(T',\vec{T}')$ is the \emph{flankification} of $(T,\vec{T})$.\\

We now proceed to adapting these steps to the templicial setting. Generalising the first of the above reductions to templicial objects is fairly straightforward. This is done in Proposition \ref{proposition: reduction to flanked flags}. For the second reduction we have to restrict to templicial objects that have non-degenerate simplices (Definition \ref{definition: latching obj. of a temp. obj.}). Our definition of $\mathfrak{C}_{\mathcal{V}}[X]$ has the advantage that there is no reference to $X$ in the indexing category of the colimit involved, which allows for more categorical and shorter proofs of the reduction steps.

\begin{Not}
Given an integer $n\geq 0$, let us denote
$$
\fnec[n]
$$
for the category of pairs $(T,\vec{T})$ where $T$ is a necklace and $\vec{T} = (T_{0},...,T_{n})$ is a flag of length $n$ on $T$. A morphism $(T,\vec{T})\rightarrow (U,\vec{U})$ in $\fnec[n]$ is a necklace map $f: T\rightarrow U$ such that $f(T_{i}) = U_{i}$ for all $i\in [n]$. Further, we let
$$
\ffnec[n]
$$
denote the full subcategory of $\fnec[n]$ spanned by flagged necklaces whose flags are flanked. Note that a morphism in $\ffnec[n]$ is necessarily active and surjective on vertices.
\end{Not}

\begin{Lem}\label{lemma: inclusion of ffnec in fnec}
Let $n\geq 0$. Flankification extends to a functor $\fnec[n]\hookrightarrow \ffnec[n]$ which is right-adjoint to the inclusion $\iota: \ffnec[n]\hookrightarrow \fnec[n]$.
\end{Lem}
\begin{proof}
Denote $\gamma(T,\vec{T})$ for the flankification of a flagged necklace $(T,\vec{T})$. If $k$ is the number of beads of $T$, we obtain a morphism $\epsilon: \iota\gamma(T,\vec{T})\rightarrow (T,\vec{T})$ in $\fnec[n]$ with underlying morphism $[k]\simeq T_{n}\hookrightarrow [p]$ in $\fint$.
Given $(U,\vec{U})\in \ffnec[n]$ with $(U,q)$ a necklace, and a morphism $f: \iota(U,\vec{U})\rightarrow (T,\vec{T})$ in $\fnec[n]$, we have in particular that $T_{n} = f(U_{n}) = f([q])$. So the morphism $f: [q]\rightarrow [p]$ in $\fint$ factors uniquely as some $g: [q]\rightarrow [k]$ followed by $[k]\hookrightarrow [p]$. Moreover, $g$ defines a morphism $(U,\vec{U})\rightarrow \gamma(T,\vec{T})$ in $\ffnec[n]$ such that $\epsilon\circ \iota(g)$. This completes the proof.
\end{proof}

\begin{Prop}\label{proposition: reduction to flanked flags}
Let $(X,S)$ be a templicial object and $a,b\in S$. Then for every $n\geq 0$, we have a canonical isomorphism
$$
\mathfrak{C}_{\mathcal{V}}[X]_{n}(a,b)\simeq \colim_{(T,\vec{T})\in\ffnec[n]}X_{T}(a,b)
$$
\end{Prop}
\begin{proof}
In view of \eqref{diagram: simplification coequalizer}, we have a coequaliser for every integer $n\geq 0$:
\[\begin{tikzcd}
	{\underset{\substack{f: T\rightarrow U\\ \vec{T}\text{ flag on }T\\ \text{of length }n}}{\coprod}X_{U}(a,b)} & {\underset{\substack{T\in \nec\\ \vec{T}\text{ flag on }T\\ \text{of length }n}}{\coprod}X_{T}(a,b)} & {\mathfrak{C}_{\mathcal{V}}[X]_{n}(a,b)}
	\arrow["\alpha", shift left=1, from=1-1, to=1-2]
	\arrow["\beta"', shift right=1, from=1-1, to=1-2]
	\arrow[two heads, from=1-2, to=1-3]
\end{tikzcd}\]
where $\alpha$ is given by $X(f)$ and $\beta$ is given by applying $f$ to $\vec{T}$, for a necklace morphism $f: T\rightarrow U$. We thus have a canonical isomorphism
$$
\mathfrak{C}_{\mathcal{V}}[X]_{n}(a,b)\simeq \colim_{(T,\vec{T})\in\fnec[n]}X_{T}(a,b)
$$
Now as the inclusion $\ffnec[n]\hookrightarrow \fnec[n]$ is a left adjoint by Lemma \ref{lemma: inclusion of ffnec in fnec}, the corresponding functor between opposite categories is a right adjoint and thus a final functor. Hence, the result follows.
\end{proof}

\begin{Rem}
The simplicial structure of $\mathfrak{C}_{\mathcal{V}}[X](a,b) = \colim^{X_{\bullet}(a,b)}N\mathcal{P}_{(-)}$ is given by that of $N\mathcal{P}_{T}$, i.e. by deleting and copying terms in a flag, but the simplicial structure on $\colim_{(T,\vec{T})\in\ffnec}X_{T}(a,b)$ is slightly more difficult. The degeneracy maps and inner face maps are still given by respectively copying and deleting terms in the flags. The outer face maps however are given by first deleting the term $T_{0}$ or $T_{n}$ from a flag $(T_{0},...,T_{n})$ and then applying the flankification functor.
\end{Rem}

Let $(X,S)$ be a templicial object with non-degenerate simplices and $T$ a necklace which we write as $\{0 = t_{0} < t_{1} < t_{2} < ... < t_{k} = p\}$. Then we denote
$$
X^{nd}_{T} = X^{nd}_{t_{1}}\otimes_{S} X^{nd}_{t_{2}-t_{1}}\otimes_{S} ... \otimes_{S} X^{nd}_{p-t_{k-1}}\in \mathcal{V}\Quiv_{S}
$$
where $X^{nd}_{n}$ denotes the quiver of non-degenerate simplices of Definition \ref{definition: latching obj. of a temp. obj.}.

\begin{Prop}\label{proposition: reduction to totally non-degenerate necklaces}
Let $(X,S)$ be a templicial object with non-degenerate simplices. For all $n\geq 0$ and $a,b\in S$, we have an isomorphism in $\mathcal{V}$:
$$
\mathfrak{C}_{\mathcal{V}}[X]_{n}(a,b)\simeq \coprod_{\substack{T\in \nec\\ \vec{T}\text{ flanked flag}\\ \text{ of length }n}}\!\!\!X^{nd}_{T}(a,b)
$$
\end{Prop}
\begin{proof}
By Proposition \ref{proposition: reduction to flanked flags} and Lemma \ref{lemma: temp. Eilenberg-Zilber}, we have an isomorphism
$$
\mathfrak{C}_{\mathcal{V}}[X]_{n}(a,b)\simeq \colim_{(T,\vec{T})\in\ffnec[n]}\coprod_{\substack{f_{i}: [t_{i}-t_{i-1}]\twoheadrightarrow [n_{i}]\\ i\in \{1,...,k\}}}\left(X^{nd}_{n_{1}}\otimes_{S} ...\otimes_{S} X^{nd}_{n_{k}}\right)(a,b)
$$
where we've written $T = \{0 = t_{0} < t_{1} < ... < t_{k} = p\}$ for any $(T,\vec{T})\in \ffnec[n]$. Now let $f: (T,p)\rightarrow (U,q)$ be an active necklace map such that its underlying morphism $f: [p]\rightarrow [q]$ in $\fint$ is surjective. We can uniquely decompose $f = f_{1} + ... + f_{k}$ with $f_{i}: [t_{i}-t_{i-1}]\twoheadrightarrow [n_{i}]$ in $\simp_{surj}$ for all $i\in \{1,...,n\}$. Moreover, given a flag $\vec{T}$ of length $n$ on $T$, there is a unique flanked flag $\vec{U} = (U_{0},...,U_{n})$ on $U$ such that $f: T\rightarrow U$ lifts to a morphism $f: (T,\vec{T})\rightarrow (U,\vec{U})$ in $\ffnec[n]$ (simply set $U_{i} = f(T_{i})$). It follows that
\begin{align*}
\mathfrak{C}_{\mathcal{V}}[X]_{n}(a,b)&\simeq \colim_{(T,\vec{T})\in\ffnec[n]}\coprod_{\substack{(T,\vec{T})\rightarrow (U,\vec{U})\\ \text{ in }\ffnec[n]}}X^{nd}_{U}(a,b)\\
&\simeq \coprod_{(U,\vec{U})\in \ffnec[n]}\colim_{\substack{(T,\vec{T})\rightarrow (U,\vec{U})\\ \text{in }\ffnec[n]}}X^{nd}_{U}(a,b)\simeq \coprod_{(U,\vec{U})\in \ffnec[n]}X^{nd}_{U}(a,b)
\end{align*}
The last isomorphism is obtained by noting that the colimit on the left hand side is indexed over the category $\left((\ffnec[n])_{/(U,\vec{U})}\right)^{op}$, which is connected, and the functor involved is constant on $X^{nd}_{U}(a,b)$.
\end{proof}

\subsection{Comparison with the templicial nerve}\label{subsection: Comparison with the templicial nerve}

Analogous to the classical homotopy coherent nerve, we show that the templicial homotopy coherent nerve $N^{hc}_{\mathcal{V}}$ restricts to the templicial nerve $N_{\mathcal{V}}$ (see \S\ref{subsection: The templicial nerve}) when applied to ordinary $\mathcal{V}$-enriched categories.

Consider the $\mathcal{V}$-enriched left-adjoint $\pi_{0}: S\mathcal{V}\rightarrow \mathcal{V}$ to the functor $\underline{(-)}: \mathcal{V}\rightarrow S\mathcal{V}$ sending every object $V\in \mathcal{V}$ to the constant functor on $V$. Then for any $Y\in S\mathcal{V}$, we have a reflexive coequaliser:
\begin{equation}\label{diagram: connected components coequalizer}
\begin{tikzcd}
	{Y_{1}} & {Y_{0}} & {\pi_{0}(Y)}
	\arrow["{d_{0}}", shift left=2, from=1-1, to=1-2]
	\arrow["{d_{1}}"', shift right=2, from=1-1, to=1-2]
	\arrow["{s_{0}}"{description}, from=1-2, to=1-1]
	\arrow[two heads, from=1-2, to=1-3]
\end{tikzcd}
\end{equation}
For example if $\mathcal{V} = \Set$, then $\pi_{0}$ is the functor taking the set of connected components of a simplicial set.

As the monoidal product $\otimes$ of $\mathcal{V}$ preserves colimits in each variable, it follows that $\pi_{0}$ is strong monoidal and thus we have an induced adjunction
\[\begin{tikzcd}
	{\mathcal{V}\Cat_{\Delta}} & {\mathcal{V}\Cat}
	\arrow[""{name=0, anchor=center, inner sep=0}, "{\pi_{0}}", shift left=2, from=1-1, to=1-2]
	\arrow[""{name=1, anchor=center, inner sep=0}, "{\underline{(-)}}", shift left=2, from=1-2, to=1-1]
	\arrow["\dashv"{anchor=center, rotate=-90}, draw=none, from=0, to=1]
\end{tikzcd}\]

\begin{Prop}\label{proposition: homotopy coherent nerve of discrete simplicial cat.}
We have a natural isomorphism
$$
N^{hc}_{\mathcal{V}}\circ \underline{(-)}\simeq N_{\mathcal{V}}
$$
\end{Prop}
\begin{proof}
By Proposition \ref{proposition: nerve is temp functor of constant functor}, we have $N_{\mathcal{V}}\simeq (-)^{temp}\circ \underline{(-)}$ for $\underline{(-)}: \mathcal{V}\Cat\rightarrow \mathcal{V}\Cat_{\nec}$. Thus it suffices to show that we have an isomorphism $\mathfrak{n}\circ \underline{(-)}\simeq \underline{(-)}$ of functors $\mathcal{V}\rightarrow \mathcal{V}^{\nec^{op}}$. Take an object $A\in \mathcal{V}$ and denote $[-,-]$ for the internal hom of $\mathcal{V}$. Since the simplicial set $N(\mathcal{P}_{T})$ clearly only has one connected component, it follows from the fact that $F$ preserves colimits that
$$
[FN\mathcal{P}_{T},\underline{A}]\simeq [\pi_{0}FN\mathcal{P}_{T},A]\simeq [F(\pi_{0}N\mathcal{P}_{T}),A]\simeq [F(\{*\}),A]\simeq A
$$
for all necklaces $T$. It follows that $\mathfrak{n}(\underline{A})$ is isomorphic to the constant functor on $A$. Clearly, this isomorphism is natural in $A$ as desired.
\end{proof}

\begin{Def}\label{definition: homotopy cat. functor}
It immediately follows from Proposition \ref{proposition: homotopy coherent nerve of discrete simplicial cat.} that the templicial nerve $N_{\mathcal{V}}$ has a left-adjoint given by the composite
$$
h_{\mathcal{V}} = \pi_{0}\circ \mathfrak{C}_{\mathcal{V}}: \ts\mathcal{V}\rightarrow \mathcal{V}\Cat
$$
which we call the \emph{homotopy category functor}.
\end{Def}

\begin{Rem}\label{remark: temp. homotopy cat. generalizes classical homotopy cat.}
Note that $h_{\mathcal{V}}$ necessarily recovers the classical homotopy category functor $h: \SSet\rightarrow \Cat$ when $\mathcal{V} = \Set$ by Remark \ref{remark: temp. nerve generalizes classical nerve}.
\end{Rem}

\begin{Cor}
Let $(X,S)$ be a templicial object with $a,b\in S$. Then we have a reflexive coequaliser:
\[\begin{tikzcd}
	{\underset{\substack{T\in \nec\\ T\neq \{0\}}}{\coprod}X_{T}(a,b)} & {\underset{p > 0}{\coprod}X^{\otimes p}_{1}}(a,b) & {h_{\mathcal{V}}X(a,b)}
	\arrow["\beta"', shift right=1, from=1-1, to=1-2]
	\arrow["\alpha", shift left=1, from=1-1, to=1-2]
	\arrow[two heads, from=1-2, to=1-3]
\end{tikzcd}
\]
with $\alpha$ and $\beta$ induced by the unique active and inert necklace maps $([k],k)\rightarrow (T,p)$ and $([p],p)\rightarrow (T,p)$ of Remark \ref{remark: simplices and spines in cat. of necklaces} respectively, for any necklace $(T,p)$ with $k$ beads.
\end{Cor}
\begin{proof}
It directly follows from Proposition \ref{proposition: reduction to flanked flags} and \eqref{diagram: connected components coequalizer} that for all $a,b\in S$, we have the following (reflexive) coequaliser:
\begin{equation}\label{diagram: homotopy cat. coequalizer 1}
\begin{tikzcd}
	{\underset{\substack{T\in \nec_{-}\\ T\neq \{0\}}}{\colim}X_{T}(a,b)} & {\underset{\substack{[p]\in \simp_{surj}\\ p > 0}}{\colim}X^{\otimes p}_{1}(a,b)} & {h_{\mathcal{V}}X(a,b)}
	\arrow["\beta'"', shift right=1, from=1-1, to=1-2]
	\arrow["\alpha'", shift left=1, from=1-1, to=1-2]
	\arrow[two heads, from=1-2, to=1-3]
\end{tikzcd}
\end{equation}
where $\nec_{-}$ denotes the subcategory of $\nec$ consisting of all active necklace maps that are surjective on vertices, and $\alpha'$ and $\beta'$ are defined similarly to $\alpha$ and $\beta$. Via the epimorphism $\coprod_{T}X_{T}(a,b)\twoheadrightarrow \colim_{T}X_{T}(a,b)$, we may replace the left hand colimit by $\coprod_{T}X_{T}(a,b)$.

To show that we may also replace the right hand colimit, observe that any surjective necklace map $f: ([p],p)\twoheadrightarrow ([q],q)$ with $q > 0$ can be factored as an inert map $([p],p)\rightarrow (T,p)$ followed by some $\sigma: (T,p)\rightarrow ([q],q)$ such that $T$ has $q$ beads and the unique active map $([q],q)\hookrightarrow (T,p)$ is a section of $\sigma$. Indeed, let $\sigma: [p]\twoheadrightarrow [q]$ be the underlying morphism of $f$ in $\fint$. Then $\sigma$ has a section $\delta$ in $\fint$. Now simply set $T = \delta([p])$.
\end{proof}

Recall the commutativity results for the templicial nerve (Proposition \ref{proposition: nerve commutativity results}). While it follows immediately from Proposition \ref{proposition: free and categorification commute} that $h_{\mathcal{V}}\circ \tilde{F}\simeq \mathcal{F}\circ h$, the homotopy category functors and forgetful functors do not commute in general, as the following example shows.

\begin{Ex}\label{example: homotopy and forgetful don't commute}
Let $\mathcal{V} = \Mod(k)$ be the category of $k$-modules over with $k$ an arbitrary unital commutative ring. Let us denote $h_{k} = h_{\Mod(k)}$. Consider the templicial $k$-module $X = \tilde{F}(\partial\Delta^{2})$. Then the hom-object $(h_{k}X)(0,2)\in \Mod(k)$ is isomorphic to
$$
F(h(\partial\Delta^{2})(0,2)) = F(\{
\begin{tikzpicture}[baseline=(0)]
\filldraw
(0,0) circle (2pt) node(0){}
(0.5,0.4) circle (2pt) node(1){}
(1,0) circle (2pt) node(2){};
\draw[->] (0) -- (1);
\draw[->] (1) -- (2);
\draw
(0) node[left]{0}
(1) node[above]{1}
(2) node[right]{2};
\end{tikzpicture}
,
\begin{tikzpicture}[baseline=(base)]
\draw
(0,0) node(base){};
\filldraw
(0,0.1) circle (2pt) node(0){}
(1,0.1) circle (2pt) node(2){};
\draw[->] (0) -- (2);
\draw
(0) node[left]{0}
(2) node[right]{2};
\end{tikzpicture}
\})\simeq k\oplus k
$$
On the other hand, note that each edge in $\tilde{U}(X)$ between two given vertices is uniquely determined by an element $a_{i}\in k$. So the set $h\tilde{U}(X)(0,2)$ consists of equivalence classes of sequences of edges $(a_{1},...,a_{n})$ from $0$ to $2$ in $\tilde{U}(X)$. One can check that
$$
h\tilde{U}\tilde{F}(\partial\Delta^{2})(0,2)\simeq U(k)\amalg_{U(0)} U(k)
$$
which identifies a sequence $(a_{1},...,a_{n})$ with its product $a_{n}\cdots a_{1}$ in $k$. The two terms $U(k)$ correspond to paths either passing through the vertex $1$ or not. Now the induced map $h\tilde{U}\tilde{F}(\partial\Delta^{2})(0,2)\rightarrow U((h_{k}X)(0,2))$ on hom-sets corresponds to the canonical map
$$
U(k)\amalg_{U(0)} U(k)\rightarrow U(k\oplus k)
$$
which is certainly not a bijection if $k$ is not the zero ring. Hence, the canonical functor
$$
h\tilde{U}(X)\rightarrow \mathcal{U}(h_{k}X)
$$
is not an equivalence of categories.
\end{Ex}

In the next section, we will restrict to a special class of templicial objects which we call quasi-categories in $\mathcal{V}$. It turns out that for a quasi-category $X$ in $\mathcal{V}$, the canonical functor $h\tilde{U}(X)\rightarrow \mathcal{U}(h_{\mathcal{V}}X)$ is always an isomorphism (under suitable hypotheses on $\mathcal{V}$), see  Corollary \ref{corollary: forget commute with homotopy for fibrant temp. obj.} below.

\section{Quasi-categories in a monoidal category}\label{section: Quasi-categories in a monoidal category}

Quasi-categories are models for $(\infty,1)$-categories first introduced by Joyal \cite{joyal2002quasi} as simplicial sets satisfying the weak Kan condition in the sense of Boardman and Vogt \cite{boardmanvogt}. That is, a simplicial set $X$ is a quasi-category if every simplicial map $\Lambda^{n}_{j}\rightarrow X$ from an inner horn can be extended to a map $\Delta^{n}\rightarrow X$ from the standard simplex. In \cite{joyal2004notes}, Joyal equips $\SSet$ with a model structure in which the fibrant objects are precisely the quasi-categories. In this section, we introduce the natural analogue of quasi-categories in the templicial context (Definition \ref{definition: necklace functor lifts inner horns}). However, in view of Example \ref{example: free non-cell-cofibrant temp. obj.}, we express the lifting condition in the category $\mathcal{V}^{\nec^{op}}$, rather than $\ts\mathcal{V}$. Nonetheless, we still recover classical quasi-categories when $\mathcal{V} = \Set$ (Proposition \ref{proposition: fibrant temp. sets are quasi-cats.}). We continue in \S\ref{subsection: Nerves and quasi-categories} by showing our main result that the templicial nerve produces quasi-categories in $\mathcal{V}$ from locally Kan $S\mathcal{V}$-categories (Corollary \ref{corollary: hc-nerve preserves fibrant objects}). In \S\ref{subsection: Simplification of the homotopy category} we discuss the homotopy category of a quasi-category in $\mathcal{V}$.

\subsection{Horn filling in necklaces}\label{subsection: Horn filling in necklaces}

For integers $0\leq j\leq n$, we denote by $\Lambda^{n}_{j}$ the $j$th horn of the $n$-simplex. That is, $\Lambda^{n}_{j}$ is the union of all the faces of $\Delta^{n}$, except the $j$th face. In order to define quasi-categories in $\mathcal{V}$, we wish to consider the usual horn lifting property in the category $\mathcal{V}^{\nec^{op}}$ via Construction \ref{construction: necklace cat. associated to temp. obj.}. In this case, it is convenient to express the horn as a union of necklaces, rather than faces.

\begin{Prop}\label{proposition: horns and boundaries as unions of faces}
For all integers $0 < j < n$,
$$
(\Lambda^{n}_{j})_{\bullet}(0,n) = \bigcup_{\substack{i=1\\ i\neq j}}^{n-1}\delta_{i}(\Delta^{n-1})_{\bullet}(0,n)\cup \bigcup_{k=1}^{n-1}(\Delta^{k}\vee \Delta^{n-k})_{\bullet}(0,n)
$$
as a subfunctor of $\Delta^{n}_{\bullet}(0,n)$ in $\Set^{\nec^{op}}$.
\end{Prop}
\begin{proof}
For all $0 < k,i < n$ with $i\neq j$, we have inclusions $\Delta^{k}\vee \Delta^{n-k}\subseteq \Lambda^{n}_{j}$ and $\delta_{i}(\Delta^{n-1})\subseteq \Lambda^{n}_{j}$ in $\SSet$. It follows that
$$
\bigcup_{\substack{i=1\\ i\neq j}}^{n-1}\delta_{i}(\Delta^{n-1})_{\bullet}(0,n)\cup \bigcup_{k=1}^{n-1}(\Delta^{k}\vee \Delta^{n-k})_{\bullet}(0,n) \subseteq (\Lambda^{n}_{j})_{\bullet}(0,n)
$$

Conversely, let $f: T\rightarrow (\Lambda^{n}_{j})_{0,n}$ be a map in $\SSet_{*,*}$ with $(T,p)$ a necklace. Suppose first that $f$ is surjective on vertices. As the unique non-degenerate $n$-simplex of $\Delta^{n}$ is not contained in $\Lambda^{n}_{j}$, there must be some $k\in T$ such that $0 < f(k) < n$. Therefore, $f$ factors through $\Delta^{l}\vee \Delta^{n-l}$ with $l = f(k)$. Now suppose that $f$ is not surjective on vertices. Then $f$ must factor through $\delta_{i}(\Delta^{n-1})$ for some $i\in [n]\setminus \{j\}$. As a map in $\SSet_{*,*}$, $f$ always reaches the vertices $0$ and $n$ of $\Delta^{n}$ and thus $0 < i < n$.
\end{proof}

\begin{Ex}
The outer horns aren't as well-behaved in $\Set^{\nec^{op}}$ as the inner horns. For example, $\Lambda^{2}_{0}$ is the pushout $\Delta^{1}\amalg_{\{0\}}\Delta^{1}$ in $\SSet$, but $(\Lambda^{2}_{0})_{\bullet}(0,2)$ is isomorphic to just $\Delta^{1}_{\bullet}(0,1)$ as all maps $T\rightarrow (\Lambda^{2}_{0})_{0,2}$ in $\SSet_{*,*}$ must factor through the edge $0\rightarrow 2$ of $\Lambda^{2}_{0}$.
\end{Ex}

The following corollary expresses the advantage of working in the functor category $\mathcal{V}^{\nec^{op}}$. While not every simplex of a templicial object is represented by a templicial morphism (see Example \ref{example: free non-cell-cofibrant temp. obj.}), it is represented by a morphism in $\mathcal{V}^{nec^{op}}$.

\begin{Cor}\label{corollary: elementwise charac. of necklace morphisms}
Let $(X,S)$ be a templicial object with $a,b\in S$.
\begin{enumerate}[1.]
\item \label{item: necklace morphism from simplex} Let $(T,p)$ be a necklace. There is a bijective correspondence between morphisms\\ $\tilde{F}(T)_{\bullet}(0,p)\rightarrow X_{\bullet}(a,b)$ in $\mathcal{V}^{\nec^{op}}$ and elements $\sigma\in U(X_{T}(a,b))$.
\item \label{item: necklace morphism from horn} Let $0 < j < n$ be integers. There is a bijective correspondence between morphisms $\tilde{F}(\Lambda^{n}_{j})_{\bullet}(0,n)\rightarrow X_{\bullet}(a,b)$ in $\mathcal{V}^{\nec^{op}}$ and elements
$$
x_{k}\in U((X_{k}\otimes_{S} X_{n-k})(a,b))\quad \text{and}\quad y_{i}\in U(X_{n-1}(a,b))
$$
for all $0 < k,i < n$ with $i\neq j$, which satisfy:
\begin{itemize}
\item for all $0 < i < i' < n$ with $i\neq j\neq i'$,
$$
d_{i'-1}(y_{i}) = d_{i}(y_{i'}),
$$
\item for all $0 < k < l < n$,
$$
(\id_{X_{k}}\otimes \mu_{l-k,n-l})(x_{k}) = (\mu_{k,l-k}\otimes \id_{X_{n-l}})(x_{l})
$$
\item for all $0 < k < n-1$ and $0 < i < n$ with $i\neq j$,
$$
\mu_{k,n-k-1}(y_{i}) =
\begin{cases}
(d_{i}\otimes \id_{X_{n-k-1}})(x_{k+1}) & \text{if }i\leq k\\
(\id_{X_{k}}\otimes d_{i-k})(x_{k}) & \text{if }i > k
\end{cases}
$$
\end{itemize}
\end{enumerate}
\end{Cor}
\begin{proof}
A morphism $F(T_{\bullet}(0,p))\simeq \tilde{F}(T)_{\bullet}(0,p)\rightarrow X_{\bullet}(a,b)$ is equivalent to a map $T_{\bullet}(0,p)\rightarrow U(X_{\bullet}(a,b))$ in $\Set^{\nec^{op}}$, which corresponds to an element $\sigma\in U(X_{T}(a,b))$ by the Yoneda lemma. This shows \ref{item: necklace morphism from simplex}. Statement \ref{item: necklace morphism from horn} follows from Proposition \ref{proposition: horns and boundaries as unions of faces}.
\end{proof}

\begin{Def}\label{definition: necklace functor lifts inner horns}
Let $Y: \nec^{op}\rightarrow \mathcal{V}$ be a functor. We say $Y$ \emph{lifts inner horns} if for all $0 < j < n$ any lifting problem
\[\begin{tikzcd}
	{\tilde{F}(\Lambda^{n}_{j})_{\bullet}(0,n)} & Y \\
	{\tilde{F}(\Delta^{n})_{\bullet}(0,n)}
	\arrow[from=1-1, to=2-1]
	\arrow[from=1-1, to=1-2]
	\arrow[dashed, from=2-1, to=1-2]
\end{tikzcd}\]
has a solution in $\mathcal{V}^{\nec^{op}}$. We say $Y$ \emph{lifts inner horns uniquely} if every such lifting problem has a unique solution in $\mathcal{V}^{\nec^{op}}$.

We call a templicial object $(X,S)$ in $\mathcal{V}$ a \emph{quasi-category in $\mathcal{V}$} if the functor $X_{\bullet}(a,b)$ lifts inner horns for all $a,b\in S$. In this case, we will refer to the elements of $S$ as the \emph{objects} of $X$ and to elements of $U(X_{1}(a,b))$ as the \emph{morphisms} $a\rightarrow b$ in $X$.
\end{Def}

\begin{Rem}
Let $Y: \nec^{op}\rightarrow \mathcal{V}$ be a functor. Note that by the adjunction $F\dashv U$, $Y$ lifts inner horns in $\mathcal{V}^{\nec^{op}}$ if and only if the composite $UY: \nec^{op}\rightarrow \Set$ lifts inner horns in $\Set^{\nec^{op}}$.
\end{Rem}

As for ordinary quasi-categories, there is an elementwise characterisation of quasi-categories in $\mathcal{V}$, although it is bit more cumbersome to describe.

\begin{Prop}\label{proposition: fibrant temp. obj. charac.}
Let $(X,S)$ be a templicial object. The following statements are equivalent.
\begin{enumerate}[(1)]
\item $X$ is a quasi-category in $\mathcal{V}$.
\item Let $a,b\in S$ and $0 < j < n$. For all collections of elements $(x_{k})_{k=1}^{n-1}$, $(y_{i})^{n-1}_{i=1,i\neq j}$ satisfying the conditions of Corollary \ref{corollary: elementwise charac. of necklace morphisms}.\ref{item: necklace morphism from horn}, there is an element $z\in U(X_{n}(a,b))$ such that
$$
\mu_{k,n-k}(z) = x_{k}\quad \text{and}\quad d_{i}(z) = y_{i}
$$
for all $0 < k,i < n$ with $i\neq j$.
\end{enumerate}
\end{Prop}
\begin{proof}
This immediately follows from Corollary \ref{corollary: elementwise charac. of necklace morphisms}.
\end{proof}

\begin{Rem}
Note the similarities with the classical elementwise characterisation of a quasi-category. The elements $y_{i}$ with $0 < i < n$, $i\neq j$ represent all inner faces of the horn $\Lambda^{n}_{j}$. They still have to satisfy the same conditions as usual. However, the two outer faces of the horn are replaced by the elements $x_{k}$ with $0 < k < n$. The two new conditions of Corollary \ref{corollary: elementwise charac. of necklace morphisms}.\ref{item: necklace morphism from horn} merely express that these outer faces are glued to each other and to the inner faces in the appropriate way.

Indeed, in case $\mathcal{V} = \Set$ we recover the classical notion of a quasi-category.
\end{Rem}

\begin{Prop}\label{proposition: fibrant temp. sets are quasi-cats.}
A simplicial set is a quasi-category if and only if it is a quasi-category in $\Set$ (in the sense of Definition \ref{definition: necklace functor lifts inner horns}) .
\end{Prop}
\begin{proof}
Let $X$ be a simplicial set, considered as a templicial set with $X_{0}$ its set of vertices. Then the assignment $(x_{k})_{k=1}^{n-1}\mapsto (x^{1}_{n-1},x^{2}_{1})$ defines a bijection between the set of all collections of elements
$$
\left(x_{k} = (x^{1}_{k},x^{2}_{k})\in X_{k}\times X_{n-k}\right)_{k=1}^{n-1}
$$
satisfying $(x^{1}_{k},\mu_{l-k,n-l}(x^{2}_{k})) = (\mu_{k,l-k}(x^{1}_{k}),x^{2}_{l})$ for all $0 < k < l < n$, and the set of all pairs $(y_{n},y_{0})\in X_{n-1}\times X_{n-1}$ satisfying $d_{n-1}(y_{0}) = d_{0}(y_{n})$. It follows that condition $(2)$ of Proposition \ref{proposition: fibrant temp. obj. charac.} is equivalent to
\begin{itemize}
\item[$(2')$] Let $0 < j < n$. Consider elements $y_{i}\in X_{n-1}$ for all $0\leq i\leq n$ with $i\neq j$, which satisfy for all $0\leq i < i'\leq n$ with $i\neq j\neq i'$:
$$
d_{i'-1}(y_{i}) = d_{i}(y_{i'})
$$
Then there is an element $z\in X_{n}$ such that $d_{i}(z) = y_{i}$ for all $0\leq i\leq n$ with $i\neq j$.
\end{itemize}
But this precisely expresses that $X$ is a quasi-category.
\end{proof}

\subsection{Nerves and quasi-categories}\label{subsection: Nerves and quasi-categories}

We show that the earlier defined templicial versions of classical nerves give examples of quasi-categories in $\mathcal{V}$.

\begin{Lem}\label{lemma: necklace temp. adjunction counit fibration}
Let $\mathcal{C}$ be a necklace category with object $A$ and $B$. Consider the canonical morphism $\epsilon: \mathcal{C}^{temp}_{\bullet}(A,B)\rightarrow \mathcal{C}_{\bullet}(A,B)$ induced by the counit of the adjunction $(-)^{nec}\dashv (-)^{temp}$. Given integers $0 < j < n$, any lifting problem in $\mathcal{V}^{\nec^{op}}$:
\[\begin{tikzcd}
	{\tilde{F}(\Lambda^{n}_{j})_{\bullet}(0,n)} & {\mathcal{C}^{temp}_{\bullet}(A,B)} \\
	{\tilde{F}(\Delta^{n})_{\bullet}(0,n)} & {\mathcal{C}_{\bullet}(A,B)}
	\arrow[from=1-1, to=2-1]
	\arrow[from=1-1, to=1-2]
	\arrow[from=2-1, to=2-2]
	\arrow["\epsilon", from=1-2, to=2-2]
	\arrow[dashed, from=2-1, to=1-2]
\end{tikzcd}\]
has a unique solution.
\end{Lem}
\begin{proof}
The top horizontal morphism corresponds to some collections of elements $(x_{k})_{k=1}^{n-1}$ and $(y_{i})_{i=1,i\neq j}^{n-1}$ with $x_{k}\in U((\mathcal{C}^{temp}_{k}\otimes \mathcal{C}^{temp}_{n-k})(a,b))$ and $y_{i}\in U(\mathcal{C}^{temp}_{n-1}(a,b))$, satisfying the conditions of Corollary \ref{corollary: elementwise charac. of necklace morphisms}.\ref{item: necklace morphism from horn}. Moreover, the bottom horizontal morphism corresponds to an element $z'\in U(\mathcal{C}_{\{0 < n\}}(a,b))$ and the commutativity of the diagram comes down to the condition that $\mathcal{C}(\nu_{k,n-k})(z') = m_{\{0 < k\},\{0 < n-k\}}(p_{k}\otimes p_{n-k})(x_{k})$ and $\mathcal{C}(\delta_{i})(z') = p_{n-1}(y_{i})$ for all $0 < k,i < n$ with $i\neq j$.

Then by the limit diagram \eqref{diagram: temp. obj. associated to neck. obj. limit}, there exists a unique element $z\in U(\mathcal{C}^{temp}_{n}(a,b))$ such that $\mu_{k,n-k}(z) = x_{k}$ for all $0 < k < n$, and $p_{n}(z) = z'$. Again by \eqref{diagram: temp. obj. associated to neck. obj. limit}, we have that for all $0 < k,i < n$ with $i\neq j$:
\begin{align*}
&\mu_{k,n-1-k}(d_{i}(z)) =
\begin{cases}
(d_{i}\otimes \id_{\mathcal{C}^{temp}_{n-k-1}})(\mu_{k+1,n-k}(z)) & \text{if }i\leq k\\
(\id_{\mathcal{C}^{temp}_{k}}\otimes d_{i-k})(\mu_{k,n-k}(z)) & \text{if }i > k
\end{cases}
= \mu_{k,n-1-k}(y_{i})\\
&p_{n-1}(d_{i}(z)) = \mathcal{C}(\delta_{i})p_{n}(z) = \mathcal{C}(\delta_{i})(z') = p_{n-1}(y_{i})
\end{align*}
and thus $d_{i}(z) = y_{i}$. Hence, the element $z$ determines the unique solution to the lifting problem.
\end{proof}

\begin{Prop}\label{proposition: fibrant temp. obj. assoc. to necklace cat.}
Let $\mathcal{C}$ be a necklace category with object set $S$. Suppose that for all $A,B\in \Ob(\mathcal{C})$, $\mathcal{C}_{\bullet}(A,B)$ lifts inner horns. Then $\mathcal{C}^{temp}$ is a quasi-category in $\mathcal{V}$.
\end{Prop}
\begin{proof}
This is immediate from Lemma \ref{lemma: necklace temp. adjunction counit fibration}.
\end{proof}

\begin{Cor}
For any small $\mathcal{V}$-category $\mathcal{C}$, the templicial object $N_{\mathcal{V}}(\mathcal{C})$ is a quasi-category in $\mathcal{V}$.  
\end{Cor}
\begin{proof}
This follows from Propositions \ref{proposition: nerve is temp functor of constant functor} and \ref{proposition: fibrant temp. obj. assoc. to necklace cat.}.
\end{proof}

\begin{Cor}\label{corollary: hc-nerve preserves fibrant objects}
Let $\mathcal{C}$ be small simplicial $\mathcal{V}$-category. Assume that for all objects $A$ and $B$ of $\mathcal{C}$, the simplicial set $U(\mathcal{C}(A,B))$ is a Kan complex. Then the templicial object $N^{hc}_{\mathcal{V}}(\mathcal{C})$ is a quasi-category in $\mathcal{V}$.
\end{Cor}
\begin{proof}
By Proposition \ref{proposition: fibrant temp. obj. assoc. to necklace cat.}, it suffices to check that for all $A,B\in \Ob(\mathcal{C})$, the functor
$$
\mathfrak{n}(\mathcal{C}(A,B))_{\bullet} = [FN\mathcal{P}_{(-)},\mathcal{C}(A,B)]: \nec^{op}\rightarrow \mathcal{V}
$$
lifts inner horns in $\mathcal{V}^{\nec^{op}}$. By the adjunction $\mathfrak{s}\dashv \mathfrak{n}$, this is equivalent to showing that for all $0 < j < n$, every morphism $\mathfrak{s}(\tilde{F}(\Lambda^{n}_{j})_{\bullet}(0,n))\rightarrow \mathcal{C}(A,B)$ in $S\mathcal{V}$ extends to $\mathfrak{s}(\tilde{F}(\Delta^{n})_{\bullet}(0,n))$. Now by Proposition \ref{proposition: free and categorification commute} and the adjunction $F\dashv U$, this is further equivalent to the following lifting problem in $\SSet$:
\[\begin{tikzcd}
	{\mathfrak{C}[\Lambda^{n}_{j}](0,n)} & {U(\mathcal{C}(A,B))} \\
	{\mathfrak{C}[\Delta^{n}](0,n)}
	\arrow[from=1-1, to=1-2]
	\arrow[from=1-1, to=2-1]
	\arrow[dashed, from=2-1, to=1-2]
\end{tikzcd}\]
This has a solution by the fact that $U(\mathcal{C}(A,B))$ is a Kan complex, as was shown in \cite[Proposition 1.1.5.10]{lurie2009higher} and is given in more detail in \cite[Tag 00LH]{kerodon} (beware that in the latter, the notation $\mathrm{Path}$ is used instead of $\mathfrak{C}$).
\end{proof}

\begin{Cor}\label{corollary: fibrant temp. obj. has underlying quasi-cat.}
Let $X$ be a quasi-category in $\mathcal{V}$. Then $\tilde{U}(X)$ is a quasi-category.
\end{Cor}
\begin{proof}
This follows from Propositions \ref{proposition: free and nec functors commute}, \ref{proposition: fibrant temp. sets are quasi-cats.} and \ref{proposition: fibrant temp. obj. assoc. to necklace cat.}.
\end{proof}

The converse to Corollary \ref{corollary: fibrant temp. obj. has underlying quasi-cat.} does not hold in general.

\begin{Ex}
Consider the over category $\mathcal{V} = \Ab/\mathbb{Z}$ of abelian groups $A$ with a $\mathbb{Z}$-linear map $p: A\rightarrow \mathbb{Z}$. Then $\mathcal{V}$ is bicomplete and symmetric monoidal closed with monoidal unit given by $\id_{\mathbb{Z}}: \mathbb{Z}\rightarrow \mathbb{Z}$. The forgetful functor $U: \mathcal{V}\rightarrow \Set$ associates to every map $p: A\rightarrow \mathbb{Z}$ the set $\{a\in A\mid p(a) = 1\}$.

Now consider the simplicial set $\Delta^{2}\amalg_{\{0,2\}}\Lambda^{2}_{1}$:
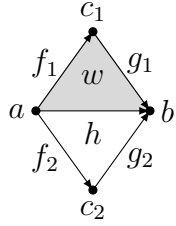
\begin{center}
\begin{tikzpicture}[scale=1.5]
\filldraw[fill=gray,opacity=0.3]
(0,0) -- (0.5,0.7) -- (1,0);
\filldraw
(0,0) circle (1pt) node[left]{$a$}
(0.5,0.7) circle (1pt) node[above]{$c_{1}$}
(0.5,-0.7) circle (1pt) node[below]{$c_{2}$}
(1,0) circle (1pt) node[right]{$b$}
(0.5,0.3) node{$w$};
\draw[-latex] (0,0) -- node[left,pos=0.6]{$f_{1}$} (0.5,0.7);
\draw[-latex] (0,0) -- node[left,pos=0.6]{$f_{2}$} (0.5,-0.7);
\draw[-latex] (0.5,0.7) -- node[right,pos=0.4]{$g_{1}$} (1,0);
\draw[-latex] (0.5,-0.7) -- node[right,pos=0.4]{$g_{2}$} (1,0);
\draw[-latex] (0,0) -- node[below,pos=0.5]{$h$} (1,0);
\end{tikzpicture}
\end{center}
Set $X = \tilde{F}(\Delta^{2}\amalg_{\{0,2\}}\Lambda^{2}_{1})\in S_{\otimes}\Ab$. We can promote $X$ to a templicial object in $\mathcal{V}$ by equipping it with $\mathbb{Z}$-linear maps $p: X_{n}(x,y)\rightarrow \mathbb{Z}$ defined as follows:
$$
p(f_{1}) = p(f_{2}) = p(g_{2}) = p(h) = 1,\quad p(g_{2}) = 2\quad \text{and}\quad p(w) = 1
$$
Then for example $U(X_{1}(a,c_{2})) = \{f_{2}\}$ but $U(X_{1}(c_{2},b)) = \emptyset$. Consider the functor $\tilde{U}: \ts\mathcal{V}\rightarrow \SSet$ as induced by $U$ above (not by $\Ab\rightarrow \Set$). Then it follows that $\tilde{U}(X)\simeq \Delta^{2}\amalg_{\{0\}}\Delta^{1}$, which is clearly a quasi-category.

However, $X$ is not a quasi-category in $\mathcal{V}$. To see this, consider the element
$$
\alpha = f_{2}\otimes g_{2} - f_{1}\otimes g_{1} \in U((X_{1}\otimes X_{1})(a,b))
$$
(note that indeed, $(p\otimes p)(\alpha) = p(f_{2})p(g_{2}) - p(f_{1})p(g_{1}) = 1$). But there exists no element $\xi\in U(X_{2}(a,b))$ such that $\mu_{1,1}(\xi) = \alpha$.
\end{Ex}

We end this subsection by characterising the essential image of the templicial nerve functor $N_{\mathcal{V}}: \mathcal{V}\Cat\rightarrow \ts\mathcal{V}$ in terms of horn fillings.

\begin{Prop}\label{proposition: nerve equiv.}
Let $(X,S)\in\ts\mathcal{V}$. Consider the following statements.
\begin{enumerate}[(1)]
\item $(X,S)$ is isomorphic to the templicial nerve of a small $\mathcal{V}$-category.
\item For all $a,b\in S$, $X_{\bullet}(a,b)$ lifts inner horns uniquely.
\end{enumerate}
Then $(1)$ implies $(2)$. Moreover, if the functor $U: \mathcal{V}\rightarrow \Set$ is conservative, then $(1)$ and $(2)$ are equivalent.
\end{Prop}
\begin{proof}
Let $\mathcal{C}$ be a small $\mathcal{V}$-category. We wish to show that $N_{\mathcal{V}}(\mathcal{C})_{\bullet}(A,B)$ lifts inner horns uniquely for all $A,B\in \Ob(\mathcal{C})$. Since $N_{\mathcal{V}}\simeq (-)^{temp}\circ \underline{(-)}$ (Proposition \ref{proposition: nerve is temp functor of constant functor}), it suffices by Lemma \ref{lemma: necklace temp. adjunction counit fibration} to note that the lifting problem
\[\begin{tikzcd}
	{\tilde{F}(\Lambda^{n}_{j})_{\bullet}(0,n)} & {\underline{\mathcal{C}}(A,B)} \\
	{\tilde{F}(\Delta^{n})_{\bullet}(0,n)}
	\arrow[from=1-1, to=1-2]
	\arrow[from=1-1, to=2-1]
	\arrow[dashed, from=2-1, to=1-2]
\end{tikzcd}\]
has a unique solution for all $0 < j < n$, which is trivial.

Assume that $(2)$ holds and that $U$ is conservative. By \eqref{equation: universal monoid}, it suffices to show that each comultiplication morphism $\mu_{k,n-k}$ with $0 < k < n$ is an isomorphism. Take $x_{k}\in U(X_{k}\otimes_{S} X_{n-k})$. By induction on $n$, we can define for any $0 < l < n$ with $l\neq k$:
$$
x_{l} =
\begin{cases}
(\id_{X_{l}}\otimes \mu^{-1}_{k-l,n-k})(\mu_{l,k-l}\otimes \id_{X_{n-k}})(x_{k}) & \text{if }l < k\\
(\mu^{-1}_{k,l-k}\otimes \id_{X_{n-l}})(\id_{X_{k}}\otimes \mu_{l-k,n-l})(x_{k}) & \text{if }l > k
\end{cases}
$$
Further set, for all $0 < i < n$ with $i\neq k$:
$$
y_{i} =
\begin{cases}
\mu^{-1}_{k-1,n-k}(d_{i}\otimes \id_{X_{n-k}})(x_{k}) & \text{if }i < k\\
\mu^{-1}_{k,n-k-1}(\id_{X_{k}}\otimes d_{i-k})(x_{k}) & \text{if }i > k
\end{cases}
$$
It follows that the elements $(x_{l})^{n-1}_{l=1}$ and $(y_{i})^{n-1}_{i = 1, i\neq k}$ satisfy the conditions of Corollary \ref{corollary: elementwise charac. of necklace morphisms}.\ref{item: necklace morphism from horn} and thus there is a unique element $z\in U(X_{n}(a,b))$ such that $\mu_{l,n-l}(z) = x_{l}$ and $d_{i}(z) = y_{i}$ for all $0 < l,i < n$ with $i\neq k$. In particular $\mu_{k,n-k}(z) = x_{k}$. For any other $z'\in U(X_{n}(a,b))$ with $\mu_{k,n-k}(z') = x_{k}$, it follows from the definitions of the $x_{l}$ and $y_{i}$ that also $\mu_{l,n-l}(z') = x_{l}$ and $d_{i}(z') = y_{i}$ for all $0 < l,i < n$ with $i\neq k$. Thus $z' = z$ and hence the map
$$
U(\mu_{k,n-k}): U(X_{n}(a,b))\rightarrow U((X_{k}\otimes X_{n-k})(a,b))
$$
is a bijection. As $U$ is conservative, $\mu_{k,n-k}: X_{n}\rightarrow X_{k}\otimes X_{n-k}$ is an isomorphism of $\mathcal{V}$-enriched quivers.
\end{proof}

\subsection{Simplification of the homotopy category}\label{subsection: Simplification of the homotopy category}

We now turn our attention to the homotopy category $h_{\mathcal{V}}X$ when $X$ is a quasi-category in $\mathcal{V}$. As is the case in the classical situation, this allows for a simpler description of $h_{\mathcal{V}}X$.

\begin{Con}\label{construction: left-pinched mapping space in low dimensions}
Let $(X,S)$ be a templicial object and $a,b\in S$. We define an object $\Hom^{L}_{X}(a,b)_{1}\in \mathcal{V}$ by the following pullback:
\[\begin{tikzcd}
	{\Hom^{L}_{X}(a,b)_{1}} & {X_{2}(a,b)} \\
	{X_{1}(a,b)} & {(X_{1}\otimes_{S} X_{1})(a,b)}
	\arrow["{\pi_{2}}"{pos=0.4}, from=1-1, to=1-2]
	\arrow["{\pi_{1}}"', from=1-1, to=2-1]
	\arrow["{-\otimes s^{X}_{0}}"', from=2-1, to=2-2]
	\arrow["\mu^{X}_{1,1}", from=1-2, to=2-2]
\end{tikzcd}\]
Further, we denote $d_{1} = \pi_{1}$, $d_{0} = d^{X}_{1}\pi_{2}$ and we let $s_{0}: X_{1}(a,b)\rightarrow \Hom^{L}_{X}(a,b)_{1}$ be the unique morphism such that $\pi_{1}s_{0} = \id_{X_{1}(a,b)}$ and $\pi_{2}s_{0} = s^{X}_{1}$. We obtain a reflexive pair:
\[\begin{tikzcd}
	{\Hom^{L}_{X}(a,b)_{1}} & {X_{1}(a,b)}
	\arrow["{d_{0}}", shift left=2, from=1-1, to=1-2]
	\arrow["{d_{1}}"', shift right=2, from=1-1, to=1-2]
	\arrow["{s_{0}}"{description}, from=1-2, to=1-1]
\end{tikzcd}\]
and we define an object $h'_{\mathcal{V}}X(a,b)$ as the coequaliser of this pair:
\begin{equation}\label{diagram: homotopy cat. coequalizer 2}
\begin{tikzcd}
	{\Hom^{L}_{X}(a,b)_{1}} & {X_{1}(a,b)} & {h'_{\mathcal{V}}X(a,b)}
	\arrow["{d_{0}}", shift left=1, from=1-1, to=1-2]
	\arrow["{d_{1}}"', shift right=1, from=1-1, to=1-2]
	\arrow["q", two heads, from=1-2, to=1-3]
\end{tikzcd}
\end{equation}
\end{Con}

\begin{Rem}
It is possible to extend Construction \ref{construction: left-pinched mapping space in low dimensions} to obtain a simplicial object $\Hom^{L}_{X}(a,b): \simp^{op}\rightarrow \mathcal{V}$ which generalises the \emph{left-pinched morphism space} of a simplicial set (as defined in \cite[Tag 01KX]{kerodon}). In particular, we have $\Hom^{L}_{X}(a,b)_{0} = X_{1}(a,b)$. The morphisms $d_{0},d_{1}: \Hom^{L}_{X}(a,b)_{1}\rightrightarrows X_{1}(a,b)$ and $s_{0}: X_{1}(a,b)\rightarrow \Hom^{L}_{X}(a,b)_{1}$ then constitute the lowest dimensional face and degeneracy morphisms of $\Hom^{L}_{X}(a,b)$. We will not go into them here however and leave their investigation to later research.
\end{Rem}

\begin{Rem}\label{remark: hom-sets underlying homotopy cat. of a fibrant temp. obj.}
Note that as $U$ preserves pullbacks, we find that $U(\Hom^{L}_{X}(a,b)_{1})$ is the set of all $2$-simplices $\sigma\in \tilde{U}(X)$ with $d_{0}(\sigma) = s_{0}(b)$ and $d_{1}d_{2}(\sigma) = a$. In other words, it describes homotopies between two edges $a\rightarrow b$ in $\tilde{U}(X)$.

Assuming that $\tilde{U}(X)$ is a quasi-category and that $U$ preserves reflexive coequalisers, it follows that we have an isomorphism:
$$
U(h'_{\mathcal{V}}X(a,b))\simeq h\tilde{U}(X)(a,b)
$$
and the canonical morphism $X_{1}(a,b)\twoheadrightarrow h'_{\mathcal{V}}X(a,b)$ precisely takes the homotopy class $[f]$ in $h\tilde{U}(X)$ of any $f\in U(X_{1}(a,b))$.
\end{Rem}

\begin{Lem}\label{lemma: comp. well defined in homotopy category of fibrant temp. obj.}
Assume that $U: \mathcal{V}\rightarrow \Set$ preserves reflexive coequalisers. Let $X$ be a quasi-category in $\mathcal{V}$ with objects $a$ and $b$. For any $w,w'\in U(X_{2}(a,b))$ we have
$$
(q\otimes q)\mu_{1,1}(w) = (q\otimes q)\mu_{1,1}(w')\quad \Rightarrow\quad q(d^{X}_{1}(w)) = q(d^{X}_{1}(w'))
$$
in $h'_{\mathcal{V}}X(a,b)$.
\end{Lem}
\begin{proof}
Let $Q$ denote the quiver given by $\Hom^{L}_{X}(a,b)_{1}$ for all objects $a$ and $b$ of $X$.
Let $\sigma\in U((Q\otimes Q)(a,b))$ and $w,w'\in U(X_{2}(a,b))$ be such that $\mu_{1,1}(w) = (d_{0}\otimes d_{0})(\sigma)$ and $\mu_{1,1}(w') = (d_{1}\otimes d_{1})(\sigma)$. Then:
\begin{itemize}
\item Consider the following elements $x_{1} = (d_{1}\otimes s^{X}_{0}d_{1})(\sigma)\in U((X_{1}\otimes X_{2})(a,b))$, $x_{2} = (\pi_{2}\otimes d_{1})(\sigma)\in U((X_{2}\otimes X_{1})(a,b))$ and $y_{2} = w\in U(X_{2}(a,b))$. These define a morphism $\tilde{F}(\Lambda^{3}_{1})_{\bullet}(0,3)\rightarrow X_{\bullet}(a,b)$ which extends to an element $z\in U(X_{3}(a,b))$. Setting $w'' = d^{X}_{1}(z)\in U(X_{2}(a,b))$, we have $d^{X}_{1}(w'') = d^{X}_{1}(w)$.
\item Likewise, consider the elements $x_{1} = (d_{0}\otimes \pi_{2})(\sigma)\in U((X_{1}\otimes X_{2})(a,b))$, $x_{2} = w''\otimes s^{X}_{0}(b)\in U((X_{2}\otimes X_{1})(a,b))$ and $y_{2} = w'$. These define a morphism $\tilde{F}(\Lambda^{3}_{1})_{\bullet}(0,3)\rightarrow X_{\bullet}(a,b)$ which extends to an element $z\in U(X_{3}(a,b))$. Then set $\tau = d^{X}_{1}(z)\in U(X_{2}(a,b))$.
\end{itemize}
It follows that $\mu_{1,1}(\tau) = d^{X}_{1}(w)\otimes s^{X}_{0}(b)$ and $d^{X}_{1}(\tau) = d^{X}_{1}(w')$. Hence, we have $qd^{X}_{1}(w) = qd^{X}_{1}(w')$.

As the diagram \eqref{diagram: homotopy cat. coequalizer 2} is a reflexive coequaliser, it is preserved by $-\otimes -$ in both variables simultaneously so that we again have a reflexive coequaliser
\[\begin{tikzcd}
	{(Q\otimes Q)(a,b)} & {(X_{1}\otimes X_{1})(a,b)} & {(h'_{\mathcal{V}}X\otimes h'_{\mathcal{V}}X)(a,b)}
	\arrow["{d_{0}\otimes d_{0}}", shift left=1, from=1-1, to=1-2]
	\arrow["{d_{1}\otimes d_{1}}"', shift right=1, from=1-1, to=1-2]
	\arrow["{q\otimes q}", two heads, from=1-2, to=1-3]
\end{tikzcd}\]
Now assume that $(q\otimes q)\mu_{1,1}(w) = (q\otimes q)\mu_{1,1}(w')$. As $U$ preserves reflexive coequalisers, there exist $\alpha_{0}, ..., \alpha_{n}\in U((X_{1}\otimes X_{1})(a,b))$ such that $\mu_{1,1}(w) = \alpha_{0}$, $\alpha_{n} = \mu_{1,1}(w')$ and for all $i\in\{1, ..., n\}$ there exists a $\sigma\in U((Q\otimes Q)(a,b))$ such that
\begin{align*}
&\alpha_{i-1} = (d_{0}\otimes d_{0})(\sigma)\text{ and }(d_{1}\otimes d_{1})(\sigma) = \alpha_{i}\\
\text{or}\qquad &\alpha_{i-1} = (d_{1}\otimes d_{1})(\sigma)\text{ and }(d_{0}\otimes d_{0})(\sigma) = \alpha_{i}
\end{align*}
For every $0 < i < n$, $\alpha_{i}$ defines a horn $\tilde{F}(\Lambda^{2}_{1})_{\bullet}(0,2)\rightarrow X_{\bullet}(a,b)$ which we can extend to an element $w_{i}\in U(X_{2}(a,b))$ so that $\mu_{1,1}(w_{i}) = \alpha_{i}$. Thus it follows by the previous that
$$
qd_{1}(w) = qd_{1}(w_{1}) = ... = qd_{1}(w_{n-1}) = qd_{1}(w')
$$
\end{proof}

\begin{Lem}\label{lemma: lift of coequalizer}
Assume that $U: \mathcal{V}\rightarrow \Set$ is faithful. Let $g: X\rightarrow Y$ and $f: X\rightarrow Z$ be morphisms in $\mathcal{V}$ such that $g$ is a regular epimorphism. Suppose that for all $x,y\in U(X)$, we have
$$
g(x) = g(y)\quad \Rightarrow\quad f(x) = f(y)
$$
Then there exists a unique morphism $h: Y\rightarrow Z$ such that $hg = f$.
\end{Lem}
\begin{proof}
Denote the kernel pair $X\times_{Y}X\rightrightarrows X$ of $g$ by $\pi_{1}$ and $\pi_{2}$. Since $g$ is the coequaliser of this pair, it suffices to show that $f\pi_{1} = f\pi_{2}$. As $U$ is faithful, this is equivalent to showing that for all $(x,y)\in U(X)\times_{U(Y)}U(X)$, we have $f(x) = f(y)$. But this is equivalent to the hypothesis on $f$ and $g$.
\end{proof}

\begin{Con}\label{construction: homotopy cat. of a fibrant temp. obj.}
Assume that $U: \mathcal{V}\rightarrow \Set$ is faithful and preserves and reflects reflexive coequalisers. Let $(X,S)$ be a quasi-category in $\mathcal{V}$. We construct a $\mathcal{V}$-enriched category $h'_{\mathcal{V}}X$ whose hom-objects are given by $h'_{\mathcal{V}}X(a,b)$ of Construction \ref{construction: left-pinched mapping space in low dimensions}. Let $h'_{\mathcal{V}}X$ denote the quiver given by $h'_{\mathcal{V}}X(a,b)$ for all $a,b\in S$, and let $q: X_{1}\rightarrow h'_{\mathcal{V}}X$ denote the canonical quiver morphism.

First define $u: I_{S}\xrightarrow{s_{0}} X_{1}\xrightarrow{q} h'_{\mathcal{V}}X$. Note that $U$ also reflects regular epimorphisms (as they are the coequaliser of their kernel pair). Thus as $X$ is a quasi-category in $\mathcal{V}$, the comultiplication $\mu_{1,1}: X_{2}\rightarrow X_{1}\otimes_{S} X_{1}$ is a regular epimorphism. Further, $q$ is a regular epimorphism by definition. Now $-\otimes -$ preserves reflexive coequalisers in each variable and thus also regular epimorphisms. It follows that $q^{\otimes 2}\circ \mu_{1,1}$ is a regular epimorphism as well. Using Lemmas \ref{lemma: comp. well defined in homotopy category of fibrant temp. obj.} and \ref{lemma: lift of coequalizer}, we have a unique quiver morphism $m: h'_{\mathcal{V}}X\otimes_{S} h'_{\mathcal{V}}X\rightarrow h'_{\mathcal{V}}X$ such that the following diagram commutes:
\[\begin{tikzcd}
	{X_{2}} & {X^{\otimes 2}_{1}} & {(h'_{\mathcal{V}}X)^{\otimes 2}} \\
	& {X_{1}} & {h'_{\mathcal{V}}X}
	\arrow["{\mu_{1,1}}", from=1-1, to=1-2]
	\arrow["{q^{\otimes 2}}", from=1-2, to=1-3]
	\arrow["{m}", from=1-3, to=2-3]
	\arrow["{d_{1}}"', from=1-1, to=2-2]
	\arrow["q"', from=2-2, to=2-3]
\end{tikzcd}\]
Given a $2$-simplex $(\alpha_{i,j})_{1\leq i < j\leq 2}$ (see Remark \ref{remark: simplices of underlying simp. set}) of $\tilde{U}(X)$ with vertices $a$, $b$ and $c$, we have $\mu_{1,1}(\alpha_{02}) = \alpha_{01}\otimes \alpha_{12}$ and thus $m(q(\alpha_{01})\otimes q(\alpha_{02})) = q(d_{1}(\alpha_{02}))$. Therefore, the induced map
$$
U(h'_{\mathcal{V}}X(a,b))\times U(h'_{\mathcal{V}}X(b,c))\rightarrow U(h'_{\mathcal{V}}X(a,b)\otimes h'_{\mathcal{V}}X(b,c)) \xrightarrow{U(m_{a,b,c})} U(h'_{\mathcal{V}}X(a,c))
$$
coincides with the composition law of $h\tilde{U}(X)$ under the isomorphisms supplied by Remark \ref{remark: hom-sets underlying homotopy cat. of a fibrant temp. obj.}. The element $u_{a} = q(s_{0}(a)): I\rightarrow h'_{\mathcal{V}}X(a,a)$ is clearly the identity at $a$ in $h\tilde{U}(X)$. It then follows from the faithfulness of $U$ that $m$ is associative and unital with respect to $u$. So we obtain a $\mathcal{V}$-category $h'_{\mathcal{V}}X$.

Note that by construction we have an isomorphism of categories
$$
\mathcal{U}(h'_{\mathcal{V}}X)\simeq h\tilde{U}(X)
$$
\end{Con}

\begin{Prop}\label{proposition: two homotopy cats. of fibrant temp. obj. coincide}
Assume that $U: \mathcal{V}\rightarrow \Set$ is faithful and preserves and reflects reflexive coequalisers. The assignment $X\mapsto h'_{\mathcal{V}}X$ of Construction \ref{construction: homotopy cat. of a fibrant temp. obj.} extends to a functor $h'_{\mathcal{V}}$ from the full subcategory of $\ts\mathcal{V}$ spanned by all quasi-categories in $\mathcal{V}$ to $\mathcal{V}\Cat$, which is left-adjoint to the templicial nerve functor $N_{\mathcal{V}}$.

In particular, there exists a canonical isomorphism of $\mathcal{V}$-enriched categories:
$$
h_{\mathcal{V}}X\simeq h'_{\mathcal{V}}X
$$
for every quasi-category $X$ in $\mathcal{V}$.
\end{Prop}
\begin{proof}
It follows from Construction \ref{construction: homotopy cat. of a fibrant temp. obj.} and Lemma \ref{lemma: nerve unique extension} that we have a unique templicial morphism $\eta_{X}: X\rightarrow N_{\mathcal{V}}(h'_{\mathcal{V}}X)$ such that $\eta_{X_{1}}: X_{1}\rightarrow h'_{\mathcal{V}}X$ is precisely $q$. We claim that $\eta_{X}$ is the unit of an adjunction $h'_{\mathcal{V}}\dashv N_{\mathcal{V}}$.

Now let $\mathcal{C}$ be an arbitrary small $\mathcal{V}$-category and $(\zeta,f): X\rightarrow N_{\mathcal{V}}(\mathcal{C})$ a templicial morphism. Then by Lemma \ref{lemma: nerve unique extension}, $\zeta: f_{!}X\rightarrow N_{\mathcal{V}}(\mathcal{C})$ corresponds to a quiver morphism $H: X_{1}\rightarrow f^{*}(\mathcal{C})$ such that the diagrams \eqref{diagram: nerve unique extension} commute. Letting $Q$ denote the quiver given by $\Hom^{L}_{X}(a,b)_{1}$ for all objects $a$ and $b$ of $X$, we have a commutative diagram
\[\begin{tikzcd}
	Q & {X_{2}} & {X_{1}} & {f^{*}(\mathcal{C})} \\
	{X_{1}} & {X^{\otimes 2}_{1}} & {f^{*}(\mathcal{C})^{\otimes 2}} & {f^{*}(\mathcal{C}^{\otimes 2})}
	\arrow["{f^{*}(m_{\mathcal{C}})}"', from=2-4, to=1-4]
	\arrow["H", from=1-3, to=1-4]
	\arrow["{d^{X}_{1}}"', from=1-2, to=1-3]
	\arrow["{\pi_{2}}"', from=1-1, to=1-2]
	\arrow["{d_{1} = \pi_{1}}"', from=1-1, to=2-1]
	\arrow["{-\otimes s^{X}_{0}}", from=2-1, to=2-2]
	\arrow["{\mu^{X}_{1,1}}", from=1-2, to=2-2]
	\arrow[from=2-3, to=2-4]
	\arrow["{H^{\otimes 2}}", from=2-2, to=2-3]
	\arrow["{H\otimes u}"', bend right = 10pt, from=2-1, to=2-3]
	\arrow["{d_{0}}", bend left = 10pt, from=1-1, to=1-3]
\end{tikzcd}\]
It follows that $Hd_{0} = Hd_{1}: Q\rightarrow f^{*}(\mathcal{C})$ and thus there exists a unique quiver morphism $H': h'_{\mathcal{V}}X\rightarrow f^{*}(\mathcal{C})$ such that $H'q = H$. By construction, $H'$ defines a $\mathcal{V}$-functor $h'_{\mathcal{V}}X\rightarrow \mathcal{C}$ which is clearly unique such that $N_{\mathcal{V}}(H)\circ \eta_{X} = (\zeta,f)$.
\end{proof}

\begin{Cor}\label{corollary: forget commute with homotopy for fibrant temp. obj.}
Assume that $U: \mathcal{V}\rightarrow \Set$ is faithful and preserves and reflects reflexive coequalisers. Let $X$ be a quasi-category in $\mathcal{V}$. The canonical functor
$$
h\tilde{U}(X)\rightarrow \mathcal{U}(h_{\mathcal{V}}X)
$$
is an isomorphism of categories.
\end{Cor}
\begin{proof}
This is now an immediate consequence of Proposition \ref{proposition: two homotopy cats. of fibrant temp. obj. coincide} and the fact that $\mathcal{U}(h'_{\mathcal{V}}X)\simeq h\tilde{U}(X)$.
\end{proof}

\printbibliography

\end{document}